\documentclass[notitlepage,leqno,10pt]{article}

\usepackage{latexsym}
\usepackage{amsmath}
\usepackage{amsfonts}
\usepackage{amssymb}
\usepackage{amsthm}
\usepackage{color}
\renewcommand{\a }{\alpha }
\renewcommand{\b }{\beta }
\renewcommand{\d}{\delta }
\newcommand{\D }{\Delta }

\newcommand{\e }{\varepsilon }

\renewcommand{\l }{\lambda }

\newcommand{\Ric} {{\rm Ric}}

\newcommand{\n }{\nabla }
\newcommand{\var }{\varphi }

\newcommand{\s }{\sigma }
\renewcommand{\S }{\Sigma}

\renewcommand{\o }{\omega }

\newcommand{\pa }{\partial}

\newcommand{\ov}{\overline}

\newcommand{\be}{\begin{equation}}
\newcommand{\ee}{\end{equation}}

\newcommand{\R}{\mathbb{R}}

\newcommand{\T}{\mathbb{T}}

\newcommand{\Z}{\mathbb{Z}}

\renewcommand{\P}{\mathbb{P}}
\newcommand{\DD}{\mathbb{D}}

\newcommand{\N}{\mathbb{N}}

\newcommand{\la}{\langle}
\newcommand{\ra}{\rangle}
\newcommand{\Riem}{\mathrm{Riem}}
\newcommand{\Sc}{\mathrm{Sc}}
\newcommand{\rd}{\mathrm{d}}

\newcommand{\bfex}{\mathbf{e}_{x}}
\newcommand{\bfey}{\mathbf{e}_{y}}
\newcommand{\bfez}{\mathbf{e}_{z}}

\author{Norihisa Ikoma 
\\ \vspace{-0.1cm}
\footnotesize Faculty of Mathematics and Physics
\\ \vspace{-0.1cm}
       \footnotesize       Institute of Science and Engineering
       \\       \vspace{-0.1cm}
 \footnotesize              Kanazawa University
         \\     \vspace{-0.1cm}
   \footnotesize            Kakuma, Kanazawa
              \\ \vspace{-0.1cm}
    \footnotesize           Ishikawa 9201192, JAPAN
\\
\footnotesize \texttt{ikoma@se.kanazawa-u.ac.jp}
\and  Andrea Malchiodi 
\\ \vspace{-0.1cm}
\footnotesize Scuola Normale Superiore 
\\ \vspace{-0.1cm}
\footnotesize Piazza dei Cavalieri, 7 
\\ \vspace{-0.1cm}
\footnotesize  56126 Pisa, ITALY
\\
\footnotesize \texttt{andrea.malchiodi@sns.it} 
\and  Andrea Mondino 
\\ \vspace{-0.1cm}
\footnotesize Mathematics Institute
\\ \vspace{-0.1cm}
\footnotesize Zeeman Building 
\\ \vspace{-0.1cm}
\footnotesize University of Warwick
\\ \vspace{-0.1cm}
\footnotesize Coventry CV4 7AL, UK
 \\ \footnotesize \texttt{A.Mondino@warwick.ac.uk}}

\title{Embedded area-constrained Willmore tori  of small area in Riemannian three-manifolds I: Minimization}

\begin{document}

\hyphenation{un-coun-ta-bly}

\newtheorem{lem}{Lemma}[section]
\newtheorem{pro}[lem]{Proposition}
\newtheorem{thm}[lem]{Theorem}
\newtheorem{rem}[lem]{Remark}
\newtheorem{cor}[lem]{Corollary}
\newtheorem{df}[lem]{Definition}
\newtheorem{ex}[lem]{Examples}

\maketitle


\

\begin{abstract}
We construct embedded Willmore tori with small area constraint in Riemannian three-manifolds  under some curvature condition used to prevent M\"obius degeneration.  The construction relies on a Lyapunov-Schmidt reduction; to this aim we establish  new geometric expansions of  exponentiated small symmetric Clifford tori and analyze the  sharp asymptotic behavior of degenerating tori under the action of the M\"obius group.    In this first work we prove two existence results by minimizing or maximizing a suitable reduced functional, in particular we obtain embedded area-constrained Willmore tori (or, equivalently, toroidal critical points of the Hawking mass under area-constraint) in compact 3-manifolds with constant scalar curvature and in the double Schwarzschild space. In a forthcoming  paper new existence theorems will be achieved via Morse theory. 
\end{abstract}

\begin{center}

\bigskip\bigskip


\bigskip

\centerline{\bf AMS subject classification: } 
49Q10, 53C21, 53C42, 35J60, 83C99.
\end{center}

\section{Introduction}

This is the first of a series of two papers where the construction of embedded area-constrained  Willmore tori in Riemannian three-manifolds is performed under curvature conditions: here via minimization/maximization,  while in the second paper \cite{IMM2} via Morse theory. 
\\

Let us start by introducing  the Willmore functional. Given an immersion $i:\Sigma \hookrightarrow (M,g)$ of a closed (compact, without boundary) two-dimensional surface $\Sigma$ into a Riemannian $3$-manifold $(M,g)$, the \emph{Willmore functional} is defined by 
\be\label{eq:defW}
W(i):=\int_{\Sigma} H^2 \, d\sigma
\ee
where $d\sigma$ is the area form induced by the immersion and  $H$ is the mean curvature  (we adopt the convention that $H$ is the sum of the principal curvatures or, in other words, $H$ is the trace of the second fundamental form $A_{ij}$  with respect to the induced metric $\bar{g}_{ij}$: $H:=\bar{g}^{ij} A_{ij}$).

An immersion $i$ is called \emph{Willmore surface} (or Willmore immersion) if it is a critical point of the Willmore functional with respect to normal perturbations or, equivalently, if it satisfies the associated Euler-Lagrange equation
\be\label{eq:WillmoreEq}
\Delta_{\bar{g}} H + H |\mathring{A}|^2 + H \Ric(n, n)=0,
\ee
where $\Delta_{\bar{g}}$ is the Laplace-Beltrami operator corresponding to the induced metric $\bar{g}$, $(\mathring{A})_{ij}:=A_{ij}-\frac{1}{2}H\bar{g}_{ij}$ is the trace-free second fundamental form, $n$ is a normal unit vector to $i$, and $\Ric$ is the Ricci tensor of the ambient manifold $(M,g)$. Notice that \eqref{eq:WillmoreEq} is a  fourth-order nonlinear elliptic PDE in the immersion map $i$.
\\ 

The Willmore functional was first introduced for immersions into the Euclidean space in 
20's and 30's by Blaschke and Thomsen, who were looking for  a conformal invariant theory which included  minimal surfaces. Since  minimal surfaces trivially satisfy  \eqref{eq:WillmoreEq} and - as they proved - $W$ is invariant under M\"obius transformations in the Euclidean space, they detected the  class of Willmore surfaces  as the  natural  conformally invariant generalization of minimal surfaces and called them {\em conformal minimal surfaces}. The topic was then rediscovered in the 60's  by Willmore \cite{Will}, who proved that the global minimum among all closed  immersed surfaces was strictly achieved by round spheres and conjectured that the genus one minimizer was the Clifford torus (and its images via M\"obius transformations).  This long standing conjecture has been recently proved by Marques-Neves \cite{MN} by min-max techniques (partial results towards the Willmore conjecture were previously obtained by Li-Yau \cite{LY}, Montiel-Ros \cite{MonRos},  Ros \cite{Ros}, Topping \cite{Top} and others, and that a crucial role in the proof of the conjecture is played by a result of Urbano \cite{Urb}). Let us also mention the fundamental works on the Willmore functional by Simon \cite{SiL} who proved the  existence of a smooth genus one minimizer of  $W$ in $\R^m$ (later generalized to higher genus by Bauer-Kuwert \cite{BK} who were inspired by a paper of Kusner \cite{Kus}; an alternative proof of these results has been recently given by Rivi\`ere \cite{Riv2} using the so called {\em parametric approach}), by Kuwert-Sch\"atzle \cite{KS} and by Rivi\`ere \cite{Riv1}.  

Probably because of the richness of the symmetries preserving the class of Willmore surfaces and because of
the simplicity and the universality of its definition, the Willmore functional shows up in very different fields of sciences and technology. It appears for instance in biology in the study of lipid bilayer cell membranes under the name {\em Helfrich energy}, in general relativity as being the main term in the so called {\em Hawking
Mass}, in string theory in high energy physics concerning the definition of  Polyakov's extrinsic action, in
elasticity theory this as a free energy of the non-linear plate Birkhoff theory, in optics, lens design, etc.
\\

Let us stress that all the aforementioned results about Willmore surfaces concern immersions into  Euclidean space (or, equivalently by conformal invariance, for immersions into a round sphere); the literature about Willmore immersions into curved Riemannian manifolds, which has interest in
applications as it might model non-homogeneous environments, is much more recent: the first existence result was established in \cite{Mon1} where the third author constructed  embedded Willmore spheres in a perturbative setting (see also  \cite{Mon2}, and the more recent work \cite{CM} in collaboration with Carlotto; related perturbative results under area constraint have been obtained by Lamm-Metzger-Schulze \cite{LMS}, by Lamm-Metzger \cite{LM1}-\cite{LM2}, and by the third author in collaboration with Laurain \cite{LM}).

The global problem, i.e. the existence of smooth immersed spheres minimizing quadratic curvature functionals in compact  Riemannian three-manifolds,  was then studied by the third author in collaboration with Kuwert and Schygulla in \cite{KMS} (see also \cite{MonSch} for the non compact case).  In  collaboration with Rivi\`ere \cite{MR1}-\cite{MR2}, the third author  developed the necessary tools for the calculus of variations of the Willmore functional in Riemannian manifolds and proved  existence of area-constrained Willmore spheres in homotopy classes  (as well as the existence of Willmore spheres under various assumptions and  constraints).
\\

Since all the above existence results  in Riemannian manifolds concern spherical surfaces, a natural question is about the existence of higher genus Willmore surfaces in general curved spaces; in particular we will focus here on the genus one case. 

Let us mention that if the ambient space has some special symmetry then the Willmore equation  \eqref{eq:WillmoreEq}  simplifies and one can construct explicit examples (see for instance the paper of Wang \cite{Wang} in case the ambient manifold is a product and the work of Barros-Ferr\'andez-Lucas-Merono  \cite{BFLM} for the  case of warped product metrics); we also mention the work of Chen-Li \cite{ChenLi} where the existence of stratified weak branched immersions of arbitrary genus minimizing quadratic curvature functionals under various constraints is investigated. 
\\

The goal of the present (and the subsequent \cite{IMM2}) work is to construct smooth embedded Willmore tori with small area 
constraint in Riemannian three-manifolds, under some curvature condition but without any symmetry assumption. 
Let us recall that area-constrained Willmore surfaces satisfy the equation
\[
\Delta_{\bar{g}} H + H |\mathring{A}|^2 + H \Ric(n, n)=\lambda H,
\]
for some $\lambda\in \R$ playing the role of Lagrange multiplier. These immersions are naturally linked to the Hawking mass
\[
m_H(i):=\frac{\sqrt{Area(i)}}{64 \pi^{3/2}} \left(16\pi-W(i)\right),
\]
since, clearly,  the critical points of the Hawking mass under area constraint are exactly the  area-constrained Willmore immersions (see \cite{Bray,ChYau,LMS} and the references therein for more material about the Hawking mass).
\\

\noindent The main result of this paper is the following (recall that the area of the 
Clifford torus in the Euclidean space is $4 \sqrt{2} \pi^2$). 

\begin{thm}\label{t:1}
Let $(M,g)$ be a compact $3$-dimensional Riemannian manifold. 
Denote by $\Ric$ and ${\rm Sc}$ the Ricci and the scalar curvature of $(M,g)$ 
respectively, and suppose either 
	\begin{equation} \label{eq:Assump1}
		3 \sup_{P \in M} \left( {\rm Sc}_{P} - \inf_{| \nu |_{g}=1} 
		{\rm Ric}_{P} ( \nu, \nu) \right) 
		> 2 \sup_{P \in M} {\rm Sc}_{P},
	\end{equation}
or else 
       \begin{equation} \label{eq:Assump2}
          	3 \inf_{P \in M} \left( {\rm Sc}_{P} - \sup_{| \nu |_{g}=1} 
		{\rm Ric}_{P} ( \nu, \nu) \right) 
		< 2 \inf_{P \in M} {\rm Sc}_{P}. 
	\end{equation}
Then there exists $\e_0>0$ such that for every  $\e \in (0,\e_0)$ there exists a smooth embedded Willmore torus in $(M,g)$ with constrained 
area equal to $4 \sqrt{2} \pi^2 \e^2$.

More precisely, these surfaces are obtained as normal graphs over exponentiated (M\"obius transformations of) Clifford tori and the  corresponding graph functions, once dilated by a factor $1/\e$, converge to $0$ in $C^{4,\a}$-norm as $\e\to 0$ with decay rate $O(\e^2)$. 
\end{thm}

\begin{rem}\label{rem:2tori}
If both the conditions \eqref{eq:Assump1} and \eqref{eq:Assump2} are satisfied then  there exist at least \emph{two} smooth embedded Willmore tori in $(M,g)$ with constrained 
area equal to $4 \sqrt{2} \pi^2 \e^2$ (see Remark \ref{Rem:Pf2Tori} for the proof).
\end{rem}

\
The assumptions of Theorem \ref{t:1} are rather mild and allow to include a large class of examples, as explained in the next corollaries and remarks. See also Remark \ref{r:sectional} for an interpretation of the quantities 
in \eqref{eq:Assump1} and \eqref{eq:Assump2} in terms of sectional curvatures. 

\begin{cor}\label{cor:Rconst}
Let $(M,g)$ be a compact $3$-dimensional manifold with constant scalar curvature. 
Then there exists $\e_0>0$ such that for every  $\e \in (0,\e_0)$ there exist at least two smooth embedded Willmore tori in $(M,g)$ with constrained 
area equal to $4 \sqrt{2} \pi^2 \e^2$.
%
\end{cor}
Corollary \ref{cor:Rconst} easily follows by Remark  \ref{rem:2tori} 
once we observe that if ${\rm Sc}\equiv S\in \R$ and 
if $(M,g)$ does not have constant sectional curvature, 
then Schur's lemma implies the existence of $P \in M$ 
at which $(M,g)$ is not isotropic. 
With the aid of the expressions in Remark \ref{r:sectional},  both 
the  assumptions \eqref{eq:Assump1} and \eqref{eq:Assump2} are satisfied. 
On the other hand, if $(M,g)$ has constant sectional curvature $\bar{K} \in \R$ 
then it is conformally equivalent to (a quotient) of the Euclidean space 
(indeed either it is a quotient of the three-sphere ${\mathbb S}^3$ 
or of the Hyperbolic three-space ${\mathbb H}^3$ 
both conformally equivalent to the Euclidean three-space). 
Now it is well known that the functional 
$$W_{cnf}(i):=\int_\Sigma [H^2+ 4 \bar{K}] d\sigma= W(i)+4 \bar{K} \,Area(i)$$
is conformally invariant (see for instance \cite{WEINER}). Of course the area-constrained critical points of $W$ are exactly the area-constrained critical points of $W_{cnf}$, but then by conformal invariance it follows that Euclidean Clifford tori correspond to area-constrained Willmore tori in the space form.

\begin{rem}
The class of compact $3$-manifolds with constant scalar curvature include many remarkable examples of ambient spaces  which play an important role in contemporary surface theory. Trivial cases are compact quotients of space forms (notice that the same existence result applies to the standard non compact space forms as explained below), but more generally any homogeneous three manifold has of course constant scalar curvature. Examples of compact homogeneous spaces are $\mathbb{S}^2 \times \mathbb{S}^1$, Berger spheres and any compact quotient of a three-dimensional Lie Group. The study of special surfaces (minimal, constant mean curvature, totally umbilic) in homogeneous spaces is a very active area of research: for instance let us mention the recent monograph of Meeks-Perez \cite{MePe},  the recent paper on CMC spheres in homogeneous three-spheres by Meeks-Mira-Perez-Ros \cite{MeMiPeRo} and the classification of totally umbilical surfaces in homogeneous three-spaces by Manzano-Souam \cite{MaSo}.

Let us stress that most of the results in this setting  are for genus 0 surfaces and for second order problems, so the originality of our result lies in  both exploring higher genus surfaces and higher order problems (recall that the Willmore equation is of fourth order, while minimal,CMC, and totally umbilical surface equations are of second order). 
\end{rem}

Before passing to describe the organization of the paper let us mention that our techniques can be adapted to handle also some non compact framework. Indeed, as (remarkable) example, we prove the existence of toroidal area-constrained critical points of the Hawking mass in the Schwarzschild space.

Recall that the Schwarzschild metric of mass $m>0$ is given by $(\R^3 \setminus \{0\}, g_{Sch})$ where $g_{ij}(x) = (1 + \frac{m} {2r})^{4} \delta_{ij}$, where of course $r=\sqrt{(x^1)^2+(x^2)^2+(x^3)^2}$. The metric is spherically symmetric,  conformally equivalent to the Euclidean one and  asymptotically flat, it has
zero scalar curvature, and the  sphere at $\{r = m/2\}$  is totally geodesic. In fact,
the Schwarzschild metric is symmetric under the mapping $r \mapsto \frac{m^2}{4r}$ and therefore it has two 
asymptotically flat ends.

In the next remark we recall what is known about minimal and CMC surfaces in Schwarzschild metric (we thank Alessandro Carlotto for a discussion about this point).

\begin{rem}[Minimal and CMC surfaces in Schwarzschild]
\begin{itemize}
\item In the Schwarzschild space there are no non-spherical closed minimal surfaces: indeed it is known that if a closed minimal surface $\Sigma$ in the Schwarzschild space is non-spherical then it must be totally geodesic and in particular totally umbilic, but the fact of being totally umbilic is conformally invariant and therefore $\Sigma$ is totally umbilic also in the Euclidean space. But then $\Sigma$ is a Euclidean round sphere which contradicts that $\Sigma$ was not spherical. More precisely, by a  maximum principle argument  using comparison with CMC slices, it is possible to show that  the only embedded closed minimal hypersurface in $(\R^3 \setminus \{0\}, g_{Sch})$ is the horizon $\{r = m/2\}$, which in fact   is totally geodesic.
\item Regarding  CMC surfaces in  $(\R^3 \setminus \{0\}, g_{Sch})$, it was proved by Brendle \cite{Brendle} that the only embedded closed CMC surfaces in the outer Schwarzschild $(\R^3 \setminus B_{m/2}(0), g_{Sch})$ are the spherical slices $\{r=const\}$ (let us mention that the results of Brendle include a larger class of warped products metrics). The embeddedness assumption is crucial for this classification result,  in view of possible phenomena analogous to the Wente tori (which are immersed and CMC) in $\R^3$.
It is also essential that the closed surfaces do not intersect the horizon $\{r = m/2\}$. Indeed, solving the isoperimetric problem in  $(\R^3 \setminus \{0\}, g_{Sch})$ for small volumes, it is expected (by perturbative arguments \'a la  Pacard-Xu \cite{PX}) that the isoperimetric surfaces are spherical surfaces intersecting  $\{r = m/2\}$ (regarding the  isoperimetric problem in Schwarzschild  see also the Ph.D. thesis of  Bray \cite{Bray} and the recent paper by Brendle-Eichmair \cite{BrEi}).
\\

\end{itemize}

\noindent
Summarizing, it is known that in  $(\R^3 \setminus \{0\}, g_{Sch})$ there are \emph{no non-spherical embedded minimal} surfaces and it is expected there are \emph{no  non-spherical embedded CMC} surfaces. 
\end{rem}

In sharp contrast to the aforementioned situation, our next theorem asserts the existence of  \emph{embedded tori} which are critical points of the Hawking mass under area constraint. 

\begin{thm}\label{t:2}
Let $(\R^3 \setminus \{0\}, g_{Sch})$, with $g_{ij}(x) = (1 + \frac{m} {2r})^{4} \delta_{ij}$, be the Schwarzschild metric of mass $m>0$.
Then there exists $\e_0>0$ such that for every  $\e \in (0,\e_0)$ there exist smooth embedded tori (infinitely-many, by 
symmetry) in $(\R^3 \setminus \{0\},g_{Sch})$ which are critical points of the Hawking mass with constrained 
area equal to $4 \sqrt{2} \pi^2 \e^2$.
%
\end{thm}

\begin{rem}\label{rem:ALE} 
With analogous arguments to the proof of Theorem \ref{t:2} one can actually prove existence of smooth embedded tori which are critical points of the Hawking mass with small constrained 
area in \emph{asymptotically locally Euclidean  (ALE) scalar flat} 3-manifolds (for more details see Remark \ref{rem:proofRem}). More precisely, the following two conditions are sufficient for our arguments.

1) $(M,g)$ is a complete non compact 3-manifold whose scalar curvature vanishes identically: $\Sc\equiv0$.

2) Fixed some base point  $x_0\in M$, there exists $r>0$ with the following property: for every $\epsilon>0$ there exists $R_{\epsilon}>0$ such that for any $x\in M\setminus B_{R_\epsilon}^M(x_0)$ there exists a diffeomorphism  $\Psi: B_r^{\R^3}(0)\to B_r^M(x)$ satisfying $\|\delta_{ij}-(\Psi^* g)_{ij}\|_{C^2(B^{\R^3}_r(0))} \leq \epsilon$. 
\\Notice that  condition 1) is equivalent to the constrained Einstein equations in the vacuum case, and 2) is a mild uniform control of the local geometry of $M$ together with a mild asymptotic condition. 
\end{rem}

Now we pass to outline the structure of the paper and the 
main ideas of the construction: a Lyapunov-Schmidt reduction. Roughly, such technique can be summarized as follows (for a systematic presentation with examples see the monograph of Ambrosetti and the second author \cite{am2}; the original technique is due to Lyapunov-Schmidt but the variational implementation  we are going to describe is inspired by the work of Ambrosetti-Badiale \cite{ab1}-\cite{ab2}): one is interested in finding critical points of a real valued functional $I_\e$ from infinite dimensional space $X$,  knowing that for $\e>0$ small enough there  exists a finite dimensional manifold $Z_\e\subset X$ made of  {\em almost critical points of $I_\e$} (in the sense that the differential $I'_\e$ is small enough on $Z_\e$ in a suitable sense). 

If one also knows that the second differential $I''_\e$ is non degenerate on $Z_\e$, or more precisely $I''_\e$ restricted to the orthogonal complement of the tangent space to $Z_\e$ is non degenerate (here {\em orthogonal complement} must be  considered just formally as  a motivation, but this is rigorous  if $X$ is a Hilbert space), then one can solve an auxiliary equation (given by the projection of the equation $I'_\e=0$ onto the orthogonal complement of $TZ_\e$) and reduce the problem of finding critical  points of the functional $I_\e:X\to \R$ to finding critical points of a suitable  functional $\Phi_\e: Z_\e \to \R$; of course the advantage being the reduction of the problem to studying  a function of \emph{finitely many} variables. 
\\

In our case the functional $I_\e$  is of course the Willmore energy $W$ defined in \eqref{eq:defW} and  $X$ is the space of smooth immersions (with  area constraint $\e$) from the Clifford torus into the Riemannian manifold  $(M,g)$.  Observe that, by the conformal invariance of the Willmore functional in $\R^3$,   the family of (images under M\"obius transformations of)  Clifford tori in $\R^3$ form a non compact critical manifold for the Euclidean Willmore functional; moreover from the classical paper of Weiner \cite{WEINER} it is known that the second variation of  $W$ is non degenerate on this critical manifold  and by the recent gap-theorem proved by Nguyen and the third author \cite{MonNgu} this critical manifold is isolated in energy from the next Willmore torus. In Section 2 we describe in some detail the M\"obius maps which preserve the area of the Clifford torus, as 
well as Weiner's result.

Since at small scales a Riemannian metric approaches the Euclidean one, it is natural to expect that the images  of small Clifford tori via   exponential map form a manifold of {\em almost critical points} of $W$ in the above sense. 
This will be proved, with quantitative estimates, in Section \ref{ss:appsol}. After that, the finite-dimensional 
reduction of the problem will be carried out: in Proposition \ref{p:lyap} for every (exponentiated) torus we will 
construct a  perturbation which will solve our problem up to some further Lagrange multipliers given by the 
Jacobi fields of translations, rotations and M\"obius inversions. The abstract construction in Section \ref{Sec:FDR} is rather standard, nevertheless the main obstacle here is due to the fact that the action of the M\"obius group on the Clifford torus is non-compact.

To overcome this issue we incorporate in the abstract construction the variational structure of the problem. 
We will compare the Willmore energy of (exponentiated) symmetric VS degenerating tori. These 
expansions are worked-out in Section \ref{Sec:Expansions} (see in particular  Proposition \ref{p:exp-W-to}  and Proposition \ref{p:expdegtorus} respectively), where the effect of curvature is taken into account. Degenerate small tori 
look like geodesic spheres with small handles, so for these we check that the scalar curvature of $M$ 
plays the main role in the expansion (compare to the results in \cite{LM1} and \cite{Mon1}). For symmetric 
tori instead we observe an effect due to a combination of the scalar curvature and the sectional curvature 
of the plane of symmetry of the torus (which can be expressed in terms of the Ricci tensor in the 
axial direction of the torus), see Remark \ref{r:sectional}. 
These expansions are probably the  main contribution of the present work; we believe that they might 
play a role in further developments of the topic, especially in ruling-out possible degeneracy 
phenomena under global (non-perturbative) variational approaches to the problem, as it has already happened 
for the case of Willmore spheres. 

In Section \ref{Sec:ProofT1} we will use the assumptions in Theorem \ref{t:1} for ruling out 
M\"obius degeneration by direct energy comparison, which would be more costly. We will 
consider a sequence of compact sets invading the family of M\"obius inverted tori
 and we will show that the Willmore energy in the interior is strictly lower (or higher) than the energy on their boundary. 
This will prove the existence of a critical point despite the non-compactness  of the problem. 
Finally, some examples to the applicability of Theorem \ref{t:1} are provided.

\subsection*{Acknowledgements}
The first author was supported by JSPS Research Fellowships 24-2259. 
The second author is supported by the project {\em Geometric Variational Problems} from Scuola Normale Superiore and 
by MIUR Bando PRIN 2015 2015KB9WPT$_{001}$.  He is also member of GNAMPA as part of INdAM. 
The third author acknowledges  the support of  the ETH fellowship. 
This work was done while the first author visited the University of Warwick and SISSA: 
he would like to express his gratitude for the hospitality.


\section{Preliminaries}


		In this section, we discuss the M\"obius inversions 
preserving the area of the Clifford Torus $\T$. 
Moreover, we collect some properties of the Willmore functional. 
Hereafter, we use the notations $\la \cdot, \cdot \ra$ and 
$g_0(\cdot,\cdot)$ for the Euclidean metric.


\subsection{Inverted tori with fixed area}\label{ss:invfixarea}


This subsection is devoted to the analysis of the M\"obius inversions 
of $\T$ which preserve the area of the torus. 
Through this paper, we parametrize the Clifford torus $\T$ by 
	\[
		\T := \left\{ X(\varphi,\theta) \; : \;
		\varphi, \theta \in [0,2\pi] \right\} 
	\]
where 
	\begin{equation}\label{eq:def-X}
		X(\varphi,\theta) := 
		\left( (\sqrt{2}+\cos \varphi) \cos \theta, \ 
		(\sqrt{2} + \cos \varphi) \sin \theta, \ 
		\sin \var  \right).
	\end{equation}
By this parametrization, the area element $d \sigma_{\T}$ is given by 
	\begin{equation}\label{eq:area-element}
		d \sigma_{\T} = \left( \sqrt{2} + \cos \varphi \right) 
		d \varphi d \theta.
	\end{equation}

\noindent		For $x_0 \in \R^3$ and 
$\eta > 0$, the M\"obius inversion with respect to $\partial B_{\eta}(x_0)$ 
is defined by 
\begin{equation}\label{eq:Phix0eta}
\Phi_{x_0,\eta}(x) := 
		\frac{\eta^2}{|x - x_0|^2} ( x - x_0 ) + x_0.
\end{equation}
For any smooth compact surface 
$\Sigma \subset \R^3 \backslash \{x_0\}$, 
we set $\overline{\Sigma} := \Phi_{x_0,\eta}(\Sigma)$ 
and we denote 
the volume elements of $\Sigma$ and $\overline{\Sigma}$ 
by $d \sigma_{\Sigma}$ and $d \sigma_{\overline{\Sigma}}$ respectively. 
By the conformality of $\Phi_{x_0,\eta}$ it turns out that 
	\begin{equation}\label{1}
		d \sigma_{\overline{\Sigma}} = 
		\frac{\eta^4}{|x-x_0|^4} d \sigma_\Sigma. 
	\end{equation}

\noindent We are interested in two opposite limit behaviours of the M\"obius inversions applied to $\T$, 
namely for large and small inversion radii. 
First of all, we observe the existence and basic properties 
of the M\"obius inversions preserving the area of $\T$.

	\begin{lem}\label{101}
		Let $\bfex := (1,0,0)$. 
		Then for any $\eta > 0$, there exists a unique $\xi_{\eta} > \sqrt{2}+1$ 
		such that 
			\[
				|\Phi_{-\xi_{\eta} \bfex, \eta} ( \T ) |_{g_{0}} 
				= 4 \sqrt{2} \pi^2 = | \T |_{g_{0}}
			\]
		where $|A|_{g_{0}}$ denotes the area of a surface $A$ in the Euclidean space. 
		Moreover, the map $\eta \mapsto \xi_{\eta}$ 
		is strictly increasing in $(\sqrt{2} + 1, \infty)$ and 
		smooth. Finally, one observes that 
			\[
				\lim_{\eta \to 0} \xi_{\eta} = \sqrt{2} + 1, \qquad 
				\lim_{\eta \to \infty} \xi_{\eta} = \infty.
			\]
	\end{lem}

	\begin{proof}
First we notice that $\T \subset \R^{3} \backslash \{ \xi \bfex \}$ 
for each $\xi > \sqrt{2}+1$, so $\Phi_{-\xi \bfex, \eta} ( \T )$ 
is well-defined for every $\xi > \sqrt{2}+1$ and $\eta > 0$. 
Thus by \eqref{eq:area-element} and \eqref{1}, we observe 
	\[
		A(\xi, \eta) := 
		| \Phi_{-\xi \bfex, \eta} (\T) |_{g_{0}}
		= \eta^4  \int_{0}^{2\pi} \int_0^{2\pi} 
		\frac{\sqrt{2} + \cos \varphi} 
		{| X(\varphi,\theta) + \xi \bfex|^4} 
		d \varphi d \theta.
	\]
Moreover, it is easily seen that 
	\begin{equation}\label{eq:mono-A}
		A \in C^{\infty} 
		\left( \left( \sqrt{2} + 1 , \infty \right) \times (0,\infty) \right), 
		\quad 
		\frac{\partial A}{\partial \eta} (\xi, \eta ) > 0 > 
		\frac{\partial A}{\partial \xi} ( \xi , \eta) 
		\quad {\rm for\ all}\ (\xi,\eta) \in 
		\left( \sqrt{2} + 1 , \infty \right) \times (0,\infty). 
	\end{equation}
Thus from 
	\[
		\lim_{\eta \to 0} A(\xi,\eta) = 0, \qquad 
		\lim_{\eta \to \infty} A(\xi, \eta) = \infty,
	\]
we may find a unique $\xi_{\eta} > \sqrt{2} + 1$ so that 
$A(\xi_{\eta} , \eta ) = 4 \sqrt{2} \pi^{2}$. 
By the implicit function theorem, $\xi_{\eta}$ is smooth and 
strictly increasing from \eqref{eq:mono-A}. 
Finally, it is obvious to see that 
	\[
		\lim_{\eta \to 0} \xi_{\eta} = \sqrt{2}+1, \qquad 
		\lim_{\eta \to \infty} \xi_{\eta} = \infty, 
	\]
which completes the proof. 
	\end{proof}

\

\noindent First we observe the behaviour of 
$\xi_{\eta}$ and $\Phi_{-\xi_{\eta},\eta}$ as $\eta \to \infty$. 

	\begin{lem}\label{102}
		As $\eta \to \infty$, there holds 
			\[
				\frac{\xi_{\eta}}{\eta} = 1 + O(\eta^{-2}), \quad 
				\Phi_{-\xi_{\eta} \bfex, \eta} (x) 
				\to x - 2 \la x, \bfex \ra \bfex 
				\quad {\rm in}\ C^{\infty}_{\rm loc}(\R^{3}). 
			\]
	\end{lem}

	\begin{rem}\label{103}
		The map 
		$x \mapsto x - 2 \langle x, \bfex \rangle 
		\bfex$ 
		is the reflection with respect to the plane 
		$[x^1 = 0] := \{ (0,x^2,x^3) \; : \;  x^2,x^3 \in \R\}$.  
	\end{rem}

	\begin{proof}[Proof of Lemma \ref{102}]
First, we prove $\xi_{\eta} \eta^{-1} = 1 + O(\eta^{-2})$. 
Since $\xi_{\eta} \to \infty$ as $\eta \to \infty$, one has 
	\[
		| \xi_{\eta}^{-1} X + \bfex |^{4} 
		= 1 + 4 \xi_{\eta}^{-1} ( \sqrt{2} + \cos \varphi ) \cos \theta 
		+ O(\xi_{\eta}^{-2}).
	\]
Thus by a Taylor expansion we obtain 
	\[
		\begin{aligned}
			4 \sqrt{2} \pi^{2} 
			&= \frac{\eta^{4}}{\xi_{\eta}^{4}} 
			\int_{0}^{2\pi} \int_{0}^{2\pi} \frac{\sqrt{2}+\cos \varphi}
			{ | \xi_{\eta}^{-1} X + \bfex |^{4} } d \theta d \varphi 
			\\
			&= \frac{\eta^{4}}{\xi_{\eta}^{4}} 
			\int_{0}^{2\pi} \int_{0}^{2\pi} 
			( \sqrt{2} + \cos \varphi ) 
			\left\{ 1 - 4 \xi_{\eta}^{-1} ( \sqrt{2} + \cos \varphi ) 
			\cos \theta + O(\xi_{\eta}^{-2}) \right\} d \theta d \varphi.
		\end{aligned}
	\]
Noting that 
	\[
		\int_{0}^{2\pi} \int_{0}^{2\pi} ( \sqrt{2} + \cos \varphi )^{2} 
		\cos \theta d \theta d \varphi = 0,
	\]
we have 
	\[
		4 \sqrt{2} \pi^{2} = \frac{\eta^{4}}{\xi_{\eta}^{4}} 
		\left\{ 4\sqrt{2} \pi^{2} + O(\xi_{\eta}^{-2})  \right\}. 
	\]
Hence, $\xi_{\eta} \eta^{-1} \to 1$ as $\eta \to \infty$ and 
$\xi_{\eta} \eta^{-1} = 1 + O(\eta^{-2})$.

		For the latter claim, we first remark that 
	\[
		\Phi_{-\xi_{\eta} \bfex, \eta} (x) 
		= \frac{\eta^{2}}{| x + \xi_{\eta} \bfex|^{2}} x 
		+ \xi_{\eta} \left( \frac{\eta^{2}}{|x+\xi_{\eta}\bfex|^{2}} 
		- 1 \right) \bfex 
	\]
and 
	\[
		\frac{\eta^{2}}{|x+\xi_{\eta}\bfex|^{2}} = 
		\left( \frac{\eta}{\xi_{\eta}} \right)^{2} 
		\frac{1}{| \bfex + \xi_{\eta}^{-1} x |^{2}} 
		= ( 1 + O(\eta^{-2})) 
		\left( 1 - 2 \xi_{\eta}^{-1} \la x, \bfex \ra \bfex 
		+ O(\eta^{-2}) \right). 
	\]
Hence, 
	\[
		\Phi_{-\xi_{\eta} \bfex, \eta} (x) 
		= ( 1 + O(\eta^{-1})) x - 2 \la x , \bfex \ra \bfex 
		+ O(\eta^{-1}),
	\]
and the claim follows. 
	\end{proof}

\bigskip

\noindent		Next, we consider the behaviour of the M\"obius inversions 
as $\eta \to 0$. For this purpose, it is convenient to fix a symmetric 
point of $\Phi_{-\xi_{\eta} \bfex, \eta}$. 
Since either the translations or reflections do not affect areas, 
from Lemma \ref{101} there exists  $\tilde{\xi}_{\eta}>0$ 
for every $\eta > 0$ so that 
	\[
		| \Phi_{0,\eta} ( \T_{\tilde{\xi}_{\eta}} ) |_{g_{0}} 
		= 4 \sqrt{2} \pi^{2} \quad 
		{\rm where} \quad 
		\T_{\tilde{\xi}_{\eta}} := \T - \left( \sqrt{2} + 1 + \tilde{\xi}_{\eta} \right) 
		\bfex.
	\]
Furthermore, $\tilde{\xi}_{\eta}$ is smooth, strictly increasing in $(0,\infty)$ 
and satisfies 
	\[
		\lim_{\eta \to 0} \tilde{\xi}_{\eta} = 0.
	\]

\

\noindent
Next, we remark the following two properties of 
the M\"obius inversions:

\noindent 
(i) $\Phi_{0,\eta}(x) = \Phi_{0,1}( \eta^{-2} x)$ for every $\eta>0$ and 
$x \in \R^{3} \backslash \{0\}$. 

\noindent
(ii) For any $c_{0} \in \R\backslash \{0\}$, 
$\Phi_{0,1}$ maps the locus $\{x^{1}=c_{0}\}$ into a sphere 
of radius $1/(2 |c_{0}|)$ centred at $(c_{0}/ (2 |c_{0}|^{2}) ) \bfex$. 

\noindent
Then we define the following functions: 
	\[
		\begin{aligned}
			Y(\varphi,\theta, \eta ) &:= X(\varphi,\theta) - 
			\left( \sqrt{2} + 1 + \tilde{\xi}_{\eta} \right) \bfex, 
			\\
			Z(\bar{\varphi},\bar{\theta}, \eta) 
			&:= \Phi_{0,\eta} 
			\left( Y(\eta^2 \bar{\varphi}, \eta^2 \bar{\theta}, \eta) \right)
			= \Phi_{0,1} \left( 
			\eta^{-2} Y(\eta^2 \bar{\varphi}, \eta^2 \bar{\theta}, \eta)
			\right)
		\end{aligned}
	\]
for $(\bar{\varphi},\bar{\theta}) \in \R^{2}$. 
We shall show the following result:

\begin{lem}\label{lem:degTori}
{\rm (i)} 
$\lim_{\eta \to 0} \eta^{2} / \tilde{\xi}_{\eta} 
= 2 \sqrt[4]{2 \pi^{2}}$. 
Furthermore, $\eta^2 / \tilde{\xi}_\eta = 2 \sqrt[4]{2 \pi^2} + O(\eta^2)$. 
		
\noindent
{\rm (ii)} 
$\Phi_{0,\eta}(\T_{\tilde{\xi}_{\eta}})$ converges to 
the sphere with radius $\sqrt[4]{2\pi^{2}} $ 
centred at $-\sqrt[4]{2\pi^{2}} \bfex$ 
in the following sense: 
for any $R>0$ and $k \in \N$, if $\eta \leq 1/ R^4$, then 
	\[
		\left\| Z(\cdot, \cdot, \eta) - Z_0 \right\|_{C^k( [-R,R]^2 )}
		\leq C_{k} \eta^{3/2}
	\]
as $\eta \to 0$, where $C_k$ depends only on $k$ and 
$Z_0$ is defined by 
	\[
		Z_0(\bar{\varphi},\bar{\theta}) := 
		\Phi_{0,1} \left(  
		- \frac{1}{2 \sqrt[4]{2\pi^{2}}} \bfex
		 + (\sqrt{2}+1) \bar{\theta} \bfey + \bar{\varphi} \bfez
		 \right), \quad 
		 \bfey := ( 0, 1, 0), \quad 
		 \bfez := ( 0, 0, 1).
	\]
	\end{lem}

	\begin{proof}
For assertion (i), it is enough to prove 
	\begin{equation}\label{eq:eta2-xi}
		\frac{\eta^{4}}{\tilde{\xi}_{\eta}^{2}} 
		= 4 \sqrt{2} \pi + O(\eta^2) 
	\end{equation}
as $\eta \to 0$. To this aim, for each $\eta \in (0,1]$, we set 
	\[
		I_{\eta} := \{ (\varphi,\theta) \in [-\pi,\pi]^{2} \; : \;
		\varphi^{2} + (\sqrt{2}+1)^{2} \theta^{2} \leq \eta^{2} \}, 
		\qquad 
		J_{\eta} := [-\pi,\pi]^{2} \setminus I_{\eta}. 
	\]
Recall that $\Phi_{0,\eta}(\T_{\tilde{\xi}_{\eta}})$ has the fixed area 
$4 \sqrt{2} \pi^{2}$: 
	\begin{equation}\label{eq:area}
		4 \sqrt{2} \pi^{2} = \eta^{4} \int_{-\pi}^{\pi} \int_{-\pi}^{\pi} 
		\frac{\sqrt{2}+\cos \varphi}{|Y(\varphi,\theta,\eta)|^{4}} 
		d \varphi d \theta 
		= \eta^{4} \left( \int_{J_{\eta}} + \int_{I_{\eta}} \right) 
		\frac{\sqrt{2}+\cos \varphi}{|Y(\varphi,\theta,\eta)|^{4}} 
		d \varphi d \theta 
		=: \tilde{I}_{1} + \tilde{I}_{2}
	\end{equation}
where we used $[-\pi,\pi]^{2}$ instead of $[0,2\pi]^{2}$.

		First we shall show $\tilde{I}_{1} = O(\eta^{2})$. 
By a Taylor expansion of $Y$ around $(0,0)$, we have 
	\begin{equation}\label{eq:exp-Y}
		\begin{aligned}
			|Y(\varphi,\theta,\xi_{\eta})|^{2} 
			&=\varphi^{2} + (\sqrt{2}+1)^{2} \theta^{2} + \tilde{\xi}_{\eta}^{2} 
			+ \left\{ (\sqrt{2}+1) \theta^{2} + \varphi^{2} \right\} 
			\tilde{\xi}_{\eta} + O( \varphi^{4} + \theta^{4}).
		\end{aligned} 
	\end{equation}
Thus we may find  $C_{0}>0$, which is independent of $\eta \in (0,1]$,  so that 
	\[
		C_{0}(\varphi^{2} + (\sqrt{2}+1)^{2} \theta^{2} ) 
		\leq |Y(\varphi,\theta,\eta)|^{2} 
		\quad 
		{\rm for\ all}\ (\varphi,\theta) \in J_{\eta} \quad 
		{\rm and}\quad \eta \in (0,1].
	\]
Hence, using the change of variables 
$(\varphi,\theta)=( r \cos \Theta,  (\sqrt{2}+1)^{-1} r \sin \Theta)$ and 
noting that $J_{\eta} \subset \{ (r,\Theta) \in [\eta,4\pi] \times [0,2\pi] \}$, 
we obtain 
	\[
		\int_{J_\eta} 
		\frac{\sqrt{2}+\cos \varphi}{|Y(\varphi,\theta,\eta)|^{4}} 
		d \varphi d \theta 
		\leq 
		C_{1} \int_{J_{\eta}} 
		\frac{d \varphi d\theta}
		{(\varphi^{2} + (\sqrt{2}+1)^{2} \theta^{2})^{2}}
		= \frac{C_{1}}{\sqrt{2}+1} 
		\int_{\eta}^{4\pi} \int_{0}^{2\pi} 
		\frac{r}{r^{4}} d r d \Theta 
		\leq C_{2} \eta^{-2}.
	\]
Thus multiplying the above inequality by $\eta^{4}$ 
we get $\tilde{I}_{1}=O(\eta^{2})$.

		Next, we examine $\tilde{I}_{2}$.  
Recalling \eqref{eq:exp-Y}, we have 
	\begin{equation}\label{eq:ineq-Y2}
		\left\{1-O(\eta^{2}) - \tilde{\xi}_{\eta} \right\} 
		( \varphi^{2} + (\sqrt{2}+1)^{2} \theta^{2} ) + \tilde{\xi}_{\eta}^{2}
		\leq |Y(\varphi,\theta,\eta)|^{2} 
		\leq \left\{ 1+O(\eta^{2}) + \tilde{\xi}_{\eta}\right\} 
		\left\{ \varphi^{2} + (\sqrt{2}+1)^{2} \theta^{2} \right\} 
		+ \tilde{\xi}_{\eta}^{2}
	\end{equation}
for every $(\varphi,\theta) \in I_{\eta}$. 
Now we compute 
	\[
		\eta^{4} \int_{I_{\eta}} \frac{(\sqrt{2}+1) d \varphi d \theta}
		{\left[  \left\{1+O(\eta^{2}) + \tilde{\xi}_{\eta} \right\} \left\{ \varphi^{2} 
		+ (\sqrt{2}+1)^{2} \theta^{2} \right\} + \tilde{\xi}_{\eta}^{2} \right]^{2}}. 
	\]
Using the same change of variables as above and noting 
that $\tilde{\xi}_{\eta} \to 0$ as $\eta \to 0$, we have 
	\begin{equation}\label{eq:est-I-eta}
		\begin{aligned}
			&\eta^{4} \int_{I_{\eta}} \frac{(\sqrt{2}+1) d \varphi d \theta}
			{\left[  \left\{1+O(\eta^{2}) + \tilde{\xi}_{\eta}\right\} \left\{ \varphi^{2} 
			+ (\sqrt{2}+1)^{2} \theta^{2} \right\} + \tilde{\xi}_{\eta}^{2} \right]^{2}}
			\\ 
			= & 2 \pi \eta^{4} 
			\int_{0}^{\eta} \frac{r dr}
			{ \left[\left\{1 + O(\eta^{2}) + \tilde{\xi}_{\eta} \right\}r^{2} 
			+ \tilde{\xi}_{\eta}^{2} \right]^{2}  } 
			\\
			=& \frac{\pi \eta^{4}}{1+O(\eta^{2}) + \tilde{\xi}_{\eta}} 
			\left( \frac{1}{\tilde{\xi}_{\eta}^{2}} 
			- \frac{1}
			{\eta^{2} \left( 1 + \tilde{\xi}_{\eta} + O(\eta^{2}) \right) 
			+ \tilde{\xi}_{\eta}^{2}  } 
			\right)
			\\
			=& \left\{ 1  - \tilde{\xi}_{\eta} - O(\eta^{2})
			- O(\tilde{\xi}_{\eta}^{2}) \right\} \pi 
			\left( \frac{\eta^{4}}{\tilde{\xi}_{\eta}^{2}} + O(\eta^{2}) \right). 
		\end{aligned}
	\end{equation} 
Combining \eqref{eq:area}, \eqref{eq:ineq-Y2} and \eqref{eq:est-I-eta} 
with $\tilde{I}_{1}=O(\eta^{2})$ and 
$\sqrt{2} + \cos \varphi = \sqrt{2} + 1 + O(\eta^{2})$ in $I_{\eta}$, 
one may observe that 
	\begin{equation}\label{eq:xi-eta}
		\left\{ 1 - \tilde{\xi}_{\eta} + O(\eta^{2}) + O(\tilde{\xi}_{\eta}^{2} ) 
		\right\} \frac{\eta^{4}}{\tilde{\xi}_{\eta}^{2}}  
		= 4 \sqrt{2} \pi + O(\eta^{2}). 
	\end{equation}
Since $\tilde{\xi}_{\eta} \to 0$ as $\eta \to 0$, we have 
$\eta^4 / \tilde{\xi}_\eta^2 \to 4 \sqrt{2} \pi$. 
Hence, combining \eqref{eq:xi-eta}, we get \eqref{eq:eta2-xi}.

\

		For assertion (ii), due to a Taylor expansion, 
we first remark that 
	\[
		\cos \theta = 1 + R_c(\theta), \quad 
		\sin \theta = \theta + R_s (\theta)
	\]
where 
	\begin{equation}\label{eq:RcRs}
		\left| \frac{d^k}{d \theta^k} R_c(\theta)  \right| 
		\leq C_k |\theta|^{(2-k)_+}  , \quad 
		\left| \frac{d^k}{d \theta^k} R_s(\theta) \right| 
		\leq C_k | \theta |^{(3-k)_+}  
	\end{equation}
for all $k \in \N$ and $| \theta | \leq 1$ where $a_+:=\max\{0,a\}$. 
Using these expansions, we obtain 
	\[
		\eta^{-2} Y(\eta^{2} \bar{\varphi}, \eta^{2} \bar{\theta}, \eta ) 
		= \left( -\eta^{-2} \tilde{\xi}_\eta + R_{Y,\eta,1}(\bar{\varphi}, \bar{\theta}), 
		\quad (\sqrt{2} + 1) \bar{\theta} + R_{Y,\eta,2} 
		(\bar{\varphi}, \bar{\theta}) , \quad 
		\bar{\varphi} + R_{Y,\eta,3}(\bar{\varphi},\bar{\theta})  \right)
	\]
where 
	\begin{equation*}
		\begin{aligned}
			R_{Y,\eta,1} (\bar{\varphi}, \bar{\theta}) 
			&:=  (\sqrt{2}+1) \eta^{-2} R_c(\eta^2 \bar{\theta}) 
			+ \eta^{-2} R_c (\eta^2 \bar{\varphi}) 
			+ \eta^{-2} R_c (\eta^2 \bar{\varphi}) R_c (\eta^2 \bar{\theta} ),
			\\
			R_{Y,\eta,2} (\bar{\varphi},\bar{\theta}) &:= 
			(\sqrt{2} + 1) \eta^{-2} R_s(\eta^2 \bar{\theta} ) 
			+  \bar{\theta} R_c(\eta^2 \bar{\varphi}) 
			+ \eta^{-2} 
			R_c( \eta^2 \bar{\varphi} ) R_s( \eta^2 \bar{\theta} ),
			\\
			R_{Y,\eta,3} (\bar{\varphi}, \bar{\theta}) &:= 
			\eta^{-2} R_s(\eta^2 \bar{\varphi}).
		\end{aligned}
	\end{equation*}
By \eqref{eq:RcRs}, it is easy to verify that  
	\[
		\left\| R_{Y,\eta,1} \right\|_{C^k([-R,R]^2)} 
		\leq C_k \eta^{3/2}, \quad 
		\left\| R_{Y,\eta,2} \right\|_{C^k([-R,R]^2)} 
		\leq C_k \eta^3, \quad 
		\left\| R_{Y,\eta,3} \right\|_{C^k([-R,R]^2)} 
		\leq C_k \eta^{3}
	\]
for all $k \in \N$ provided $\eta \leq 1/R^4$. 
Noting the assertion (i) and 
	\[
		\| \Phi_{0,1} \|_{C^k (\R^3 \setminus B_{r}(0) ) } 
		\leq C_{r}
	\]
for each $r>0$, the assertion (ii) holds. 
	\end{proof}

\

\noindent By Lemmas \ref{101} \ref{102}, \ref{lem:degTori} 
and Remark \ref{103}, composing the inversions in Lemma \ref{101} 
with suitable translations, rotations and reflections, 
one can obtain a smooth two-dimensional family 
of transformations of the Clifford torus $\T$ which contains the identity and which degenerates to spheres centred on the $xy$-plane 
and passing through the origin of $\R^3$. Even though the identity map corresponds to choosing $\eta = \infty$ in Lemma \ref{101}, 
the smoothness of the entire family (near  the identity map) can be checked noticing that the above inversions can also be 
obtained composing  M\"obius transformations in $S^3$ (including the identity map) and stereographic projections from $S^3$ to $\R^3$.  
To summarize the discussion, we state the following proposition.

\begin{pro}\label{p:disk} There exists a smooth family of conformal immersions $T_\omega$ of $\T$ into $\R^3$, parametrized by 
$\o \in \DD$, $\DD$ being the unit disk of $\R^2$, which preserve the area of $\T$ and for which the following  hold 

\item{a)}  $T_0 = {\rm Id}$;  

\item{b)} for $\o \neq 0$, $T_\omega$ is an inversion with respect to a sphere centred at a point in $\R^3$ aligned to $\o$ (viewed as  
an element of $\R^3$); 

\item{c)} as $|\omega|$ approaches $1$, $T_\omega(\T)$ degenerates to a sphere of radius $\sqrt[4]{2\pi^{2}} $ centred at 
$\sqrt[4]{2\pi^{2}} \frac{\o}{|\o|}$.

In what follows, we use the symbol $\T_{\omega}$ for $T_{\omega}(\T)$.

\end{pro}

\begin{rem}\label{r:projective}
If we wish to rotate in $\R^3$ the images $\T_{\omega}$,  the resulting  surfaces can be 
parametrized by the family of tangent vectors to $\R \P^2$ with length less than $1$.   
\end{rem}


\subsection{Basic properties of the Willmore  functional} \label{SS:BasicW}


		In this subsection, we state basic properties of the Willmore functional 
as well as the non-degeneracy condition of $\T$. 
First of all, we begin with a basic property of the Willmore functional $W_0$ 
in the Euclidean space for immersions $i : \Sigma \to \R^3$ 
$$
     W_0(i(\Sigma)) = \int_{\Sigma} H^2 d \s,  
$$
namely its conformal invariance. 

\begin{pro}\label{p:Mobinv}
Let $\Sigma$ be a compact closed surface of class $C^2$  and let $i : \Sigma \to \R^3$ be an immersion. 
Then, if $\l > 0$ and if  $\Phi_{x_0,\eta}$ is as in \eqref{eq:Phix0eta}, one has the invariance properties 
$$
  j) \quad W_0(\l i(\Sigma)) = W_0(i(\Sigma)) \quad \quad \hbox{ and } \quad \quad jj) \quad W_0((\Phi_{x_0,\eta} 
  \circ i)(\Sigma)) = W_0(i(\Sigma)) 
  \quad \hbox{ provided } x_0 \not\in i(\Sigma). 
$$
\end{pro}

\

\noindent We also recall the first and second variation formulas for the Willmore functional 
$$
  W(i(\Sigma)) = \int_{\Sigma} H^2 d \s 
$$
in a general setting, namely for a surface immersed in a three-dimensional manifold $M$. We perturb the 
surface through  a  variation with normal speed $\var \, n$, where $n$ stands for the 
unit outer normal to the image of $\Sigma$. 
We denote the perturbed surfaces by $F(t,p)$ ($t \in (-\e ,\e)$, $p \in \Sigma$). Calling $n=n(t,p)$  the outward pointing unit normal to the immersion $F(t,\cdot)$ we denote with $\varphi:=g(n,  \partial_t F)$ the normal velocity,
we also let $\bar{g}_{ij}$ be the 
first fundamental form of $i(\S)$ and by $d \s$ the induced area element. 

We denote by $\Riem$ the Riemann curvature tensor of $M$ at $P$, 
by $\Ric$ the Ricci tensor, by $\Sc$ the scalar curvature, 
by $A$ the second fundamental form of $i(\Sigma)$, by $\mathring{A} = A - \frac{1}{2} H \bar{g}$ the 
traceless part of $A$, by $\Delta$ the Laplace-Beltrami operator on $(\Sigma,\bar{g})$ and we define the elliptic, self-adjoint operator 
\[ 
 L \var  := - \D  \var - \var \left( |A|^2 + \Ric(n, n) \right). 
\]
Here, for $\Riem$ and $A$, we use the following convention: 
	\[
		\Riem (X,Y)Z = \n_X \n_Y Z - \n_Y \n_X Z - \n_{[X,Y]} Z, \quad 
		A(X,Y) = g \left( \n_X n, Y \right). 
	\]
We also define the one-form $\varpi$ to be the tangent component of the one-form $\Ric(n, \cdot)$ in $M$, namely 
$$
  \varpi = \Ric(n, \cdot)^{t}, 
$$
and the  $(2,0)$ tensor $T$ by 
$$
  T_{ij} = \Riem(\pa_i, n, n, \pa_j) = \Ric_{ij} + G(n, n) \bar{g}_{ij},
$$
where $G = \Ric - \frac{1}{2} \Sc  g$ is the Einstein tensor of $M$. 

We have then the following formulas, see Section 3 in \cite{LMS}.

\begin{pro}\label{p:1-2-var}
With the above notation we have the formulas 
$$
  W'(i(\Sigma))[\var] = \int_{\S } \left( L H + \frac{1}{2} H^3 \right) \var \, d \s 
$$
and 
\[
	\begin{aligned}
		 W''(i(\Sigma))[\var,\var] 
		 & =  
		 2 \int_{\S } \left[ (L \var)^2 + 
		 \frac{1}{2} H^2 |\n \var |^2 - 2 \mathring{A}(\n \var , \n \var ) \right] 
		 d \s \\ 
  		& \quad + 2 \int_{\S} \var^2  \bigg( |\n  H|^2_{g} + 2 \varpi(\n H) 
		+ H \D H + 2 g(\n^2 H, \mathring{A}) + 2 H^2 |\mathring{A}|^2_{g} 
		\\
		& \quad +  2 H g (\mathring{A}, T) - H  g (\n_n  \Ric) (n, n) 
		- \frac{1}{2} H^2 |A|^2_{g} - \frac{1}{2} H^2 \Ric(n, n) \bigg) d \s  
		\\
		& \quad + \int_{\Sigma} \left(L H + \frac{1}{2}H^3\right) 
		\left( \frac{\partial \varphi}{\partial t}\bigg|_{t=0} + H \varphi^2 \right) d \s
		\\
 		& = 2 \int_\S \var \tilde{L} \var \, d \s 
 		+ \int_{\Sigma} \left(L H + \frac{1}{2}H^3\right) 
 		\left( \frac{\partial \varphi}{\partial t}\bigg|_{t=0} + H \varphi^2 \right) d \s
	\end{aligned}
\]
where the fourth-order operator $\tilde{L}$ is defined by 
	\[ 
		\begin{aligned}
			\tilde{L} \var & = L L \var + \frac{1}{2} H^2 L \var 
			+ 2 H g (\mathring{A}, \nabla^2 \var) + 2 H \varpi(\n \var ) 
			+ 2 \mathring{A}(\n \var, \n H) \\
			& \quad +  \var  \left( |\n H|^2_{g} + 2 \varpi (\n H) + H \D H 
			+ 2 g(\n^2 H, \mathring{A}) + 2 H^2 |\mathring{A}|^2_{g} 
			+ 2 H g(\mathring{A}, T) - H (\n_n \Ric)(n, n) \right). 
		\end{aligned}
	\]
\end{pro}

\begin{rem}\label{rem:sec-der-will}
Let us identify $W'(i (\Sigma))$ with the function $LH + H^3/2$ on $\Sigma$. 
Then the derivative of this function with respect to the above perturbation 
is given by $\tilde{L} \varphi$. 	
\end{rem}

\

\noindent In particular, when $\Sigma$ is embedded in the Euclidean space $\R^3$ we obtain the following result, 
see also \cite{WEINER}. 

\begin{cor}\label{c:1-2-flat}
 $$
   W'_0(i(\Sigma))[\var] = \int_{\S } \left( L_0 H + \frac{1}{2} H^3 \right) \var  \, d \s;   \qquad \quad L_0 \psi  = 
   - \D  \psi - \psi  |A|^2, 
 $$
 and 
	\[
		\begin{aligned}
			W''_0(i(\Sigma))[\var,\var] 
			& = 
			2 \int_{\S } \left[ (L_0 \var)^2 + \frac{1}{2} H^2 |\n \var |^2 
			- 2 \mathring{A}(\n \var , \n \var ) \right] d \s 
			\\ 
			& \quad + 2 \int_{\S} \var^2  \left( |\n  H|^2 + H \D H 
			+ 2 \la \n^2 H, \mathring{A} \ra + 2 H^2 |\mathring{A}|^2 
			- \frac{1}{2} H^2 |A|^2  \right) d \s  
			\\
			& \quad 
			+ \int_{\Sigma} \left( L_0 H + \frac{1}{2} H^3 \right) 
			\left( \frac{\partial \varphi}{\partial t} \bigg|_{t=0} + H \varphi^2 \right) d \s
			\\
			& = 2 \int_\S \var \tilde{L}_{0} \var \, d \s 
			+ \int_{\Sigma} \left( L_0 H + \frac{1}{2} H^3 \right) 
			\left( \frac{\partial \varphi}{\partial t} \bigg|_{t=0} + H \varphi^2 \right) d \s, 
		\end{aligned}
	\]
where the fourth-order operator $\tilde{L}_{0}$ is defined by 
	\[ 
		\begin{aligned}
			\tilde{L}_{0} \var & = (L_0)^2 \var + \frac{1}{2} H^2 L_0 \var 
			+ 2 H \la\mathring{A}, \nabla^2 \var\ra 
			+ 2 \mathring{A}(\n \var, \n H) 
			\\ 
			& \quad + \var  \left( |\n H|^2  + H \D H + 2 \la \n^2 H, \mathring{A}\ra
			 + 2 H^2 |\mathring{A}|^2  \right). 
		\end{aligned}
	\]
\end{cor}

\

\noindent The set  $\{P+\lambda R \T_\omega \; : \; P \in \R^3, \lambda>0, R \in SO(3), \omega \in \DD \}$ made by translations, rotations and dilations of M\"obius transformations of the Clifford torus $\T$, form an eight dimensional family of surfaces.
On $\T_\o$, we define  Jacobi vector fields 
$Z_{0,\o}, Z_{1,\o}, \dots, Z_{7,\o}$ according to the following notation 
\[
Z_{0,\o} \hbox{ is generated by dilations }; \qquad \qquad 
Z_{1,\o}, Z_{2,\o}, Z_{3,\o} \hbox{ are generated by translations };
\]
\[
Z_{4,\o}, Z_{5,\o}, Z_{6,\o} \hbox{ are generated by rotations };
\]
\[
Z_{7,\o} \hbox{ is generated by 
  the M\"obius inversions described in Subsection \ref{ss:invfixarea}}. 
\]
Notice that $Z_{1,\o}, \dots, Z_{7,\o}$ all induce deformations which preserve the area. 
We write  $Z_{i,R,\o}$ for the Jacobi vector fields on 
$R \T_{\omega}$.

		Since $R \T_\o $ is diffeomorphic to $\T$ 
by the conformal map $R T_\o $, 
we may pull back a neighbourhood of $R\T_\o $ in $\R^3$ 
onto that of $\T$. 
We write $g_{0,R,\o} := (RT_\o)^{\ast} g_0$ 
for the pull back of the Euclidean metric $g_0$.  
Thus $(\T,g_{0,R,\o})$ is isometric to $(R\T_\o, g_0 )$. 
We also use $\bar{g}_{0,R,\o}$, $n_{0,R,\o}(p)$ and 
$\la \cdot , \cdot \ra_{L^2_{0,R,\o}}$ 
for the tangential metric of $g_{0,R,\o}$ on $\T$, 
the unit outer normal of $\T$ in $g_{0,R,\o}$ and 
the $L^2$-inner product with metric $\bar{g}_{0,R,\o }$. 
Moreover, we write $L^2_{0,R,\o}(\T)$ for 
$(L^2(\T), \la \cdot , \cdot \ra_{L^2_{0,R,\o}} )$. 
Remark that we may regard $Z_{i,R,\o }$ as functions on 
$\T$.

		Next, we set 
\[
	\mathcal{K}_{0,R,\omega} := 
	{\rm span}\, \{ H_{0,R,\omega}, \ Z_{1,R,\omega}, \ldots, Z_{7,R,\omega} 
	\} \subset L^2_{0,R,\o}(\T)
\]
where $H_{0,R,\omega}$ stands for the mean curvature of $\T$ in $g_{0,R,\o}$. 
Let us also denote by 
$\Pi_{0,R,\o} : L^2_{0,R,\o}(\T)\to (\mathcal{K}_{0,R,\o})^{\perp \bar{g}_{0,R,\o}}$ 
the $L^2$-projection onto $(\mathcal{K}_{0,R,\o})^{\perp \bar{g}_{0,R,\o}}$ 
where $\perp \bar{g}_{0,R,\o}$ stands for the orthogonality in $L^2_{0,R,\o}(\T)$: 
	\[
		\psi \in \mathcal{K}^{\perp \bar{g}_{0,R,\o}}_{0,R,\omega} 
		\quad \Leftrightarrow \quad 
		\left\la H_{0,R,\o}, \psi \right\ra_{L^2_{0,R,\o}} 
		= 0 =
		\left\la Z_{i,R,\omega}, \psi \right\ra_{L^2_{0,R,\o}}  
		\ {\rm (} 1 \leq i \leq 7 {\rm )}.
	\]

	 For regular functions $\varphi: \T \to \R$, 
we consider the following perturbation of $\T$: 
	\[
		(\T [\varphi])_{R,\o} := 
		\left\{ p + \varphi(p) n_{0,R,\o}(p) \; : \; p\in \T \right\}, 
		\qquad 
		R \T_\o [\varphi] := \left\{ R T_\o\left( p + \varphi(p) n_{0,R,\o} (p)\right)  
		\; : \; p \in \T  \right\}
	\]
Notice that the surface $((\T[\varphi])_{R,\o}, g_{0,R,\o})$ is isometric 
to $(R\T_{\o} [\varphi] , g_0)$. 
Concerning these perturbations, 
we have the following non-degeneracy property 
for the Willmore energy of the Clifford torus 
and its M\"obius equivalents, see Theorem 4.2 and Corollaries 1, 2 
in Weiner \cite{WEINER}. 

%

	\begin{pro}[\cite{WEINER}]\label{p:nondeg}
		Let $\T_{\omega}$ be the immersion of $\T$ in $\R^{3}$ as in 
		Proposition \ref{p:disk} and $\a \in (0,1)$: 
		then for all $R \in SO(3)$ and $\omega \in \DD$, 
		one has $W_{0}'(R \T_{\omega} ) = 0$. Moreover, $R \T_{\omega}$ 
		is non-degenerate in the following sense. 
		For any compact set $K \subset \DD \backslash \{0\}$ 
		and $\ell \in \N$, 
		there exist  constants $C_{K,\ell,1}>0$ and $C_{K,\ell,2} > 0$  
		such that for every $R \in SO(3)$ and $\omega \in K$, 
		one has the lower bound
			\begin{equation}\label{eq:CK}
				\left\| \Pi_{0,R,\o} \tilde{L}_{0,R,\o} \varphi 
				\right\|_{C^{\ell,\a} ( \T )  } 
				\geq C_{K,\ell,1} \left\| \tilde{L}_{0,R,\o} \varphi 
				\right\|_{C^{\ell,\a} ( \T  )  } 
				\geq C_{K,\ell,2} \left\|  \varphi \right\|_{C^{4+\ell,\a} (\T )  }
			\end{equation}
		for all $\varphi \in C^{4+\ell,\a}(\T) \cap 
		\mathcal{K}_{0,R,\o}^{\perp \bar{g}_{0,R,\o}}$ 
		where $\tilde{L}_{0,R,\o}$ is the operator $\tilde{L}_0$ for $R \T_\o $. 
\end{pro}

\begin{proof}
We first prove that 
	\begin{equation}\label{eq:esti-1}
		\left\| \tilde{L}_{0,R,\o} \varphi \right\|_{C^{\ell,\a} (\T)  } 
		\geq C_{K,\ell} \left\|  \varphi \right\|_{C^{4+\ell,\a} (\T)  }. 
	\end{equation}
We prove this by contradiction. Assume that 
there are $(R_m) \subset SO(3)$, $(\o_m) \subset K$ and 
$(\varphi_m) \subset C^{4+\ell,\a}(\T)$ such that 
	\begin{equation}\label{eq:as}
		\begin{aligned}
			&\| \varphi_m \|_{C^{4+\ell,\a}(\T) } = 1, \quad 
			\left\| \tilde{L}_{0,R_m,\o_m} \varphi_m \right\|_{C^{\ell,\a} (\T)  } 
			\to 0, 
			\\
			& \la H_{0,R_{m},\o_{m}}, \varphi_{m} \ra_{L^2_{0,R_m,\o_m}}
			= 0 = \la Z_{i,R_{m}\o_{m}} ,\varphi_{m} \ra_{L^2_{0,R_m,\o_m}}
			\quad (1 \leq i \leq 7).
		\end{aligned}
	\end{equation}
We may also suppose $R_m \to R_0$ and $\o_m \to \o_0$. 
Then since the embedding $C^{4+\ell,\a} (\T) \subset C^{4+\ell,\beta}(\T)$ 
is compact, by \eqref{eq:as}, we also suppose that 
$\varphi_m \to \varphi_0$ strongly in $C^{4+\ell,\beta}(\T)$.

		Next we prove $\varphi_0 \neq 0$. 
If $\varphi_0 \equiv 0$, then from the definition of $\tilde{L}_{0,R,\o}$, isolating the highest-order terms, 
we have 
	\[
		(L_{0,R_m,\o_m})^2 \varphi_m = f_m 
	\]
where $f_m$ has a strongly convergent subsequence in 
$C^{\ell,\a}(\T)$. 
Thus by the elliptic regularity theory, we observe that 
$\| \varphi_m \|_{C^{4+\ell,\a} (\T) } \to 0$, which contradicts  \eqref{eq:as}. 
Hence $\varphi_0 \not\equiv 0$.

		Next, by \eqref{eq:as} and the fact $\varphi_m \to \varphi_0$ 
strongly in $C^{4+\ell,\beta}(\T)$, we notice that 
$\tilde{L}_{0,R_0,\o_0} \varphi_0 = 0$ 
and $\varphi_0 \in C^\infty(\T)$ thanks to the elliptic regularity. 
By the result of Weiner \cite{WEINER}, we have 
	\[
		{\rm Ker}\, \tilde{L}_{0,R_0,\o_0} = {\rm span}\, 
		\{ Z_{0,R_{0},\o_{0}}, \ldots, Z_{7,R_{0},\o_{0}} \}. 
	\]
We apply a Graham-Schmidt orthogonalization 
to $Z_{1,R_{0},\o_{0}}, \ldots, Z_{7,R_0,\o_0}$ in $L^2_{0,R_0,\o_0}(\T)$ to 
get new variation fields ${Y}_{1,R_0,\o_0}, \ldots, {Y}_{7,R_0,\o_0}$. 
Then $\varphi_0$ can be expressed as 
	\[
		\varphi_0 = a_0 Z_{0, R_0, \o_0} + 
		\sum_{i=1}^7 a_i {Y}_{i,R_0,\o_0}.
	\]
Since $Z_{i,R_0,\o_0}$ ($1 \leq i \leq 7$) are Jacobi vector fields 
corresponding to the area preserving deformation, it follows that 
	\[
		\la Z_{i,R_0,\omega_0} , H_{0,R_0,\omega_0} \ra_{L^2_{0,R_0,\o_0}} = 0 
		\quad (1 \leq i \leq 7). 
	\]
By \eqref{eq:as} and the definitions of $\{Y_{i,R_0,\omega_0}\}_{1 \leq i \leq 7}$, 
one observes that 
	\begin{equation}\label{eq:27}
		0 = \la \varphi_0, H_{0,R_0,\o_0} \ra_{L^2_{0,R_0,\o_0}} = 
		a_0 \la Z_{0,R_0, \o_0 }, H_{0,R_0,\o_0} \ra_{L^2_{0,R_0,\o_0}}. 
	\end{equation}
Noting that $Z_{0,R_0,\o_0}$ is generated by dilations, 
we obtain 
	\begin{equation}\label{eq:bdd-below}
		0 < 8 \sqrt{2} \pi^{2} 
		= \frac{d}{dt} ( 1 + t)^{2} | R_{0} \T_{\o_{0}} |_{g_{0}} \Big|_{t=0}
		= \frac{d}{dt} | ( 1 + t)^{2} R_{0} \T_{\o_{0}} |_{g_{0}} \Big|_{t=0} 
		= \la Z_{0,R_{0},\o_{0}} ,H_{0,R_{0},\o_{0}} \ra_{L^2_{0,R_0,\o_0}}.
	\end{equation}
Hence, from \eqref{eq:27}, we deduce that $a_0=0$ and 
$\varphi_0 = \sum_{i=1}^7 a_i Y_{i,R_0,\omega_0}$. 
Due to \eqref{eq:as} and the definitions of $\{Y_{i,R_0,\omega_0}\}_{1 \leq i \leq 7}$, 
we have $\la \varphi_0, Y_{i,R_0,\omega_0} \ra_{L^2_{0,R_0,\o_0}} = 0$ 
($1\leq i \leq 7$), 
which yields $a_i = 0$ ($1 \leq i \leq 7$). 
However, this contradicts $\varphi_0 \not\equiv 0$. 
Thus \eqref{eq:esti-1} holds.

		Next we show that 
	\begin{equation}\label{eq:esti-2}
		\left\| \Pi_{0,R,\o} \tilde{L}_{0,R,\o} \varphi \right\|_{C^{\ell,\a} (\T)  } 
		\geq C_{K,\ell} \left\| \tilde{L}_{0,R,\o} \varphi \right\|_{C^{\ell,\a} (\T)  }. 
	\end{equation}
Again we argue by contradiction and suppose that 
there are $(R_m) \subset SO(3)$, $(\o_m) \subset K$ and 
$(\varphi_m) \subset C^{4+\ell,\a}(\T)$ such that 
	\begin{equation}\label{eq:as-2}
		\begin{aligned}
			&\left\| \Pi_{0,R_m,\o_m} \tilde{L}_{0,R_m,\o_m} 
			\varphi_m \right\|_{C^{\ell,\a} (\T)  } \to 0, \quad 
			\left\| \tilde{L}_{0,R_m,\o_m} \varphi_m \right\|_{C^{\ell,\a} (\T)  } 
			= 1,\\
			& \la H_{0,R_{m},\o_{m}}, \varphi_{m} \ra_{L^2_{0,R_m,\o_m}}
				= 0 = \la Z_{i,R_{m}\o_{m}} ,\varphi_{m} \ra_{L^2_{0,R_m,\o_m}}
				\quad (1 \leq i \leq 7).
		\end{aligned}
	\end{equation}

		By \eqref{eq:esti-1}, we observe that 
$( \| \varphi_m \|_{C^{4+\ell,\a} (\T) })$ is bounded. 
As  above, we use the 
$L^2_{0,R_m,\o_m}(\T)$-orthogonal system 
$\{Y_{i,R_m,\o_m}\}_{1 \leq i \leq 7}$ starting from 
$\{Z_{i,R_m,\o_m}\}_{1 \leq i \leq 7}$. 
Since $\la Y_{i,R_m,\o_m} , H_{0,R_m,\o_m} \ra_{L^2_{0,R_m,\o_m}} = 0$ holds 
due to the area preserving properties, 
set $Y_{0,R_m,\o_m} := H_{0,R_m,\o_m} / \| H_{0,R_m,\o_m} \|_{L^2_{0,R_m,\o_m}}$. 
Then $(Y_{i,R_m,\o_m})_{0 \leq i \leq 7}$ is an 
$L^2_{0,R_m,\o_m}(\T)$-orthogonal system. 
Using $(Y_{i,R_m,\o_m})_{0 \leq i \leq 7}$, 
$\Pi_{0,R_m,\o_m} \tilde{L}_{0,R_m,\o_m} \varphi_m$ is written as 
	\[
		\Pi_{0,R_m,\o_m} \tilde{L}_{0,R_m,\o_m} \varphi_m 
		= \tilde{L}_{0,R_m,\o_m} \varphi_m 
		- \sum_{i=0}^7 
		\la \tilde{L}_{0,R_m,\o_m} \varphi_m , Y_{i,R_m,\o_m} \ra_{L^2_{0,R_m,\o_m}} 
		Y_{i,R_m,\o_m}. 
	\]
Since $Y_{i,R_m,\o_m} \in {\rm Ker}\, \tilde{L}_{0,R_m,\o_m}$ ($1 \leq i \leq 7$), 
we get 
	\[
		\la \tilde{L}_{0,R_m,\o_m} \varphi_m , Y_{i,R_m,\o_m} \ra_{L^2_{0,R_m,\o_m} }
		= \la  \varphi_m , \tilde{L}_{0,R_m,\o_m} Y_{i,R_m,\o_m} \ra_{L^2_{0,R_m,\o_m}} 
		= 0
	\]
for all $ 1 \leq i \leq 7$. Thus 
	\begin{equation}\label{eq:ex-tL}
		\Pi_{0,R_m,\o_m} \tilde{L}_{0,R_m,\o_m} \varphi_m 
		= \tilde{L}_{0,R_m,\o_m} \varphi_m 
		- \la \tilde{L}_{0,R_m,\o_m} \varphi_m , Y_{0,R_m,\o_m} \ra_{L^2_{0,R_m,\o_m} }
			 Y_{0,R_m,\o_m}. 
	\end{equation}
Noting that 
$\la \tilde{L}_{0,R_m,\o_m} \varphi_m , Z_{0,R_m,\o_m} \ra_{L^2_{0,R_m,\o_m} }
= \la \varphi_m , \tilde{L}_{0,R_m,\o_m} Z_{0,R_m,\o_m} \ra_{L^2_{0,R_m,\o_m}}= 0$ 
and taking an $L^2_{0,R_m,\o_m}(\T)$-inner product of 
$\Pi_{0,R_m,\o_m} \tilde{L}_{0,R_m,\o_m} \varphi_m$ and $Z_{0,R_m,\o_m}$,  
we observe that 
	\[
		\begin{aligned}
			&\la \Pi_{0,R_m,\o_m} \tilde{L}_{0,R_m,\o_m} \varphi_m, Z_{0,R_m,\o_m} 
			\ra_{L^2_{0,R_m,\o_m} }
			\\
		= &- \la \tilde{L}_{0,R_m,\o_m} \varphi_m , Y_{0,R_m,\o_m} \ra_{L^2_{0,R_m,\o_m} }
			\la Y_{0,R_m,\o_m},  Z_{0,R_m,\o_m} \ra_{L^2_{0,R_m,\o_m}}.
		\end{aligned}
	\]
Recalling \eqref{eq:bdd-below} and \eqref{eq:as-2}, we infer that 
	\[
		\la \tilde{L}_{0,R_m,\o_m} \varphi_m , Y_{0,R_m,\o_m} \ra_{L^2_{0,R_m,\o_m}} 
		\to 0. 
	\]
Thus from \eqref{eq:ex-tL} we get a contradiction: 
	\[
		1 = \left\| \tilde{L}_{0,R_m,\o_m} \varphi_m \right\|_{C^{\ell,\a} (\T)  } 
		\to 0.
	\]
Hence, \eqref{eq:esti-2} holds and we complete the proof. 
\end{proof}

\begin{rem}\label{r:nondeg}
When $\o = 0$, $\T_\o$ is the symmetric Clifford torus, so the action of one among the 
rotation vector fields is trivial. A global version of Proposition \ref{p:nondeg}, including also the 
case $\o = 0$, can be expressed as the non-degeneracy of the family of inverted 
(and rotated) Clifford tori in the sense of Bott. 
In particular the constants $C_{K,1}$ and $C_{K,2}$ in 
\eqref{eq:CK} remains controlled when $\o$ approaches zero. 
For instance, when $R = {\rm Id}$ and 
$Z_{0,R,6}$ corresponds to the rotation 
along the $z$-axis, writing $\omega = (x,y)$, 
it suffices to replace the definition of 
$\mathcal{K}_{0,{\rm Id},\omega}$ by 
	\[
		\mathcal{K}_{0,{\rm Id},\omega} := 
		{\rm span}\, \left\{ H_{0,{\rm Id},\omega}, \ Z_{1,{\rm Id},\omega}, \ldots, 
		Z_{5,{\rm Id},\omega},  \frac{\partial}{\partial x} T_{\omega}, 
		\frac{\partial}{\partial y} T_{\omega}
		\right\}
	\]
near $\omega = 0$. 
\end{rem}


\section{Finite-dimensional reduction} \label{Sec:FDR}


In this section we reduce the problem to a finite-dimensional one, namely to the choice of the concentration point and 
to the M\"obius action on the scaled Clifford tori. We first introduce a family of approximate solutions to our problem,
and then we modify them properly to solve the equation up to elements in 
the Kernel of the operator $\tilde{L}$. In the next section we will adjust the parameters so that  the Willmore equation is fully solved (under area constraint).

\subsection{Approximate solutions: small tori embedded in $M$}\label{ss:appsol}

We fix a  compact set $K$ of the unit disk $\DD$ and we consider then the family 
\[
   \hat{\mathcal{T}}_{\e, K} 
   = 
   \left\{ \e  \, R \, \T_\o \; : \; R \in SO(3), \o \in K  \right\}. 
\]
We remark that, by construction, elements in $\hat{\mathcal{T}}_{\e, K}$ consist of  Willmore  surfaces in 
$\R^3$ all with area identically equal to $4 \sqrt{2} \pi^2 \e^{2} $.

\begin{rem}\label{r:inv}
Notice that, by rotation invariance of the Clifford torus $\T$, the above family $\hat{\mathcal{T}}_{\e, K}$ 
is four-dimensional and not five-dimensional. 
\end{rem}

\

		\noindent
		Next, we construct a family of surfaces in $M$ 
by  exponential maps of $\hat{\mathcal{T}}_{\e, K}$. 
Fix $P \in M$. Around $P$, we may find 
a local orthonormal frame $\{F_{P,1},F_{P,2},F_{P,3}\}$. 
By this frame, we may identify $T_PM$ with $\R^3$ and 
define the exponential map $\exp_P^g$. 
Since $M$ is compact, there exists a $\rho_{0} > 0$, 
which is independent of $P \in M$, such that 
$\exp_{P}$ is diffeomorphic on $B_{\rho_{0}}(0) \subset \R^{3}$ 
to $\exp_{P}(B_{\rho_{0}}(0))$ for every $P \in M$. 
Then we may select $\e_{0}>0$ so that 
	\[
		\e R \T_{\o} \subset B_{\rho_{0}} (0) \qquad 
		{\rm for\ every}\ (\e,R,\o) \in (0,\e_{0}] \times SO(3) \times \DD.
	\]
Hereafter, we fix $\e_{0}>0$ and consider the case $\e \in (0,\e_{0})$. 
For a compact set $K \subset \DD$, we define 
\[
   \mathcal{T}_{\e, K} 
   = \left\{ \exp_P(\Sigma) \; : \; P \in M, 
   \Sigma \in \hat{\mathcal{T}}_{\e, K}  \right\} 
   \qquad {\rm for} \ \e \in (0,\e_{0}).
\]

	\noindent
	Before proceeding further, we make some remarks. 
Since we are interested in the asymptotic behaviour of the family 
satisfying the small area constraint, 
it is useful to introduce the following metric on $M$: 
	\begin{equation}\label{eq:def-ge}
		g_{\e}(P) := \frac{1}{\e^{2}} g(P). 
	\end{equation}

		\noindent
As above, by \eqref{eq:def-ge}, 
putting $F_{\e ,P ,\alpha} := \e F_{P,\alpha}$, 
$\{F_{\e ,P,1},F_{\e ,P,2}, F_{\e ,P,3}\}$ is a local orthonormal frame 
with metric $g_\e $ 
and using this frame, we may define the exponential map $\exp_{P}^{g_\e}$. 
Let us denote by $g_P$ and $g_{\e ,P}$ the pull back of the metrics 
$g$ and $g_\e $ through $\exp_{P}^{g}$ and $\exp_{P}^{g_\e}$: 
$g_P := (\exp_P^g)^\ast g$ and $g_{\e ,P} := (\exp_P^{g_\e})^\ast g_\e $. 
Then it is easily seen that 
	\begin{enumerate}
		\item[(i)] 
		Write $W_g$ and $W_{g_\e}$ for the Willmore functional 
		on $(M,g)$ and $(M,g_{\e})$. 
		Let $\Sigma \subset M$ be an embedded surface. 
		Denote by $H_g$ and $H_{g_\e}$ the mean curvature of 
		$\Sigma$ with the metric $g$ and $g_\e $, respectively. 
		Then we have 
		\[
		H_{g_\e} = \e H_g, \quad W_{g_\e} (\Sigma) = W_{g} ( \Sigma ), \quad 
		W_{g_\e}'(\Sigma) = \e^3 W_{g}'(\Sigma).   
		\]
		In particular, $\Sigma$ is a Willmore surface with the area constraint in 
		$(M,g)$ if and only if so is in $(M,g_{\e})$. 
		\item[(ii)] 
		The exponential map $\exp_{P}^{g_{\e}}$ 
		is defined in $B_{\e^{-1} \rho_{0}} (0)$ and 
		\[
		\exp_{P}^{g}(\e z) = \exp_{P}^{g_{\e}}(z) \qquad 
		{\rm for\ all}\  |z|_{g_{0}} \leq \e^{-1} \rho_{0}.
		\]
		\item[(iii)] 
		The metric $g_P$ has the following expansion 
		(see, for instance, Lee--Parker \cite{LP}): 
			\begin{equation}\label{eq:ex-g}
			g_{P,\alpha\beta}(x) 
			= \delta_{\alpha \beta} 
			+ \frac{1}{3} R_{\alpha \mu \nu \beta} 
			x^{\mu} x^{\nu} + R(x)  
			=: \delta_{\alpha \beta} + \tilde{h}_{P,\alpha \beta}(x)
			\quad {\rm in}\ \overline{B_{\rho_{0}}(0)} \subset T_{P}M
			\end{equation}
			where $R(x)$ is the remainder term satisfying 
			$|\nabla^k R(x) | = O(|x|^{3-k})$ for $k=0,1,2,3$. 
			Moreover, since $M$ is three dimensional, one sees that 
			\begin{equation}\label{eq:RicciDec}
				\begin{aligned}
					R_{\alpha\mu\nu\beta} 
					&= 
					g_P ( {\rm Riem}(\partial_{\nu}, \partial_{\beta} ) \partial_{\mu},
					 \partial_{\alpha} ) 
					 \\
					&= 
					\frac{{\rm Sc}_{P}}{2}( g_{P,\alpha\beta} g_{P,\mu\nu} 
				- g_{P,\alpha\nu} g_{P,\mu\beta} ) 
				+ g_{P,\alpha \nu} R_{\mu \beta} 
				- g_{P,\alpha \beta} R_{\mu \nu}
				+ g_{P,\mu \beta} R_{\alpha \nu} 
				- g_{P,\mu \nu} R_{\alpha \beta}
			\end{aligned}
		\end{equation}
			where ${\rm Sc}_{P}$ and $R_{\alpha \beta}$ are the scalar curvature 
			and the Ricci tensor evaluated at $P$, respectively. 
		\item[(iv)] 
		For $g_{\e ,P}$, we have 
		\begin{equation}\label{eq:ge=d+eh}
		g_{\e ,P,\alpha \beta} (y) = 
		\d_{\alpha \beta}+ \e^2 h^\e_{P,\alpha \beta}(y) 
		\qquad {\rm for\ any}\ |y|_{g_{0}}\leq \e^{-1} \rho_{0}
		\end{equation}
		where $h^{\e}_{P,\alpha \beta}(y) = \e^{-2} \tilde{h}_{P,\alpha \beta} (\e y)$ 
		and there holds 
		\begin{equation}\label{eq:esti-he}
		|y|^{-2} |h_{P,\alpha \beta}^\e| + |y|^{-1} | \nabla h_{P,\alpha \beta}^\e | + 
		\sum_{i=2}^{k} | \nabla^{i} h_{P,\alpha \beta}^{\e} (y) |_{g_{0}} 
		\leq \tilde{h}_{0,k} 
		\end{equation}
		for any $k \in \N$ and $|y|_{g_0} \leq \e ^{-1} \rho_0$. 
		Here $\tilde{h}_{0,k}$ depends on $k$, but not on $\e$. 
		Moreover, $(P,y) \mapsto g_{\e ,P,\alpha \beta}(y)$ is smooth and 
		for any $k,\ell \in \N$, there exists a $C_{k,\ell} > 0$ such that 
		\begin{equation}\label{eq:met-deri}
		|y|^{-2}| D_P^{k+1} g_{\e ,P,\alpha \beta} (y)| 
		+ 
		|y|^{-1} | D_P^{k+1} \nabla_y g_{\e ,P,\alpha \beta}(y)| 
		+ \sum_{j=2}^{\ell}
		|D_P^{k+1} \nabla_y^{j} g_{\e ,P,\alpha \beta} (y) | 
		\leq C_{k,\ell} \e^2
		\end{equation}
		where $D_P$ denotes the differential by $P$ in the original scale of $M$ 
		(not in the rescaled one). 
		\item[(v)] 
		The family ${\mathcal T}_{\e, K}$ is rewritten as 
		\begin{equation}\label{eq:tke}
		\begin{aligned}
		{\mathcal T}_{\e, K} 
		&= 	
		\left\{ \exp_{P}^{g}(\e R \T_{\o}) \; : \; 
		P \in M, R \in SO(3), \o \in K \right\} 
		\\
		&= 
		\left\{ \exp_{P}^{g_{\e}} ( R \T_{\o} ) \; : \; 
		P \in M, R \in SO(3), \o \in K \right\}.
		\end{aligned}
		\end{equation}
	\end{enumerate}

\begin{rem}\label{r:param}
M\"obius-inverted Clifford tori in $\R^3$  can be clearly completely described by their 
translation vector, their rotation and their distortion factor. It then follows that
for exponential maps in the metric $g_\e$ (see \eqref{eq:def-ge}), except for $\o = 0$, 
the parameters $P, R, \o$ give a local smooth parametrization of ${\mathcal T}_{\e, K}$. 

As for Remark \ref{r:nondeg}, ${\mathcal T}_{\e, K}$ turns out to be globally a smooth 
manifold even near surfaces described by $\o = 0$, where one could use a different  
local parametrization. 
\end{rem}

\

		\noindent
		In what follows, we always consider our problem on $(M,g_{\e})$. 
For example, let $\mathcal{S}_{\e} \subset M$ be a family of 
smooth surfaces with $|\Sigma|_{g} = C_{0} \e^{2}$ 
for each $\Sigma \in \mathcal{S}_{\e}$. 
Using $g_{\e}$, we can work on the rescaled surfaces 
with the fixed area $C_{0}$ in normal coordinate of $g_{\e}$.

\

\noindent
		Next we shall show that 
each element of ${\mathcal{T}}_{\e ,K}$ 
is an approximate solution of $W_{g_\e}$.

\begin{lem}\label{l:appsol}
	Fix a compact set $K \subset \DD$ and $k \in \N$. 
	Then there exists $C_{K,k}>0$ such that 
$$
  \| W_{g_{\e}}'(\Sigma) \|_{C^{k}(\Sigma)} \leq C_{K,k} \e^2 
  \qquad \quad 
  \hbox{ for every } \Sigma \in {\mathcal{T}}_{\e,K} \quad 
  {\rm and} \quad \e \in (0,\e_{0}). 
$$ 
\end{lem}

\begin{proof}
By Proposition \ref{p:1-2-var}, we have 
	\be\label{eq:W'}
		W_{g_{\e}}'(\Sigma)= -\Delta_{g_{\e}} H_{g_{\e}} - 
		|A_{g_{\e}}|^2 H_{g_{\e}} 
		- H_{g_{\e}} \Ric_{g_{\e}} (n_{g_{\e}},n_{g_{\e}}) 
		+ \frac{1}{2} H_{g_\e}^3
	\ee
in the sense that $\d W_{g_{\e}}(\Sigma)[\varphi]=
\int_\Sigma W'_{g_{\e}} (\Sigma)\, \varphi \; d \s_{g_{\e}}$. 
To prove the lemma we apply this formula to a surface 
$\Sigma \in \tilde{\mathcal{T}}_{\e, K}$ endowed with the metric 
induced by $g_\e$ defined in \eqref{eq:def-ge} 
(or \eqref{eq:ge=d+eh}).

		First, by the properties of $\rho_0$ and $\e_0$, and 
the definition of $g_\e$, we may find $\rho_{K}>0$ such that 
	\[
		| x |_{g_{0}} \leq \rho_{K} \leq \frac{\rho_{0}}{\e} \qquad 
		{\rm for\ each}\ x \in R \T_{\o} \subset T_PM \quad 
		{\rm and} \quad 
		(\e,P,R,\o) \in (0,\e_{0}] \times M \times SO(3) \times K.
	\]
Moreover, from \eqref{eq:ge=d+eh} and \eqref{eq:esti-he}, 
for each $k \in \N$, there exists $C_{K,k}>0$ such that 
	\begin{equation}\label{eq:diff-ge-g0}
		\left\| g_{\e,P,\a \b} - \delta_{\a \b} 
		\right\|_{C^{k}(\overline{B_{\rho_{K}}(0)})} 
		\leq C_{K,k} \e^{2} 
		\quad {\rm for\ each} \ 
		(\e,P,R,\o) \in (0,\e_{0}] \times M \times SO(3) \times K.
	\end{equation}
By \eqref{eq:diff-ge-g0}, we may compare the geometric quantities with 
the Euclidean ones and obtain the following estimates: 
	\begin{equation}\label{eq:diff-geo-qu}
		\begin{aligned}
			&
			\left\| \Gamma_{\e ,P, \mu \nu}^{\xi} 
			\right\|_{C^{k}( \overline{B_{\rho_{K}}(0)} )} 
			\leq C_{K,k} \e^{2}, 
			\quad 
			\left\| \Ric_{\e ,P}
			\right\|_{C^{k}( \overline{B_{\rho_{K}}(0)} )} 
			\leq C_{K,k} \e^{2},
			\\
			&
			\left\| A_{\e ,P,R,\o} - A_{0,R,\o} 
			\right\|_{C^{k}(R \T_{\o} )} 
			\leq C_{K,k} \e^{2},
			\quad 
			\left\| \left( 
			\Delta_{\e ,P,R,\o} - \Delta_{0,R,\o} \right) 
			\varphi \right\|_{C^{k}(R\T_{\o})} 
			\leq C_{K,k} \e^{2} \| \varphi \|_{C^{k+2}(R\T_{\o})}
		\end{aligned}
	\end{equation}
for each $(\e,P,R,\o) \in (0,\e_{0}] \times M \times SO(3) \times K$ 
and $\varphi \in C^{k+2}(R\T_{\o})$. 
Here $\Gamma^{\xi}_{\e ,P,\mu \nu}$, 
$\Ric_{\e ,P}$, 
$A_{\e ,P,R,\o}$, 
$A_{0,R,\o}$, $\Delta_{\e ,P,R,\o}$ and 
$\Delta_{0,R,\o}$ 
stand for the Christoffel symbol and the Ricci tensor of $(B_{\rho_{K}}(0),g_{\e,P})$, 
and the second fundamental form and the Laplace--Beltrami operator of 
$R \T_{\o}$ with $g_{\e,P}$ and $g_{0}$, respectively.

Now, from \eqref{eq:W'}, \eqref{eq:diff-ge-g0} and \eqref{eq:diff-geo-qu}, 
it follows that 
	\[
		\left\| W_{g_{\e}}'(\Sigma) - W_{g_{0}}'( R \T_{\o} ) 
		\right\|_{C^{k}( R \T_{\o} )} \leq C_{K,k} \e^{2}.
	\]
Recalling that the Clifford torus 
(and all its images under M\"obius transformation, by conformal invariance of $W$ in $\R^3$) 
is a critical point of the Willmore functional, 
namely $W_{g_{0}}'( R \T_{\o} ) = 0$, Lemma \ref{l:appsol} follows. 
\end{proof}

\subsection{Reduction to a finite-dimensional problem}

Let $R \in SO(3)$ and $\o \in \DD$. 
As in Subsection \ref{SS:BasicW}, 
we pull back a neighbourhood of 
$R\T_\o $ with the metric $g_{\e ,P}$ 
onto that of $\T$. 
Namely, we define 
$g_{\e ,P,R,\o}$ by the pull back of $g_{\e ,P}$ 
via the map $R \circ T_\o $: $ g_{\e ,P,R,\o} := (R \circ T_\o)^\ast g_{\e ,P} 
= T_\o^\ast \circ R^\ast \circ (\exp_P^{g_\e})^\ast g_\e $. 
Remark that $(\T,g_{\e ,P,R,\o})$ is isometric to 
$(\exp^{g_\e}_P (R \T_\o),g_\e  )$ or 
$(R \T_\o , g_{\e ,P})$. 
Let $n_{\e ,P,R,\o}$ be the unit outer normal of $(\T,g_{\e ,P,R,\o})$.

As before, we consider perturbations of $(\T,g_{\e ,P,R,\o})$ by regular 
functions $\varphi : \T \to \R$: 
	\begin{equation}\label{eq:defSw}
		\begin{aligned}
			\left(\T[\varphi]\right)_{\e ,P,R,\o}
			&:= \left\{ p + \varphi(p) n_{\e ,P,R,\o}(p) \; : \; p \in \T  \right\}, 
			\\
			\left( R \T_\o [\varphi] \right)_{\e ,P} 
			&:= \left\{ R T_\o (p + \varphi(p) n_{\e ,P,R,\o}(p) ) \; : \; p \in \T  \right\},
			\quad 
			\Sigma_{\e ,P,R,\o}[\varphi] := \exp^{g_\e}_{P} 
			\left(  (R \T_\o [\varphi] )_{\e ,P} \right). 
		\end{aligned}
	\end{equation}
Given a positive constant $\ov{C}$, we next define the family of functions
\[
   \mathcal{M}_{\e,P,R,\o} = \left\{  
   \var\in C^{4,\a}(\T) \; : \; 
   \|\var \|_{C^{4,\a}(\T)} \leq \ov{C} \e^{2}   
   \text{ and such that} \left|\Sigma_{\e,P,R,\omega}[\var]\right|_{g_\e} 
   = 4 \sqrt{2} \pi^2
   \right\}.
\]
Since each element of $\mathcal{T}_{\e ,K}$ is described as 
$\Sigma_{\e ,P,R,\o}[0]$, we notice that, in metric $g_\e$, 
the Hausdorff distance of elements 
$\Sigma_{\e,P,R,\omega}[\var]$ 
($\var \in \mathcal{M}_{\e,P,R, \omega}$ and $\o \in K$) 
from the previous approximate solutions ${\mathcal{T}}_{\e, K} $ 
is of order $\e^2$, 
with a constant depending on $\ov{C}$ and $K$. 

\

\noindent We next define $\mathcal{M}_{\e}$ to be the Banach manifold of surfaces of the type 
$\Sigma_{\e,P,R,\omega}[\var]$ with $P$ varying in $M$, 
$\var \in \mathcal{M}_{\e,P,R,\o}$, $R$ in $SO(3)$ and $\o $ in $\DD$. 
Notice that, by construction, elements in 
$\mathcal{M}_{\e}$ all satisfy the desired area constraint. 

\

\noindent The surfaces in $\mathcal{T}_{\e, K}$   form a seven-dimensional 
sub-manifold in $\mathcal{M}_\e$: we will show that the Willmore functional is 
indeed  non-degenerate in directions orthogonal to $\mathcal{T}_{\e, K}$. 

\

		\noindent
		To observe the non-degeneracy of elements in $\mathcal{T}_{\e,K}$, 
we introduce the eight-dimensional vector space
	\[
		\mathcal{K}^{\varphi}_{\e,P,R,\o} := 
		{\rm span}\, 
		\{ H_{\e,P,R,\o}[\varphi], Z_{1,R,\o}, \ldots, Z_{7,R,\o} \}
	\]
where $\varphi \in C^{4,\a}(\T)$ and $H_{\e,P,R,\o}[\varphi]$ 
denotes the mean curvature of $\Sigma_{\e,P,R,\o}[\varphi]$ 
in the metric $g_{\e}$. 
Since we are only interested in $\varphi$ whose 
$C^{4,\alpha}$-norm is small, 
$\Sigma_{\e,P,R,\o}[\varphi]$ can be written as the normal graph on 
$\T$ and this correspondence is diffeomorphic. 
Therefore, we pull back geometric quantities of $\Sigma_{\e ,P,R,\o}[\varphi]$ onto $\T$. 
In particular, we write $\bar{g}_{\e ,P,R,\o}[\varphi]$ for pull back of 
the tangential metric of $\Sigma_{\e ,P,R,\o} [\varphi]$ on $\T$. 
Denote by $\la \cdot , \cdot \ra_{\e ,P,R,\o ,\varphi}$ 
and $L^2_{\e ,P,R,\o,\varphi}(\T)$ the $L^2$-inner product of $L^2(\T)$ 
with the metric $\bar{g}_{\e ,P,R,\o}[\varphi]$ and 
$(L^2(\T), \la \cdot , \cdot \ra_{\e ,P,R,\o ,\varphi} )$. 
Remark that 
when $\e =0$ and $\varphi = 0$, these symbols coincide 
with those introduced in Subsection \ref{SS:BasicW} and 
do not depend on $P \in M$, i.e., $g_{0,P,R,\o} = g_{0,R,\o}$ and so on.  
See the comments above Proposition \ref{p:nondeg}.

\

		\noindent
		Next, as in the proof of Proposition \ref{p:nondeg}, 
we normalize and orthogonalize $H_{\e,P,R,\o}[\varphi]$ and 
$\{Z_{i,R,\o}\}_{i=1}^{7}$ in $L^2_{\e ,P,R,\o,\varphi}(\T)$. 
Namely, we first normalize and orthogonalize  
$\{Z_{i,R,\o}\}_{i=1}^{7}$ to get $\{Y_{i,\e,P,R,\o}[\varphi]\}_{i=1}^{7}$. 
Then we obtain $Y_{0,\e,P,R,\o}[\varphi]$ from $H_{\e,P,R,\o}[\varphi]$ 
and $\{ Y_{i,\e,P,R,\o}[\varphi] \}_{i=1}^{7}$, and we may assume that 
	\[
		\left\{Y_{0,\e,P,R,\o}[\varphi], Y_{1,\e,P,R,\o}[\varphi], \ldots, 
		Y_{7,\e,P,R,\o}[\varphi] \right\}
	\]
is the $L^2_{\e ,P,R,\o,\varphi}(\T)$-orthonormal basis 
of $\mathcal{K}^{\varphi}_{\e,P,R,\o}$. 
We also define the 
$L^2_{\e ,P,R,\o,\varphi}(\T)$-projection to the space 
$(\mathcal{K}^{\varphi}_{\e,P,R,\o})^{\perp \bar{g}_{\e,P,R,\o}[\varphi]}$ by 
	\[
		\begin{aligned}
			&\Pi^{\varphi}_{\e,P,R,\o} : 
			L^2_{\e ,P,R,\o,\varphi}(\T) \to 
			(\mathcal{K}^{\varphi}_{\e,P,R,\o})^{\perp \bar{g}_{\e,P,R,\o}[\varphi]}, 
			\\ 
			&\Pi^{\varphi}_{\e,P,R,\o} \psi 
			:= \psi - \sum_{i=0}^{7} \la \psi, Y_{i,\e,P,R,\o}[\varphi] 
			\ra_{L^2_{\e ,P,R,\o,\varphi}} Y_{i,\e,P,R,\o}[\varphi].
		\end{aligned}
	\]
Finally, we also define 
$Y_{i,R,\o}[\varphi]$ and $\Pi^{\varphi}_{0,R,\o}$ 
for the Euclidean metric in a similar way to above. 
Remark that $Y_{i,0,P,R,\o}[\varphi] = Y_{i,R,\o}[\varphi]$, 
$\Pi^\varphi_{0,P,R,\o} = \Pi^\varphi_{0,R,\o}$ and 
$\Pi^0_{0,R,\o} = \Pi_{0,R,\o}$ hold.

Regarding the properties of $\{Y_{i,\e,P,R,\o}[\varphi]\}_{i=0}^{7}$ and 
$\Pi^{\varphi}_{\e,P,R,\o}$, we have

		\begin{lem}\label{l:prop-Y}
			Fix a compact set $K \subset \DD$, $\ell \in \N$ and $\alpha \in (0,1)$.
			Then there exist $r_{K,\ell,1}>0$, $\e _{K,\ell,1} > 0$ and $C_{K,\ell}$ such that  
			
			\item[\rm (a)] 
			For each $i = 0 , \ldots, 7$ and $\e \in (0,\e _{K,\ell,1} ]$, the map 
				\[
					\overline{B_{r_{K,\ell,1},C^{4+\ell,\a}}(0)} \ni \varphi 
					\mapsto Y_{i,\e,P,R,\o}[\varphi] 
					\in C^{2+\ell,\a}(\T)
				\]
			is smooth and $D^\ell_\varphi (Y_{i,\e ,P,R,\o} [\varphi] )$ 
			is a bounded map in $\overline{B_{r_{K,1},C^{4+\ell,\a}}(0)}$ 
			and the bounds only depend on $r_{K,\ell,1}$, $K$ and $\e _{K,\ell,1} $. 
			Here $B_{r_{K,\ell,1},C^{k,\a}}(0)$ denotes the ball 
			in $C^{k,\a}(\T)$ centred at the origin with the radius $r_{K,\ell,1}$ 
			and $D_\varphi$ does the derivative with respect to $\varphi$. 
			In particular, the map 
				\[
					\overline{B_{r_{K,1}, C^{4+\ell,\a}}(0)} 
					\ni \varphi \mapsto \Pi^{\varphi}_{\e,P,R,\o} 
					\in \mathcal{L}(C^{2+\ell,\a}( \T ), C^{2+\ell,\a}( \T ))
				\]
			is smooth and $D^\ell_\varphi \Pi^\varphi_{\e ,P,R,\o}$ is bounded 
			and the bounds also depend only on $r_{K,1}$, $K$ and $\e _{K,1} $. 
		
		\item[\rm (b)] 
			For every $(\e,P,R,\o) \in (0,\e _{K,1} ] \times M \times SO(3) \times K$ 
			and $\| \varphi \|_{C^{4+\ell,\a}(\T)} \leq r_{K,1}$, 
			there holds 
				\[
					\begin{aligned}
						&\| Y_{i,\e,P,R,\o}[\varphi] - Y_{i,R,\o}[\varphi] 
						\|_{C^{2+\ell,\a}( \T )} 
						+ \| D_\varphi (Y_{i,\e,P,R,\o}[\varphi]) 
						- D_\varphi (Y_{i,R,\o}[\varphi]) 
						\|_{\mathcal{L}(C^{4+\ell,\a}(\T), C^{2+\ell,\a}(\T)) }
						\\
						& \qquad 
						+ \| D_\varphi^2 (Y_{i,\e,P,R,\o}[\varphi]) 
						- D_\varphi^2 (Y_{i,R,\o}[\varphi]) 
						\|_{\mathcal{L}^2(C^{4+\ell,\a}(\T), 
							 C^{2+\ell,\a}(\T)) }
						\leq C_{K,\ell} \e^{2}
					\end{aligned}
				\]
			where $\mathcal{L}^2(C^{4+\ell,\a}(\T), C^{2+\ell,\a}(\T)  )$ is a Banach space 
			of all multilinear maps from $C^{4+\ell,\a}(\T) \times C^{4+\ell,\a}(\T)$ 
			to $C^{2+\ell,\a}(\T)$. 
			
			\item[\rm (c)] 
				For every fixed $\varphi \in \overline{B_{r_{K,1},C^{4+\ell,\a}}(0)} $, the map 
					\[
						(P,R,\o) \mapsto Y_{i,\e,P,R,\o}[\varphi] : \quad 
						M \times SO(3) \times K 
						\to C^{2+\ell,\a}(\T) 
					\]
				is smooth. Moreover, 
					\[
						\sum_{k=0}^2 
						\left\| D_{P,R,\o}^k 
						\left(  Y_{i,\e ,P,R,\o}[\varphi] - Y_{i,R,\o}[\varphi] \right) 
						\right\|_{C^{2+\ell,\a}(\T)} \leq 
						C_{K,\ell} \left( \e^2 + \| \varphi \|_{C^{4+\ell,\a}(\T)} \right). 
					\]
		\end{lem}

		\begin{proof}
We first remark that 
	\[
		\begin{aligned}
			g_{\e ,P,R,\o ,\alpha \beta} (y) 
			&= g_{\e ,P} ( R T_\o (y) ) 
			[ D(R T_\o)(y) [\mathbf{e}_\alpha] , D(R T_\o)(y) [\mathbf{e}_\beta]  ], 
			\\
			g_{0,R,\o} (y) 
			&= g_0 [ D(R T_\o)(y) [\mathbf{e}_\alpha] , D(R T_\o)(y) [\mathbf{e}_\beta]  ],
		\end{aligned}
	\]
where $\mathbf{e}_\alpha$ denotes the canonical basis in $\R^3$. 
Recalling \eqref{eq:ge=d+eh} and \eqref{eq:met-deri}, 
it is easily seen that for any $\ell \in \N$, we have 
	\begin{equation}\label{eq:diff-metric}
		\sum_{k=0}^2 
		\left\| D_{P,R,\o}^k 
		\left(  g_{\e ,P,R,\o ,\a \b} - g_{0,P,R,\o,\a \b} \right)
		\right\|_{C^{\ell} (  \overline{B_{\rho_K}(0)}  ) } 
		\leq C_{K,\ell} \e^2
	\end{equation}
where $\rho_K$ appears in the proof of Lemma \ref{l:appsol}. 
This yields 
	\begin{equation}\label{eq:diff-normal}
		\sum_{k=0}^2 
		\left\| D_{P,R,\o}^k 
		\left(  n_{\e ,P,R,\o} - n_{0,R,\o} \right)
		\right\|_{C^{\ell} (  \T  ) } 
		\leq C_{K,\ell} \e^2. 
	\end{equation}

		We define position vectors for $\Sigma_{\e ,P,R,\o}[\varphi]$ and 
$R \T_\o [\varphi]$ in $(\R^3 , g_{\e ,P,R,\o})$ and $(\R^3,g_{0,R,\o})$ by
	\[
		\mathcal{Z}_{\e ,P,R,\o} [\varphi] (p) 
		:= p + \varphi(p) n_{\e ,P,R,\o}(p) , \qquad 
		\mathcal{Z}_{0,R,\o} [\varphi] (p) 
		:=  p + \varphi (p) n_{0,R,\o}(p). 
	\]
Clearly, the map 
	\[
		M \times SO(3) \times K \times C^{4+\ell,\a}(\T) \ni (P,R,\o ,\varphi) 
		\mapsto (\mathcal{Z}_{\e ,P,R,\o} [\varphi], \mathcal{Z}_{0,R,\o} [\varphi] ) 
		\in \left( C^{4+\ell,\a}(\T,\R^3) \right)^2
	\]
is smooth. Hence, we also see the smoothness of maps 
	\[
		\begin{aligned}
			&(P,R,\o ,\varphi) \mapsto g_{\e ,P,R,\o,\a \b}( \mathcal{Z}_{\e ,P,R,\o}[\varphi] )
			\ : \ M \times SO(3) \times K \times 
			\overline{B_{r_{K,\ell,1}, C^{4+\ell,\a }}(0)} \to C^{4+\ell,\a}(\T,\R),
			\\
			&(P,R,\o ,\varphi) \mapsto \bar{g}_{\e ,P,R,\o}[\varphi] 
			\ : \ M \times SO(3) \times K \times \overline{B_{r_{K,\ell,1}, C^{4+\ell,\a }}(0)} 
			\to C^{3+\ell,\a}  (\T,  (T\T)^\ast \otimes (T\T)^\ast ),
			\\
			&(P,R,\o ,\varphi) \mapsto n_{\e ,P,R,\o}[\varphi] 
			\ : \ M \times SO(3) \times K \times \overline{B_{r_{K,\ell,1}, C^{4+\ell,\a }}(0)} 
			\to C^{3+\ell,\a} (\T, \R^3),
			\\
			&(P,R,\omega,\varphi) \mapsto H_{\e ,P,R,\o} [\varphi] \ :\ 
			M \times SO(3) \times K \times 
			\overline{B_{r_{K,\ell,1},C^{4+\ell,\a}}(0)} \to C^{2+\ell,\a}(\T)
		\end{aligned}
	\]
where $n_{\e ,P,R,\o}[\varphi]$ denotes the unit outer normal 
of $(\T [\varphi])_{\e ,P,R,\o} $ in 
the metric $g_{\e ,P,R,\o}$. 
From this fact, it is easily seen that the assertion (a) holds.

		Moreover, by \eqref{eq:diff-normal}, there holds 
	\begin{equation}\label{eq:esti-Ze-Z0}
		\begin{aligned}
			&\sum_{k=0}^2 
			\left\| D_{P,R,\o}^k \left( \mathcal{Z}_{\e ,P,R,\o} [\varphi]  
			- \mathcal{Z}_{0,R,\o} [\varphi]
			\right) \right\|_{C^{4+\ell,\a} (\T,\R^3) } 
			\leq C_{K,\ell} \| \varphi \|_{C^{4+\ell,\a}(\T)} \e^2,
			\\
			& \left\| D_{\varphi} \left( \mathcal{Z}_{\e ,P,R,\o} [\varphi] 
			- \mathcal{Z}_{0,R,\o} [\varphi]  \right) 
			\right\|_{ \mathcal{L} ( C^{4+\ell,\a }(\T), C^{4+\ell,\a }(\T,\R^3 ) )}
			\\
			& \qquad + 
			\left\| D_{\varphi}^2 \left( \mathcal{Z}_{\e ,P,R,\o} [\varphi] 
			- \mathcal{Z}_{0,R,\o} [\varphi]  \right) 
			\right\|_{ \mathcal{L}^2 ( C^{4+\ell,\a }(\T), 
			C^{4+\ell,\a }(\T,\R^3 ) )} 
			\leq C_{K,\ell} \e^2. 
		\end{aligned}
	\end{equation}
Noting that similar estimates hold 
for $g_{\e ,P,R,\o ,\a \b} ( \mathcal{Z}_{\e ,P,R,\o}[\varphi] )$, 
$\bar{g}_{\e ,P,R,\o}[\varphi]$, $n_{\e ,P,R,\o}[\varphi]$ and 
$H_{\e ,P,R,\o}[\varphi]$ thanks to \eqref{eq:diff-metric}, 
we observe that assertions (b) and (c) also hold. 
		\end{proof}

\noindent We next proceed at a non-linear level, 
exploiting Proposition \ref{p:nondeg} and using the inverse mapping theorem.

\begin{pro}\label{p:lyap}
Fix a compact subset $K$ of  $\DD$ as above. 
Then there exist positive constants 
$\bar{C}_K$ and $0<\e_{K,2} \leq \e_{K,1,1}$ such that  
for any $\e \in (0,\e_{K,2}]$ and 
$(P,R,\o) \in M \times SO(3) \times K$, 
there exists a  function 
$\var_{\e} = \var_{\e}(P,R,\omega) \in C^{5,\a}(\T)$ 
satisfying the following:
	$$
		\begin{aligned}
			& j) \quad  W_{g_{\e}}'(\Sigma_{\e,P,R,\omega} 
			[\var_{\e}]) 
			= \b_0 H_{\e,P,R,\o}[\var_\e] 
			+ \sum_{i=1}^7 \beta_i Z_{i,R,\o}; 
			& &
			jj) \quad |\Sigma_{\e,P,R,\omega}  [\var_{\e}]|_{g_{\e}} 
			= 4 \sqrt{2} \pi^2;\\
			&
			jjj) \quad \la Y_{j,\e ,P,R,\o}[\varphi_\e] , \varphi_\e 
			\ra_{L^2_{\e ,P,R,\o ,\varphi_\e}} = 0 
			\quad (1 \leq j \leq 7); 
			& & 
			jjjj) \quad \| \varphi_{\e}(P,R,\o) \|_{C^{5,\a}(\T)} 
			\leq \bar{C}_{K} \e^{2}
	 \end{aligned}
$$
for some numbers $\b_0, \dots, \b_7$. 
In particular, $\varphi_{\e}(P,R,\o) \in \mathcal{M}_{\e ,P,R,\o }$. 
Moreover, for each $\e\in (0,\e_{K,2}]$, the map 
	\[
		M \times SO(3) \times K \to C^{5,\a}(\T) : \quad 
		(P,R,\o) \mapsto \varphi_{\e}(P,R,\o)
	\]
is smooth and satisfies 
	\[
		\| D_{P,R,\o} \varphi_{\e}(P,R,\o) \|_{C^{5,\a}(\T)} 
		+ \| D^2_{P,R,\o} \varphi_\e (P,R,\o) \|_{C^{5,\a}(\T)} 
		\leq C_{K} \e^2. 
	\]
\end{pro}

\begin{proof}
We first define $G_\e (P,R,\o ,\varphi) : M \times SO(3) \times K \times 
C^{5,\a}(\T) \to C^{1,\a}(\T)$ by 
	\[
		\begin{aligned}
			G_{\e}(P,R,\o ,\varphi) := &
			\Pi^\varphi_{\e ,P,R,\o} \left( W_{g_\e}'(\Sigma_{\e ,P,R,\o} [\varphi] )\right) 
			+ \left( | \Sigma_{\e ,P,R,\o} [\varphi] |_{g_\e} - 4 \sqrt{2} \pi^2 \right) 
			H_{\e ,P,R,\o}[\varphi] 
			\\
			&+ \sum_{j=1}^7 
			\la Y_{j,\e ,P,R,\o}[\varphi], \varphi \ra_{L^2_{\e ,P,R,\o ,\varphi}} 
			Y_{j,\e ,P,R,\o}[\varphi].
		\end{aligned}
	\]
Since $\Pi^\varphi_{\e ,P,R,\o}$ is $L^2_{\e ,P,R,\o}(\T)$-projection 
into $(\mathcal{K}_{\e ,P,R,\varphi}^\varphi)^{\perp \bar{g}_{\e ,P,R,\o}[\varphi]}$ 
and $\mathcal{K}_{\e ,P,R,\varphi}^\varphi$ is spanned by 
$H_{\e ,P,R,\o}[\varphi]$ and $Y_{j,\e ,P,R,\o}[\varphi]$ ($1 \leq j \leq 7$), 
in order to prove $j)$, $jj)$ and $jjj)$, it is enough to find 
$\varphi \in C^{5,\a}(\T)$ so that $G_{\e}(P,R,\o ,\varphi) = 0$.

		To this end, we first claim that 
there exist $r_{K,2} \in (0,r_{K,1,1}]$ and $\e_{K,2} \in (0,\e_{K,1,1}]$ such that 
$D_\varphi G_\e (P,R,\o ,\varphi)$ is invertible 
for every $(P,R,\o,\varphi) \in M \times SO(3) \times K \times C^{5,\a}(\T)$ 
provided $\| \varphi \|_{C^{5,\a}(\T)} \leq r_{K,2}$ and $0 < \e \leq \e_{K,2}$. 
We notice that $G_\e $ is smooth in each variable thanks to 
Lemma \ref{l:prop-Y}. 
Furthermore, by \eqref{eq:diff-metric}, \eqref{eq:diff-normal} and 
\eqref{eq:esti-Ze-Z0}, we observe that 
	\begin{equation}\label{eq:conv-Ge}
		\begin{aligned}
			&\sum_{k=0}^2 \left\| 
			D_{P,R,\o}^k \left( G_{\e} (P,R,\o,\varphi) - G_0(R,\o ,\varphi) \right)
			 \right\|_{C^{1,\a}(\T)} \leq C_K (\e^2 + \| \varphi \|_{C^{5,\a}(\T)} )
			 ,\\ 
			 &\left\| D_{\varphi} \left( G_{\e} (P,R,\o,\varphi) - G_0(R,\o ,\varphi) \right) 
			 \right\|_{ \mathcal{L}( C^{5,\a}(\T), C^{1,\a }(\T ) } 
			 \\
			 & \qquad 
			 + \left\| D_{\varphi}^2 \left( G_{\e} (P,R,\o,\varphi) - G_0(R,\o ,\varphi) \right) 
			 \right\|_{ \mathcal{L}^2( C^{5,\a}(\T), C^{1,\a }(\T ) }
			 \leq C_{K} \e^2
		\end{aligned}
	\end{equation}
for all $(P,R,\o ,\varphi )  \in M \times SO(3) \times K \times C^{5,\a}(\T) $ 
provided $\| \varphi \|_{C^{5,\a}(\T)} \leq r_{K,1,1}$ where 
$G_0(R,\o ,\varphi)$ is a corresponding map defined for the Euclidean metric:
	\[
		\begin{aligned}
			G_0(R,\o ,\varphi) := &
			\Pi^\varphi_{0,R,\o} \left( W_{g_0}'( R\T_\o [\varphi] )\right) 
			+ \left( | R \T_\o [\varphi] |_{g_0} - 4 \sqrt{2} \pi^2 \right) H_{0,R,\o}[\varphi]
			\\
			& + \sum_{j=1}^7 \la Y_{j,R,\o}[\varphi] , \varphi \ra_{L^2_{0,R,\o ,\varphi}} 
			Y_{j,R,\o}[\varphi].
		\end{aligned}
	\]
From \eqref{eq:conv-Ge}, it suffices to show that 
$D_\varphi G_0(R,\o ,0)$ is invertible.

		To see this, we remark that for $\psi \in C^{5,\a}(\T)$ and 
$|t| \ll 1$, the surface $R \T_\o [t \psi]$ is a perturbation in the normal direction. 
Hence, by $W_{g_0}'(R\T_0) = 0$ and Corollary \ref{c:1-2-flat}, we obtain 
	\[
		D_\varphi W_{g_0}'(R\T_\o [\varphi] ) \Big|_{\varphi =0} [\psi]  
		= \frac{d}{dt} W_{g_0}'(R\T_\o [t\psi] ) \Big|_{t=0}
		= \tilde{L}_{0,R,\o} \psi.
	\]
Moreover, by $|R \T_\o |_{g_0} = 4\sqrt{2} \pi^2$, one has 
	\[
		D_\varphi G_0(R,\o,0)[\psi] 
		=  \Pi^0_{0,R,\o} \tilde{L}_{0,R,\o} \psi 
		+\la H_{0,R,\o}[0], \psi \ra_{L^2_{0,R,\o}} H_{0,R,\o}[0] 
		+ \sum_{j=1}^7 \la Y_{j,R,\o}[0], \psi \ra_{L^2_{0,R,\o}}Y_{j,R,\o}[0]. 
	\]

		Now we show that $D_\varphi G_0(R,\o ,0)$ is injective. 
Let $D_\varphi G_0(R,\o ,0)[\psi] = 0$. Then we have 
	\[
		0 = \Pi^0_{0,R,\o} \tilde{L}_{0,R,\o} \psi = \la H_{0,R,\o}[0], \psi \ra_{L^2_{0,R,\o}}
		= \la Y_{j,R,\o}[0], \psi \ra_{L^2_{0,R,\o}} \quad 
			(1 \leq j \leq 7),
	\]
which implies $\psi \in \mathcal{K}_{0,R,\o}^{\perp \bar{g}_{0,R,\o}}$. 
Thus by Proposition \ref{p:nondeg} and 
$\Pi^0_{0,R,\o} \tilde{L}_{0,R,\o} \psi= 0$, one sees $\psi \equiv 0$ 
and we infer that $D_\varphi G_0(R,\o ,0)$ is injective.

		Next we prove that $D_\varphi G_0(R,\o ,0)$ is surjective. 
Recall that $Y_{j,R,\o}[0] \in {\rm Ker}\, \tilde{L}_{0,R,\o}$ ($ 1 \leq j \leq 7$) 
and  $\la H_{0,R,\o}[0], Y_{j,R,\o }[0] \ra_{L^2_{0,R,\o}} = 0$ 
($1 \leq j \leq 7$). 
Therefore, setting $\psi_j := Y_{j,R,\o}[0] \in C^\infty(\T)$ ($1 \leq j \leq 7$), 
we have $D_\varphi G_0(P,R,\o,0) [\psi_j] = Y_{j,R,\o}[0]$. 
Thus to prove that $D_\varphi G_0(P,R,\o ,0)$ is surjective, 
it is sufficient to show that the restricted map 
	\begin{equation}\label{eq:G0}
		D_\varphi G_0(P,R,\o ,0) : C^{5,\a}(\T) \cap 
		 \left(\mathcal{K}_{0,R,\o}^{\perp \bar{g}_{0,R,\o}} 
		\oplus \mathrm{span}\, \{ H_{0,R,\o} [0] \} \right)
		\to C^{1,\a}(\T) \cap 
		\left( \mathcal{K}_{0,R,\o}^{\perp \bar{g}_{0,R,\o}} 
		\oplus \mathrm{span}\, \{ H_{0,R,\o} [0] \}\right) 
	\end{equation}
is surjective. 
For $\psi \in \mathcal{K}_{0,R,\o}^{\perp \bar{g}_{0,R,\o}}$, 
it follows that $D_\varphi G_0(P,R,\o ,0)[\psi] 
= \Pi^0_{0,R,\o} \tilde{L}_{0,R,\o} \psi $. 
Hence, we have 
	\begin{equation}\label{eq:res-G0}
		D_\varphi G_0(P,R,\o ,0)\big|_{\mathcal{K}_{0,R,\o}^{\perp \bar{g}_{0,R,\o}}} 
		= \Pi^0_{0,R,\o} \tilde{L}_{0,R,\o} : 
		C^{5,\a}(\T) \cap 
		\left(\mathcal{K}_{0,R,\o}^{\perp \bar{g}_{0,R,\o}} \right) 
		\to C^{1,\a}(\T) \cap 
		\left(\mathcal{K}_{0,R,\o}^{\perp \bar{g}_{0,R,\o}} \right)
	\end{equation}
and shall show that the map in \eqref{eq:res-G0} is surjective. 
If this claim holds true, then the map in \eqref{eq:G0} is surjective. 
Indeed, since 
	\[
		\left(\mathrm{Id} - \Pi^0_{0,R,\o} \right) 
		\left(D_\varphi G_0(P,R,\o ,0) \left[ H_{0,R,\o}[0] \right]\right) 
		=  \| H_{0,R,\o}[0] \|_{L^2_{0,R,\o}}^2 H_{0,R,\o}[0] \neq 0,
	\]
combining with the surjectivity of the map in \eqref{eq:res-G0}, 
we may observe that the map in \eqref{eq:G0} is surjective.

		Now we show that the map in \eqref{eq:res-G0} is surjective. 
For this purpose, denote by $i_{0,R,\o}$ the inclusion map 
$C^{5,\a}(\T) \cap \left(\mathcal{K}_{0,R,\o}^{\perp \bar{g}_{0,R,\o}} 
\oplus \mathrm{span}\, \{ H_{0,R,\o} [0] \} \right) \subset C^{5,\a}(\T)$. 
Then the map $\Pi^0_{0,R,\o} \tilde{L}_{0,R,\o}$ is decomposed as 
follows: $\Pi^0_{0,R,\o} \tilde{L}_{0,R,\o} 
= \Pi^0_{0,R,\o} \circ \tilde{L}_{0,R,\o} \circ i_{0,R,\o}$. 
Notice that maps $i_{0,R,\o}$ and $\Pi^0_{0,R,\o}: C^{1,\a}(\T) 
\to \mathcal{K}_{0,R,\o}^{\perp \bar{g}_{0,R,\o}}$ 
are the Fredholm operators with indices $-8$ and $8$. 
Moreover, using the elliptic regularity and the Fredholm alternative, 
we also observe that 
$\tilde{L}_{0,R,\o} : C^{5,\a}(\T) \to C^{1,\a}(\T)$ is 
the Fredholm operator with index $0$. 
Hence, $\Pi^0_{0,R,\o} \tilde{L}_{0,R,\o} 
= \Pi^0_{0,R,\o} \circ \tilde{L}_{0,R,\o} \circ i_{0,R,\o}$ 
is the Fredholm operator with index $0$. 
Since we already know that $\Pi^0_{0,R,\o} \circ \tilde{L}_{0,R,\o}$ 
is injective, one may observe that 
$\Pi^0_{0,R,\o} \tilde{L}_{0,R,\o}$ is surjective and 
we conclude that $D_\varphi G_0(P,R,\o,0)$ is invertible.

	From \eqref{eq:conv-Ge}, for some $\e _{K,2}\in (0,\e_{K,1,1}]$, 
the map $D_\varphi G_\e(P,R,\o ,0)$ is invertible 
provided $0 \leq \e \leq \e_{K,2}$. 
Thus by the inverse mapping theorem, 
we may find neighbourhoods  $U_{1,\e ,P,R,\o} \subset C^{5,\a}(\T)$  and 
$U_{2,\e ,P,R,\o} \subset C^{1,\a}(\T)$ of $0$ and $G_\e (P,R,\o ,0)$ such that 
$G_\e (P,R,\o ,\cdot) : U_{1,\e ,P,R,\o} \to U_{2,\e ,P,R,\o}$ is diffeomorphism. 
Noting \eqref{eq:conv-Ge} and a proof of the inverse mapping theorem 
(for instance, see Lang \cite[Theorem 3.1 in Chapter XVIII]{LANG}), 
one may find $r_{K,2} \in (0, r_{K,1,1} ]$ so that 
	\[
		\overline{B_{r_{K,2}, C^{5,\a}(\T) }(0)} \subset U_{1,\e ,P,R,\o},
		\quad 
		\overline{B_{2r_{K,2}, C^{1,\a}(\T) }(G_\e (P,R,\o ,0))} \subset U_{2,\e ,P,R,\o}
	\]
for all $(\e ,P,R,\o) \in [0,\e_{K,2}] \times M \times SO(3) \times K$. 
Since $\| G_\e (P,R,\o ,0) \|_{C^{1,\a}} \leq C_K \e^2$ holds 
due to Lemma \ref{l:appsol}, shrink $\e _{K,2}>0$ if necessary, 
we have 
	\[
		\overline{B_{r_{K,2}, C^{1,\a}(\T) }(0)} 
		\subset \overline{B_{2r_{K,2}, C^{1,\a}(\T) }(G_\e (P,R,\o ,0))} 
		\quad \text{for all $\e \in [0,\e _{K,2}]$},
	\]
which means that one may find  a $\varphi_\e (P,R,\o)$ uniquely in 
$U_{1,\e ,P,R,\o}$ such that $G_\e (P,R,\o ,\varphi_\e (P,R,\o)) = 0$. 
Moreover, from $\| G_\e (P,R,\o ,0) \|_{C^{1,\a}} \leq C_K \e^2$ it follows that 
$\| \varphi_\e (P,R,\o) \|_{C^{5,\a}(\T)} \leq C_K \e^2$. 
Thus assertions $j)$, $jj)$, $jjj)$, $jjjj)$ hold.

		The smoothness of $\varphi_\e (P,R,\o)$ in $(P,R,\o)$ 
follows from the smoothness of $G_{\e}( P,R,\o,\varphi)$, 
$jjjj)$, the fact that $D_\varphi G_{\e} (P,R,\o ,\varphi_\e (P,R,\o))$ is 
invertible and the implicit function theorem. 
So we only need to prove the estimates for 
$D_{P,R,\o} \varphi_\e $ and $D_{P,R,\o}^2 \varphi_\e $. 
We differentiate $G_\e (P,R,\o , \varphi_\e (P,R,\o)) = 0$ in 
$(P,R,\o)$ to obtain 
	\begin{equation}\label{eq:first-deri}
			0 = (D_{P,R,\o} G_\e)(P,R,\o ,\varphi_\e ) 
			+ (D_\varphi G_\e) (P,R,\o ,\varphi_\e) 
			[ D_{P,R,\o} \varphi_\e ].
	\end{equation}
Noting $G_0(P,R,\o ,0) = 0$ for all $(P,R,\o) \in M\times SO(3) \times K$, 
it follows that 
	\[
		D_{P,R,\o} G_0(R,\o ,0) = 0 = D_{P,R,\o}^2 G_0(R,\o,0). 
	\]
Thus by \eqref{eq:conv-Ge}, $jjjj$), \eqref{eq:first-deri} and the facts 
$D_\varphi G_\e (P,R,\o ,\varphi_\e): C^{5,\a}(\T) \to C^{1,\a}(\T)$ is invertible 
and the norm of its inverse is uniformly bounded with respect to $\e ,P,R,\o $, 
we have 
	\[
		\left\| D_{P,R,\o} \varphi_\e \right\|_{C^{5,\a}(\T)} 
		\leq \left\| (D_\varphi G_\e (P,R,\o ,\varphi_\e))^{-1} 
		\right\|_{\mathcal{L} (C^{1,\a }(\T), C^{5,\a }(\T)  ) } 
		\left\| 
		 D_{P,R,\o} G_\e (P,R,\o ,\varphi_\e)  \right\|_{C^{1,\a}(\T)} 
		\leq C_K \e^2. 
	\]
We differentiate \eqref{eq:first-deri} and get 
	\[
		\begin{aligned}
			0&= (D_{P,R,\o}^2 G_\e) (P,R,\o ,\varphi_\e) 
			+ 2 (D_{P,R,\o} D_\varphi G_\e )(P,R,\o ,\varphi_\e) 
			[D_{P,R,\o} \varphi_\e] 
			\\
			& \quad + (D_\varphi^2 G_\e) (P,R,\o ,\varphi_\e ) 
			[ (D_{P,R,\o} \varphi_\e )^2 ] 
			+ (D_\varphi G_\e) (P,R,\o ,\varphi_\e) 
			[ D_{P,R,\o}^2 \varphi_\e  ]. 
		\end{aligned}
	\]
Using the fact $\| D_{P,R,\o} \varphi_\e \|_{C^{5,\a}(\T)} \leq C_K \e^2$, 
in a similar way to above, we also have 
$\| D_{P,R,\o}^2 \varphi_\e \|_{C^{5,\a}(\T)} \leq C_K \e^2$. 
Thus we complete the proof. 
\end{proof}

\
%

\

\noindent We can finally encode the variational structure of the problem by means of the following result.

\begin{pro}\label{p:variational}
Let $K\subset \subset \DD$, $\e_{K,2}$ and $\varphi_{\e}(P,R,\o)$ 
be as in Proposition \ref{p:lyap}, and for $\e \in (0, \e_{K,2}]$ 
define the function $\Phi_\e : \mathcal{T}_{K , \e} \to \R$ by
\begin{equation}\label{eq:defPhie}  
  \Phi_\e(P,R,\o) := W_{g_\e}(\Sigma_{\e,P,R,\omega} [\var_{\e}(P,R,\omega)]), 
\end{equation}
where $\Sigma_{\e,P,R,\omega} [\var_{\e}(P,R,\omega)]$ was defined in \eqref{eq:defSw}.
Then there exists  $\e_{K,3} \in (0,\e_{K,2}]$ such that 
if $\e \in (0,\e_{K,3}]$ and 
$(P_\e, R_\e , \o _\e ) \in M \times SO(3) \times K$ is a critical point of $\Phi_\e$, 
then  $\Sigma_{\e,P_\e,R_\e,\omega_\e} [\var_{\e}(P_\e,R_\e, \omega_\e)]$ 
satisfies the area-constrained Willmore equation. 
\end{pro}

\begin{proof}
First of all, a criticality of $\Phi_\e $ in $P$ is independent 
of choices of the scales of $M$, namely $g$ and $g_\e $. 
Hence, in this proof, we adapt the rescaled metric $g_\e $ 
for the differential in $P$.

		Let $P_0 \in M$ and $(U_0,\Phi_0)$ be a normal coordinate 
centred at $P_0$ of $(M,g_\e)$. For $P \in U_0$ with $z = \Phi_0(P)$, 
the position vector for $\Sigma_{\e ,P,R,\o}[\varphi]$ in $(U_0,\Phi_0)$ 
is expressed as 
	\[
		\widetilde{\mathcal Z}_{\e ,P,R,\o}[\varphi] (p) 
		= \mathcal{X}_\e \left( 1; z , T_\e (z) (R T_\o ( \mathcal{Z}_{\e ,P,R,\o}[\varphi] (p) ) 
		) \right)
	\]
where $T_\e (z) : \R^3 \to \R^3$ is a linear transformation 
with $T_\e(0) = \mathrm{Id}$ and $\mathcal{X}_\e (t;z,v)$ a solution of 
	\[
		\frac{d^2 \mathcal{X}_\e ^\a}{d t^2} 
		+ \Gamma^\a _{\e ,\beta \gamma} ( \mathcal{X}_\e ) 
		\frac{d \mathcal{X}_\e ^\b}{d t} 
		\frac{d \mathcal{X}_\e ^\gamma}{d t} = 0, \qquad 
		\left( \mathcal{X}_\e (0;z,v) , \ \frac{d \mathcal{X}_\e}{dt}(0;z,v) \right) 
		= (z, v) \in \R^6,
	\]
and $\Gamma^\alpha_{\e ,\beta \gamma} $ stands for the Christoffel symbol 
of $(M,g_\e)$ in $(U_0,\Phi_0)$. Noting \eqref{eq:ge=d+eh}, 
one may find that 
	\begin{equation}\label{eq:esti-Xe}
		\begin{aligned}
			&\| D_z^{k_1+1} T_\e (z) \|_{L^\infty} \leq C_{k_1} \e^2, \quad 
			\| \Gamma_{\e ,\beta \gamma}^\alpha \|_{C^{k_2}} 
			\leq C_{k_2} \e^2, \\
			&
			\mathcal{X}_\e (1;z,v) = z + v + \widetilde{R}_{\e ,1}(z,v), \quad 
			\| D_z^{k_1} D_v^{k_2} \widetilde{R}_{\e ,1}(z,v) \|_{L^\infty} \leq 
			C_{k_1,k_2} \e^2
		\end{aligned}
	\end{equation}
for all $k_1,k_2 \geq 0$. Hence, 
	\[
		\widetilde{\mathcal Z}_{\e ,P,R,\o}[\varphi] (p) 
		= z + R T_\o \left( \mathcal{Z}_{\e ,P,R,\o}[\varphi](p) \right) 
		+ \widetilde{R}_{\e ,2} \left(
		z, R T_\o \left( \mathcal{Z}_{\e ,P,R,\o}[\varphi](p) \right) 
		\right)
	\]
where $\widetilde{R}_{\e ,2}(z,v)$ satisfies the same estimate to 
$\widetilde{R}_{\e ,1}(z,v)$.

		Next, denote by 
$\widetilde{Z}_{0,R,\o} (p) := R T_\o (\mathcal{Z}_{0,R,\o}[0] (p) )$ and 
$\widetilde{n}_{0,R,\o} (p)$ the position vector and 
the unit outer normal of $R \T_\o $ in $(\R^3,g_0)$. 
We remark that the differential in $z$ corresponds 
to that in $P$ and 
	\[
		Z_{m,R,\o}(p) = g_0 \left[ D_{z^m} \left( z + \widetilde{Z}_{0,R,\o}(p) \right), 
		\widetilde{n}_{0,R,\o}(p)  \right], \quad 
		Z_{j,R,\o} (p) = g_0 
		\left[ D_{R,\o} \left( z + \widetilde{Z}_{0,R,\o}(p)  \right), 
		 \widetilde{n}_{0,R,\o}(p)  \right]
	\]
for $1 \leq m \leq 3$ and $ 4 \leq j \leq 7$. 
Recalling \eqref{eq:diff-normal}, \eqref{eq:esti-Ze-Z0}, \eqref{eq:esti-Xe} and 
Proposition \ref{p:lyap}, we have 
	\[
		\begin{aligned}
			&\widetilde{n}_{\e ,P,R,\o} [\varphi_\e] - \widetilde{n}_{0,R,\o} 
			= O_{C^{4,\a}}(\e^2), \quad 
			D_{z^m} \widetilde{\mathcal Z}_{\e ,P,R,\o}[\varphi_\e] (p) 
			= \mathbf{e}_m + O_{C^{4,\a}}(\e^2), 
			\\
			& D_{R,\o} 
			\left( \widetilde{\mathcal Z}_{\e ,P,R,\o}[\varphi_\e] (p) 
			- \widetilde{\mathcal Z}_{0,R,\o} (p)  \right) 
			= O_{C^{4,\a}} (\e^2)
		\end{aligned}
	\]
where $\widetilde{n}_{\e ,P,R,\o}[\varphi_\e]$ stands for 
the outer unit normal of $\Sigma_{\e ,P,R,\o}[\varphi_\e]$ 
in $(U_0,\Phi_0)$ and $\| O_{C^{4,\a}}(\e^2) \|_{C^{4,\a}} \leq C_K \e^2$. 
Thus putting 
	\[
		\begin{aligned}
			\psi_{\e ,P,R,\o ,m}(p) &:= 
			g_{\e,P_0} \left( \widetilde{\mathcal Z}_{\e ,P,R,\o}[\varphi_\e] \right)
			\left[ D_{z^m} \widetilde{\mathcal Z}_{\e ,P,R,\o}[\varphi_\e], 
			\widetilde{n}_{\e ,P,R,\o} [\varphi_\e] \right], 
			\\
			\psi_{\e ,P,R,\o ,j}(p) &:= 
			g_{\e,P_0} \left( \widetilde{\mathcal Z}_{\e ,P,R,\o}[\varphi_\e] \right) 
			\left[ D_{R,\o} \widetilde{\mathcal Z}_{\e ,P,R,\o}[\varphi_\e], 
			\widetilde{n}_{\e ,P,R,\o} [\varphi_\e] \right]
		\end{aligned}
	\]
for $1 \leq m \leq 3$ and $4 \leq j \leq 7$, we observe that 
	\[
		\| \psi_{\e ,P,R,\o ,i} - Z_{i,R,\o} \|_{C^{4,\a}(\T)} \leq C_K \e^2
	\]
for all $ 1 \leq i \leq 7$.

		Now we differentiate $jj$) in Proposition \ref{p:lyap} by $(z,R,\o)$ to get 
	\[
		0 = \left\la H_{\e ,P,R,\o}[\varphi_\e] , \psi_{\e ,P,R,\o,i} 
		\right\ra_{L^2_{\e ,P,R,\o,\varphi_\e}}.
	\]
Next we differentiate $\Phi_\e (P,R,\o)$. Using $j$) and $jjjj$), 
one sees that 
	\begin{equation}\label{eq:dPhie}
		\begin{aligned}
			D_{z,R,\o} (\Phi_\e (P,R,\o)) 
			&= \la W_{g_\e}'(\Sigma_{\e ,P,R,\o}) [\varphi_\e] , \psi_{\e ,P,R,\o ,i} 
			\ra_{L^2_{\e ,P,R,\o}} 
			= \sum_{k=1}^7 \beta_k \la Z_{k,R,\o}, \psi_{\e ,P,R,\o ,i} 
			\ra_{L^2_{\e ,P,R,\o ,\varphi_\e}}
			\\
			&= \sum_{k=1}^7 
			\beta_k \left( \la Z_{k,R,\o}, Z_{i,R,\o} \ra_{L^2_{0,R,\o}} 
			+ O_K(\e^2) \right) 
			=: \sum_{k=1}^7 \beta_k A_{\e ,P,R,\o,ik}
		\end{aligned}
	\end{equation}
Since $Z_{1,R,\o},\ldots,Z_{7,R,\o}$ are linearly independent, 
a matrix $A_{0,R,\o}$ defined by 
$A_{0,R,\o,ik} := \la Z_{k,R,\o}, Z_{i,R,\o} \ra_{L^2_{0,R,\o}}$ 
is invertible. Thus we may find $\e_{K,3} \in (0, \e_{K,2}]$ such that 
$A_{\e ,P,R,\o}$ is also invertible provided $0< \e \leq \e_{K,3}$.

		Let $(P_\e ,R_\e ,\o_\e)$ be a critical point of $\Phi_\e $. 
Then by \eqref{eq:dPhie}, one observes that 
	\[
		0 = \sum_{k=1}^7 A_{\e ,P_\e ,R_\e ,\o_\e,ik} 
		\beta_{k} \qquad (1 \leq i \leq 7).
	\]
Since $A_{\e ,P_\e ,R_\e ,\o_\e}$ is invertible, we obtain 
$0= \beta_1 = \cdots = \beta_7$. Hence, Proposition \ref{p:variational} holds. 
\end{proof}

\

\noindent Given Proposition \ref{p:variational}, 
it is useful to understand the dependence of $\Phi_\e$ upon $\Sigma$. 
The next result gives a useful Taylor expansion of this quantity in terms of $\e$.

\begin{pro}\label{p:variational2}
Let $K \subset \DD$ be a compact set. 
Then there exists a $C_K > 0$ such that 
	\[
		\left| \Phi_\e (P,R,\o) - W_{g_\e}(\Sigma_{\e , P, R, \o}[0] ) \right| 
		\leq C_K \e ^4
	\]
holds for every $(\e,P,R,\o) \in (0,\e_{K,3}] \times M \times SO(3) \times K$. 
\end{pro}

\begin{proof}
By definition of $\Phi_\e$ as in \eqref{eq:defPhie} and 
by the Taylor expansion of the Willmore functional up to second order in the perturbation $\var_\e$, 
we get 
	\[
		\begin{aligned}
		&\Phi_\e(P,R,\o)
		\\
		=  
		&W_{g_\e} (\Sigma_{\e,P,R,\omega}[0]) + 
		\la W_{g_\e}'(\Sigma_{\e ,P,R,\o}[0]) , \varphi_{\e} \ra_{L^2_{\e ,P,R,\o ,0}}
		 + O(\|\var_{\e}(P,R,\omega) \|^2_{C^{4,\a}(\T)}).
		\end{aligned}
	\]
Observe that by Proposition \ref{p:lyap} we have 
$\|\var_{\e}(P,R,\omega) \|^2_{C^{4,\a}(\T)}\leq C_K \e^4$. 
Moreover,  since $\Sigma_{\e,P,R,\omega}[0]$ is an approximate solution 
in the sense of  Lemma \ref{l:appsol}, we have
	\[
		\left| \la W_{g_\e}'(\Sigma_{\e ,P,R,\o}[0]) , \varphi_{\e} 
		\ra_{L^2_{\e ,P,R,\o ,0}} \right| \leq C_K \e^4.
	\]
Thus proposition follows by combining the above estimates. 
\end{proof}

\section{Energy expansions} \label{Sec:Expansions}

The non-compactness of the M\"obius group might create in general difficulties in 
extremising the function $\Phi_\e$ (which would amount to solving our problem by 
Proposition \ref{p:variational}).  In this section we will expand the Willmore energy on 
exponential maps of symmetric tori (in the notation of the previous section these 
are denoted as $\Sigma_{\e, P, R,0}[0]$), and of tori which are degenerating to spheres 
($\Sigma_{\e, P, R,\o}[0]$ with $|\o|$ close to $1$). A comparison between these two energies, 
which involves some local curvature quantities, will indeed help to 
rule-out a blow-up behaviour.

\subsection{Willmore energy of symmetric tori}

Here we calculate the expansion of the Willmore functional at  small area symmetric
Clifford tori embedded in $(M,g)$ through the exponential map for each $ P \in M$.  
Before stating the next Proposition, 
we recall that $\T \subset \R^3$ stands for the symmetric Clifford torus 
with the axial vector $\mathbf{e}_z=(0,0,1)$. See the beginning of Subsection \ref{ss:invfixarea}.

\begin{pro}\label{p:exp-W-to}
	As $\e \to 0$, the following expansion holds:  		
			\begin{equation}\label{eq:ex-W-to}
				\sup_{(P,R) \in M \times SO(3)} 
				\left| W_{g_\e} (\Sigma_{\e ,P,R,0}[0]) - 8\pi^2 - 4 \sqrt{2} \pi^2 
				\left\{ {\rm Sc}_P - \Ric_P(R \mathbf{e}_z, R \mathbf{e}_z) \right\} \e^2 \right|
				\leq C_0 \e^3.
			\end{equation}
\end{pro}
\

\noindent The proof of the above proposition relies on two lemmas below. First of all, we
linearise (with respect to $g_\e$) the Willmore energy of $\T$ at the Euclidean metric.

\begin{lem}\label{l:deri}
Let $\Sigma \subset \R^{3}$ be an embedded 
smooth closed surface and 
$(g_{t})_{\alpha\beta} := \delta_{\alpha \beta} 
+ t h_{\alpha \beta}$ a perturbation of 
the Euclidean metric $g_0$. With a slight abuse of notation we will denote by $g_t$ both the metric on $M$ 
and its restriction to $\Sigma$. Denote the mean curvature of $(\Sigma, g_{t})$ 
by $H(t)$ and set 
	\[
		W(t) := \int_{\Sigma} H(t)^{2} 
		d \sigma_{t}
	\]
where $ d \sigma_{t}$ is the volume element of 
$(\Sigma,g_{t})$. Then as $t \to 0$, 
the following expansion holds: 
$$
  W(t) = W(0) + t \dot{W}(0) + O(t^{2})
$$
where 
	\begin{equation}\label{eq:W-dot-0}
		\begin{aligned}
			\dot{W}(0) =& 
			\frac{1}{2} \int_{\Sigma} 
			\biggl[  2 \sum_{i=1}^{2} 
			\left\la \frac{\partial h}{\partial n} e_{i}, e_{i} \right\ra H(0) 
			- 4 \sum_{i=1}^{2} e_{i}(h_{ni}) H(0) 
			+ 4 \sum_{i,j=1}^{2} h_{nj} 
			\left\la \n_{e_{i}} e_{i}, e_{j} \right\ra H(0) 
			\\
			& \qquad \quad 
			- 2h_{nn} H^{2}(0) + H^{2}(0) 
			{\rm tr} \left( h_{|T_{p}\Sigma} \right) 
			\biggr] d\sigma_{0}(p).
		\end{aligned}
	\end{equation}
Here $\{e_{1}, e_{2}\}$ is an orthonormal basis of 
$T_{p}\Sigma$, $n$ the unit outer normal, 
$h_{ni} := \la h n, e_{i} \ra$, $h_{nn} := \la h \, n, n \ra$, 
and $h_{|T_{p}\Sigma}$ is the induced $(2,0)$-tensor on $\Sigma$.  
\end{lem}

\noindent
Before the proof of Lemma \ref{l:deri}, 
we introduce some notations. 
First, we denote by $\n^{g_{t}}$ the Levi-Civita connection of $(\R^{3},g_{t})$. 
Using the metric $g_{t}$, for every $p \in \Sigma$, 
we choose $e_{1,t}(p)$, $e_{2,t}(p)$ and $n_{t}(p)$ 
so that $\{e_{1,t}(p),e_{2,t}(p)\}$ is an orthonormal basis of $T_{p}\S$ 
with respect to the metric $g_{t}$ and 
$n_{t}(p)$ the outer normal unit to $\Sigma$ at $p$ in the metric $g_{t}$. 
Moreover, we may assume that 
$e_{1,t}(p), e_{2,t}(p)$ and $n_t(p)$ are smooth in $p$ and $t$, 
$e_{1,0} = e_1, e_{2,0} = e_2$ and $n_{0}=n$. 
In addition, we denote by $\Gamma^{\kappa}_{\l \mu}(t)$ 
the Christoffel symbols of $(\R^{3},g_{t})$.

\

\begin{proof}
Since it is easy to see that $W(t)$ is smooth with respect to $t$, 
it is enough to show \eqref{eq:W-dot-0}. 
We divide our arguments into several steps. 
First we recall the variation of the volume element (see \cite{Besse}, Proposition 1.186)
$d \sigma_{t}$.

\medskip

\noindent
{\bf Step 1:} {\sl There holds
	\[
		\frac{d}{d t} d \sigma_t \Big|_{t=0} 
		= \frac{1}{2} {\rm tr} 
		\left( h \big|_{T_p\Sigma} \right) d \sigma_0(p).
	\] 
}

%
%
%
%

\medskip

\noindent
Next we calculate the variation of $n_{t}$.

\medskip

\noindent
{\bf Step 2:} {\sl The following hold:
	\[
		(\dot{n}_0)_i := 
		g_0(\dot{n}_0,e_{i,0}) = - h_{ni}
		=: -h(n_0,e_{i,0}), \quad 
		(\dot{n}_0)_n := 
		g_0(\dot{n}_0, n_0) 
		= - \frac{1}{2} h_{nn} 
		=: - \frac{1}{2} h(n_0,n_0).
	\]
Since $\{e_{1,0},e_{2,0},n_{0}\}$ forms an orthonormal 
basis of $\R^{3}$, we have 
	\begin{equation}\label{eq:0.1}
		\dot{n}_0 = -h_{n1} e_{1,0} - h_{n2} e_{2,0} 
		- \frac{1}{2}h_{nn} n_0.
	\end{equation}}
By definition, there holds 
	\[
		1 = g_t(n_t,n_t) = g_0(n_t,n_t) 
		+ t h(n_t,n_t), \quad 
		0= g_t (n_t,e_{i,t}) = 
		g_0(n_t,e_{i,t}) + t h(n_t,e_{i,t}).
	\]
Differentiating these equalities with respect to $t$, we obtain 
	\begin{equation}\label{eq:0.2}
		0 = 2 g_0(\dot{n}_0,n_0) + 
		h(n_0,n_0), \quad 
		0 = g_0(\dot{n}_0,e_{i,0}) 
		+ g_0(n_0,\dot{e}_{i,0}) + h(n_0,e_{i,0}).
	\end{equation}
Noting that $e_{i,t} \in T_p\Sigma$ and that $T_{p}\Sigma$ 
is invariant with respect to $t$, 
one obtains $\dot{e}_{i,0} \in T_p\Sigma$. 
Recalling that $n_0 \perp T_p\Sigma$ 
with respect to $g_0$, one sees 
$g_0(n_0,\dot{e}_{i,0}) = 0$. 
Combining this fact with \eqref{eq:0.2}, 
Step 2 holds.

\medskip

\noindent
{\bf Step 3:}  {\sl One has 
	\[
		\dot{H}(0) = 
		-\sum_{i=1}^2 e_{i,0}(h_{ni}) 
		+
		\sum_{i,j=1}^2 h_{nj} 
		g_0( \n^{g_0}_{e_{i,0}} e_{i,0}, e_{j,0}) 
		- \frac{1}{2} h_{nn} H(0) 
		+ \frac{1}{2} \sum_{i=1}^2 
		g_0 \left( \frac{\partial h}{\partial n_0} 
		e_{i,0}, e_{i,0} \right). 
	\]
}

\medskip

We fix $p \in \S$ and calculate $\dot{H}(0)$ at $p$. 
Let $A_t$ be the second fundamental form of $\S$ in $(\R^3,g_t)$, i.e., 
	\[
		A_t(X,Y) := g_t( \nabla^{g_t}_X n_t, Y )
	\]
and set  
$A_{ij} := A_0(e_{i,0},e_{j,0})$. 
Recalling that $\{e_{1,t},e_{2,t}\}$ is an orthonormal basis of 
$T_p\Sigma$, we get 
	\[
		\begin{aligned}
			H(t) &= \sum_{i=1}^2 
			A_t(e_{i,t},e_{i,t}) 
			= 
			\sum_{i=1}^2 
			g_t( \n^{g_t}_{e_{i,t}} n_t, e_{i,t} )  
			= \sum_{i=1}^2
			\left\{ g_0(\n^{g_t}_{e_{i,t}} n_t, e_{i,t} ) 
			+ t h( \n^{g_t}_{e_{i,t}} n_t, e_{i,t}) 
			 \right\}
		\end{aligned}
	\]
and 
	\begin{equation}\label{eq:0.3}
		\dot{H}(0) 
		= \sum_{i=1}^2 
		\left\{ \frac{d }{d t}  
		g_0(\n^{g_t}_{e_{i,t}} n_t, e_{i,t} ) \Big|_{t=0} 
			+ h( \n^{g_0}_{e_{i,0}} n_0, e_{i,0}) 
			 \right\}.
	\end{equation}

		Next we compute the second term in \eqref{eq:0.3}. 
Since 
	\[
		\n^{g_0}_{e_{i,0}} n_0 
		= \sum_{j=1}^2 A_{ij} e_{j,0}, 
 	\]
we have 
	\begin{equation}\label{eq:0.4}
		\sum_{i=1}^2 h(\n^{g_0}_{e_{i,0}} n_0 , 
		e_{i,0} ) = 
		 \sum_{i,j = 1}^2 
		A_{ij} h(e_{j,0},e_{i,0}) = 
		 \sum_{i,j=1}^2 A_{ij} h_{ij}.
	\end{equation}
Let us turn next to the first term in \eqref{eq:0.3}: let $(x^1,x^2,x^3)$ be a global chart of $\R^3$ with 
$g_0(\partial /\partial x^{\alpha}, \partial / \partial x^{\beta} ) = \delta_{\a \b}$. 
Noting that 
	\[
		\frac{\partial g_{\alpha \l,0}}
		{\partial x^\mu} = 0 \quad 
		{\rm for\ all}\ \alpha,\l,\mu, \quad 
		\Gamma^{\kappa}_{\l \mu} (t) 
		= \frac{1}{2} g_t^{\kappa \alpha} 
		\left( \frac{\partial g_{\alpha \mu,t}}
		{\partial x^\l} 
		- \frac{\partial g_{\l \mu,t}}{\partial x^\alpha} 
		+ \frac{\partial g_{\alpha \l,t}}{\partial x^\mu}
		\right),
	\]
it follows that 
	\begin{equation}\label{eq:0.5}
		\begin{aligned}
			\frac{d}{d t} 
			\Gamma^{\kappa}_{\l \mu}(t)\Big|_{t=0} 
			&= 
			\frac{1}{2} \frac{d g^{\kappa\alpha}_t}{d t} 
			\Big|_{t=0} \left( \frac{\partial g_{\alpha \mu,0}}
		{\partial x^\l} 
		- \frac{\partial g_{\l \mu,0}}{\partial x^\alpha} 
		+ \frac{\partial g_{\alpha \l,0}}{\partial x^\mu}
		\right) 
		+ \frac{1}{2}g^{\kappa \alpha}_0 
		\left( \frac{\partial h_{\alpha \mu}}{\partial x^\l} 
		- \frac{\partial h_{\l \mu}}{\partial x^\alpha} 
		+ \frac{\partial h_{\alpha \l}}{\partial x^\mu} 
		\right)
		\\
		&= \frac{1}{2} \delta^{\kappa \alpha} 
		\left( \frac{\partial h_{\alpha \mu}}{\partial x^\l} 
		- \frac{\partial h_{\l \mu}}{\partial x^\alpha} 
		+ \frac{\partial h_{\alpha \l}}{\partial x^\mu} 
		\right).
		\end{aligned}
	\end{equation}
Thus, writing 
	\[
		e_{i,t} = e_i^\kappa(t) \frac{\partial}
		{\partial x^\kappa} , \quad  \qquad 
		n_t = n^\kappa(t) \frac{\partial}
		{\partial x^\kappa},
	\]
we get 
	\[
		\begin{aligned}
			\frac{d}{d t} \n^{g_t}_{e_{i,t}} n_{0,t} 
			\Big|_{t=0} 
			&= 
			\left( 
			\dot{e}^\l_i \frac{\partial n^\kappa}{\partial x^\l} 
			+ e^\l_i \frac{\partial \dot{n}^\kappa}{\partial x^\l} 
			+ \dot{\Gamma}^\kappa_{\l \mu} e^\l_i n^\mu 
			+ \Gamma^\kappa_{\l\mu} \dot{e}^\l_i n^\mu 
			+ \Gamma^\kappa_{\l\mu} e^\l_i \dot{n}^\mu
			\right) 
			\frac{\partial}{\partial x^\kappa}
		\\
			&= \n^{g_0}_{\dot{e}_{i,0}} n_0 
			+ \n^{g_0}_{e_{i,0}} \dot{n}_0 
			+ \dot{\Gamma}^\kappa_{\l\mu} e^\l_i n^\mu 
			\frac{\partial}{\partial x^\kappa}
		\end{aligned}
	\]
and 
	\begin{equation}\label{eq:0.6}
		\begin{aligned}
			\frac{d}{d t} g_0 ( \n^{g_t}_{e_{i,t}} n_t, e_{i,t} ) \Big|_{t=0} 
			&= g_0( \n^{g_0}_{\dot{e}_{i,0}} n_0, e_{i,0} ) 
			+ g_0 ( \n^{g_0}_{e_{i,0}} \dot{n}_0, e_{i,0} ) 
			+ \dot{\Gamma}^\kappa_{\l \mu} 
			e^\l_i n^\mu e^\alpha_i \delta_{\kappa\alpha}
			+ g_0\Big( \n^{g_0}_{e_{i,0}} n_0, \dot{e}_{i,0} \Big). 
		\end{aligned}
	\end{equation}
By \eqref{eq:0.5}, we have 
	\[
		\begin{aligned}
			\dot{\Gamma}^\kappa_{\l \mu} 
			e^\l_i n^\mu e^\alpha_i \delta_{\kappa\alpha}
			&=
			\frac{1}{2} \delta_{\kappa\alpha} 
			\delta^{\kappa\beta} 
			\left( \frac{\partial h_{\beta \mu}}{\partial x^\l} 
			- \frac{\partial h_{\l \mu}}{\partial x^\beta} 
			+ \frac{\partial h_{\beta \l}}{\partial x^\mu}
			\right) e^\l_i n^\mu e^\alpha_i
			=
			\frac{1}{2} 
			\left( \frac{\partial h_{\beta \mu}}{\partial x^\l} 
			- \frac{\partial h_{\l \mu}}{\partial x^\beta} 
			+ \frac{\partial h_{\beta \l}}{\partial x^\mu}
			\right) e^\l_i n^\mu e^\beta_i.
		\end{aligned}
	\]
Noting that, by a relabelling of the indices  
	\[
		\frac{\partial h_{\beta \mu}}{\partial x^{\l}} e^{\l}_{i} n^{\mu} e_{i}^{\beta} 
		= \frac{\partial h_{\l \mu}}{\partial x^{\beta}} 
		e^{\l}_{i} n^{\mu} e_{i}^{\beta},
	\]
we obtain 
	\begin{equation}\label{eq:0.7}
		\sum_{i=1}^2 \dot{\Gamma}^\kappa_{\l \mu} 
			e^\l_i n^\mu e^\alpha_i \delta_{\kappa\alpha} 
		= 
		\sum_{i=1}^2 \frac{1}{2} g_0 \left(  
		\frac{\partial h}{\partial n_0} e_{i,0}, e_{i,0} \right).
	\end{equation}
We consider now $g_0(\n^{g_0}_{e_{i,0}} \dot{n}_0,e_{i,0} )$: 
recalling \eqref{eq:0.1}, one has 
	\[
		\begin{aligned}
			g_0(\n^{g_0}_{e_{i,0}} \dot{n}_0, e_{i,0} ) 
			=& g_0\left( \n^{g_0}_{e_{i,0}} 
			\left\{
			\sum_{j=1}^2 (-h_{nj} e_{j,0}) - \frac{1}{2} 
			h_{nn} n_0 \right\}, e_{i,0} \right)
			\\
			=& 
			- \sum_{j=1}^2 
			\left[ e_{i,0} \big\{ g_0( h_{nj}e_{j,0},e_{i,0} ) 
			 \big\} - g_0( h_{nj}e_{j,0}, 
			 \n^{g_0}_{e_{i,0}} e_{i,0} )  \right] 
			 \\
			 & - \left[ 
			 \frac{1}{2} g_0 
			 \Big( e_{i,0}(h_{nn}) n_0, e_{i,0} \Big) 
			 + \frac{1}{2} h_{nn} 
			 g_0 ( \n^{g_0}_{e_{i,0}} n_0, e_{i,0} )\right].
		\end{aligned}
	\]
By $g_0(e_{j,0},e_{i,0}) = \delta_{ji}$ and 
$g_0(e_{i,0},n_0) = 0$,
there holds 
	\[
		\begin{aligned}
			g_0(\n^{g_0}_{e_{i,0}} \dot{n}_0, 
			e_{i,0} ) 
			&= - \sum_{j=1}^2 e_{i,0} (h_{nj} \delta_{ij}) 
			+ \sum_{j=1}^2 h_{nj} g_0(e_{j,0}, 
			\n^{g_0}_{e_{i,0}} e_{i,0} ) 
			- \frac{1}{2} h_{nn} 
			g_0 ( \n^{g_0}_{e_{i,0}}n_0,  e_{i,0})
			\\
			&=
			- e_{i,0} (h_{ni}) 
			+ \sum_{j=1}^2 h_{nj} g_0(e_{j,0}, 
			\n^{g_0}_{e_{i,0}} e_{i,0} ) 
			- \frac{1}{2} h_{nn} 
			g_0 ( \n^{g_0}_{e_{i,0}}n_0,  e_{i,0}). 
			\\
		\end{aligned}
	\]
Noting $H(0) = \sum_{k=1}^2 g_0 ( \n^{g_0}_{e_{k,0}} n_0, e_{k,0} )$, 
we see
	\begin{equation}\label{eq:0.8}
		\sum_{i=1}^2 
		g_0 ( \n^{g_0}_{e_{i,0}} \dot{n}_0, 
		e_{i,0} ) = 
		- \sum_{i=1}^2 
		e_{i,0} ( h_{ni}) + \sum_{i,j=1}^2 
		h_{nj} g_0(e_{j,0}, \n^{g_0}_{e_{i,0}} e_{i,0}) 
		- \frac{h_{nn}}{2}H(0). 
	\end{equation}

		Next, we calculate the terms
	\[
		\sum_{i=1}^2 g_0( \n^{g_0}_{\dot{e}_{i,0}} 
		n_0, e_{i,0} ), \quad \qquad 
		\sum_{i=1}^2 g_0( \n^{g_0}_{e_{i,0}} n_0, 
		\dot{e}_{i,0} ). 
	\]
For this purpose, we first compute $\dot{e}_{i,0}$. 
Differentiating 
	\[
		1 = g_t(e_{i,t},e_{i,t}) = g_0(e_{i,t},e_{i,t}) 
		+ t h ( e_{i,t}, e_{i,t}), \quad 
		0= g_t (e_{1,t}, e_{2,t}) 
		= g_0(e_{1,t},e_{2,t}) + t h (e_{1,t}, e_{2,t}),
	\]
one observes 
	\[
		0 = 2 g_0(\dot{e}_{i,0},e_{i,0}) 
		+ h(e_{i,0}, e_{i,0}), \quad 
		0 = g_0(\dot{e}_{1,0},e_{2,0}) 
		+ g_0(e_{1,0},\dot{e}_{2,0}) 
		+ h (e_{1,0},e_{2,0}),
	\]
which is expressed as 
	\[
		(\dot{e}_{i,0})_i = - \frac{1}{2} h_{ii}, 
		\quad \qquad (\dot{e}_{1,0})_2 
		+ (\dot{e}_{2,0})_1 = -h_{12}.
	\]
Set $\dot{e}_{1,0} =  \alpha_{11} e_{1,0} + \alpha_{12} e_{2,0}$ 
and $\dot{e}_{2,0} = \alpha_{21} e_{1,0} + \alpha_{22} e_{2,0}$: 
then we see that $\alpha_{ii}=-h_{ii}/2$ and  
	\[
		\begin{aligned}
			\sum_{i=1}^2 g_0( \n^{g_0}_{\dot{e}_{i,0}} 
			n_0, e_{i,0} ) 
			&= \sum_{i=1}^2 g_0 
			\left( \n^{g_0}_{\sum_{j=1}^2 \alpha_{ij} 
			e_{j,0}} n_0, e_{i,0} \right) 
			= \sum_{i,j=1}^2 
			\alpha_{ij} g_0( \n^{g_0}_{e_{j,0}} 
			n_0, e_{i,0} ) 
			= \sum_{i,j=1}^2 
			\alpha_{ij} A_{ji}.
		\end{aligned}
	\]
Noting that  $A_{ij}=A_{ji}$, we have 
	\[
		\alpha_{12} A_{21} + \alpha_{21} A_{12} 
		= A_{12} ( \alpha_{12} + \alpha_{21}) 
		= -A_{12} h_{12} = -\frac{1}{2} h_{12}A_{12} 
		-\frac{1}{2} h_{21}A_{21}.
	\]
Hence, it follows that 
	\begin{equation}\label{eq:0.9}
		\sum_{i=1}^2 
		g_0 ( \n^{g_0}_{\dot{e}_{i,0}} n_0, e_{i,0}) 
		= - \frac{1}{2} \sum_{i,j=1}^2 A_{ij}h_{ij}. 
	\end{equation}
Similarly we have 
	\begin{equation}\label{eq:0.10}
		\sum_{i=1}^2 g_0(\n^{g_0}_{e_{i,0}} n_0, 
		\dot{e}_{i,0}) 
		= \sum_{i,j=1}^2 g_0(\n^{g_0}_{e_{i,0}} 
		n_0, \alpha_{ij} e_{j,0} ) 
		= \sum_{i,j=1}^2 \alpha_{ij}A_{ij} 
		= -\frac{1}{2} \sum_{i,j=1}^2 A_{ij} h_{ij}.
	\end{equation}

By \eqref{eq:0.3}--\eqref{eq:0.10}, 
we obtain 
	\[
		\dot{H}(0) = 
		-\sum_{i=1}^2 e_{i,0}(h_{ni}) 
		+ \sum_{i,j=1}^2 
		h_{nj} g_0( \n^{g_0}_{e_{i,0}} e_{i,0}, e_{j,0}) 
		- \frac{1}{2} h_{nn} H(0) 
		+ \frac{1}{2} \sum_{i=1}^2 
		g_0 \left( \frac{\partial h}{\partial n_0} 
		e_{i,0}, e_{i,0} \right)
	\]
and Step 3 is completed.

\medskip

\noindent
{\bf Step 4:} {\sl Conclusion}

\medskip

Finally, we calculate $\dot{W}(0)$. 
Since there holds 
	\[
		\dot{W}(0) = 2 \int_{\S} H(0) \dot{H}(0) d \sigma_{0} 
		+  \int_{\S} H^{2}(0) d \dot{\sigma_{0}}, 
	\]
steps 1 and 3 yield 
	\[
	\begin{aligned}
	\dot{W}(0) 
	=& \frac{1}{2} 
	\int_{\Sigma} \Bigg[ 
	- 4 \sum_{i=1}^2 e_{i,0} (h_{ni}) H(0) 
	+ 4 \sum_{i,j=1}^2 h_{nj} \la \nabla^{g_0}_{e_{i,0}} , e_{j,0} \ra H(0) 
	- 2 h_{nn} H^2(0) 
	\\
	& \qquad \qquad 
		+ 2 \sum_{i=1}^2 
		\left\la \frac{\partial h}{\partial n_0} e_{i,0} , e_{i,0} \right\ra H(0) 
		+ H^2(0) 
	{\rm tr} \left( 
	h\big|_{T_p\Sigma} \right)
	\Bigg] \rd \sigma_0. 
	\end{aligned}
	\]
Thus we completed the proof. 
\end{proof}

\

\noindent 
Now we illustrate how we apply Lemma \ref{l:deri} 
to prove Proposition \ref{p:exp-W-to}. 
Set $t:=\e^{2}$. Recalling \eqref{eq:ex-g} and the definition of $g_\e$, we obtain  
	\begin{equation}\label{eq:ex-g-2}
		(g_{t})_{\alpha \beta}(x) 
		= \delta_{\alpha \beta} 
		+ \frac{t}{3} R_{\alpha\mu\nu \beta} x^{\mu} x^{\nu} 
		+ \tilde{R}(t,x), \quad 
		\left| \left( \frac{\partial}{\partial x} \right)^{\alpha} 
		\tilde{R}(t,x) \right|
		\leq C_{|\alpha|} t^{3/2}
	\end{equation}
for each $|x| \leq t^{-1/2} \rho_0$, $\alpha \in \Z_{+}^{3}$ and 
all sufficiently small $t>0$.

		Next, we set $\S = R\T$ and 
	\begin{equation}\label{eq:def-h}
		h_{\alpha \beta}(x) := 
		\frac{1}{3} R_{\alpha \mu \nu \beta} x^{\mu} x^{\nu}, \quad 
		(\tilde{g}_{t})_{\alpha \beta} (x) := \delta_{\alpha \beta} 
		+ t h_{\alpha \beta} (x). 
	\end{equation}
We also define 
	\[
		\tilde{W} (t) := \int_{R\T} \tilde{H}^{2}(t) d\sigma, \quad \qquad 
		W(t) := W_{g_t}( \exp_{P} ( R\T ) ), 
	\]
where $\tilde{H}(t)$ denotes the mean curvature of $R\T$ in $(\R^3,\tilde{g}_t)$. 
Then from \eqref{eq:ex-g-2} it follows that  
	\[
		\| g_{t} - \tilde{g}_{t} \|_{C^{k}(\overline{B_{3}(0)})} 
		\leq C_{k} t^{3/2} 
	\]
for all $k \in \N$. Hence, we may find  $C_0>0$, 
which is independent of $R \in SO(3)$, so that 
	\[
		| W (t) - \tilde{W} (t) | \leq C_{0} t^{3/2}. 
	\]
Noting that $t=\e^{2}$, to prove Proposition \ref{p:exp-W-to}, 
it suffices to show 
	\[
		\dot{\tilde{W}}(0) 
		= - 4 \sqrt{2} \pi^2 \left\{ \Sc_P - \Ric_P(R\mathbf{e}_z, R \mathbf{e}_z) \right\}. 
	\]
Furthermore, using the orthonormal basis $\{R \mathbf{e}_x, R \mathbf{e}_y, R \mathbf{e}_z\}$ 
of $\R^3$, we can reduce ourselves to the case $\S = \T$ (namely $R = {\rm Id}$) and 
it suffices to prove 
	\begin{equation}\label{eq:W1-dot}
		\dot{\tilde{W}}(0)  = - 4 \sqrt{2} \pi^{2} 
		\left\{ {\rm Sc}_P - \Ric_P(\mathbf{e}_z,\mathbf{e}_z) \right\}
		= -  4 \sqrt{2} \pi^{2} 
		\left\{ R_{11} + R_{22} \right\}. 
	\end{equation}

\
	
\noindent Now we will apply Lemma \ref{l:deri} 
to obtain \eqref{eq:W1-dot}. We use the map $X$ defined in \eqref{eq:def-X} for parametrizing $\T$, 
and as $e_1, e_2$ and $n$ we choose 
	\[
		\begin{aligned}
			e_1 &:= X_\varphi 
				= \left( - \sin \varphi \cos \theta,\ 
				- \sin \varphi \sin \theta, \ 
				\cos \varphi  \right),
			\\
			e_2 &:= X_\theta/ | X_\theta | 
				= \left(  -\sin \theta, \ \cos \theta, \ 
				0 \right),
			\\
			n &:= \left( \cos \varphi \cos \theta, \ 
				\cos \varphi \sin \theta, \ 
				\sin \varphi \right). 
		\end{aligned}
	\]

\begin{lem}\label{l:basic-pro}
The following hold:

\noindent
{\rm (I)} The volume element $d \sigma$ and 
the mean curvature of $\T$ with respect to the Euclidean metric are 
	\[
		d \sigma = \left( \sqrt{2} + \cos \varphi \right) 
		d \varphi d \theta \qquad 
		{\rm and} \qquad 
		H =  \frac{\sqrt{2} + 2 \cos \varphi}{\sqrt{2} + \cos \varphi}. 
	\]

\noindent
{\rm (II)} $ \la \n_{e_{1}} e_{1}, e_{1} \ra = \la \n_{e_{2}} e_{2}, e_{2} \ra 
= \la \n_{e_{1}} e_{1} , e_{2}\ra = 0$ and 
$ \la \n_{e_{2}} e_{2} , e_{1} \ra = \sin \varphi / ( \sqrt{2} + \cos \varphi) $.

\noindent
{\rm (III)} For the Ricci tensor $\Ric_{P}$, one has  
	\[
		\begin{aligned}
			\int_{0}^{2 \pi} 
			\Ric_{P}(X,n) d \theta 
			&= \pi \left[  
				{\rm Sc}_P 
				( \sqrt{2} + \cos \varphi ) \cos \varphi 
				+ R_{33} 
				\left\{ 2 \sin^2 \varphi - 
				( \sqrt{2} + \cos \varphi ) \cos \varphi
				 \right\} \right],
			\\
			\int_0^{2\pi} 
			\Ric_P(n,n) d \theta 
			&= 
				\pi \left[ {\rm Sc}_P \cos^2 \varphi 
				+ R_{33}( 2 - 3\cos^2 \varphi ) \right],
			\\
			\int_0^{2\pi} 
			\Ric_P (X,e_1) d \theta 
			&= \pi 
				\left[ -{\rm Sc}_P ( \sqrt{2} + \cos \varphi ) 
				\sin \varphi + R_{33} 
				\left\{  
				\sqrt{2} \sin \varphi  
				+ 3 \cos \varphi \sin \varphi 
				\right\}
				\right],
			\\
			\int_0^{2\pi} \Ric_P(n,e_1) d \theta 
			&= 
				\pi \left[ -{\rm Sc}_P \cos \varphi \sin \varphi 
				+ 3 R_{33} \cos \varphi \sin \varphi \right].
		\end{aligned}
	\]
Here we remark that $X,n,e_1$ are expressed in normal coordinates 
around $P \in M$: 
\[
X=X^\alpha \left( \frac{\partial}{\partial x^\alpha} \right)_P, \ 
n = n^\alpha \left( \frac{\partial}{\partial x^\alpha} \right)_P, \ 
e_1 = e_1^\alpha \left( \frac{\partial}{\partial x^\alpha} \right)_P \in T_PM, \quad 
g(P) \left( \left( \frac{\partial}{\partial x^\alpha} \right)_P, 
\left( \frac{\partial}{\partial x^\beta} \right)_P \right) = \delta_{\alpha \beta}. 
\]

\noindent
{\rm (IV)} There holds 
	\[
		\int_{0}^{2\pi} \frac{1}{\sqrt{2} + \cos \varphi} d \varphi 
		= 2 \pi, \quad 
		\int_{0}^{2\pi} \frac{\cos \varphi}{\sqrt{2} + \cos \varphi} 
		d \varphi = 2 \pi - 2 \sqrt{2} \pi, \quad 
		\int_{0}^{2\pi} \frac{\cos^{2} \varphi}{\sqrt{2} + \cos \varphi} 
		d \varphi = 4\pi - 2 \sqrt{2} \pi. 
	\]
\end{lem}

\begin{proof}
Assertions (I) and (II) are basic computations. For (III), note that 
	\begin{equation}\label{34}
		\Ric_{P}(X,n) = R_{\alpha \beta} X^{\alpha} n^{\beta}, 
		\quad \qquad {\rm Sc}_{P} = R_{11} + R_{22} + R_{33}. 
	\end{equation}
Since 
	\[
		0 = \int_{0}^{2\pi} \cos \theta d \theta 
		= \int_{0}^{2\pi} \sin \theta d \theta 
		= \int_{0}^{2\pi} \sin \theta \cos \theta d \theta, 
		\quad  
		\pi = \int_{0}^{2\pi} \cos^{2} \theta d \theta 
		= \int_{0}^{2\pi} \sin^{2} \theta d \theta,
	\]
it follows from the definitions of $X$,$n$ and \eqref{34} that 
	\[
		\begin{aligned}
			\int_{0}^{2\pi} \Ric_P (X,n) d \theta 
			&= R_{11} \pi 
				( \sqrt{2} + \cos \varphi ) \cos \varphi  
				+ R_{22} \pi 
				(\sqrt{2} + \cos \varphi ) \cos \varphi 
				+ 2 \pi R_{33} \sin^2 \varphi 
			\\
			&= 
				\pi \left[ 
				{\rm Sc}_{P} 
				( \sqrt{2} + \cos \varphi ) \cos \varphi 
				+ R_{33} 
				\left\{ 2 \sin^2 \varphi - 
				( \sqrt{2} + \cos \varphi ) \cos \varphi
				 \right\} 
				 \right].
		\end{aligned}
	\]
We can show in a similar way for other quantities, 
so we omit the details.

		Finally, for (IV), using the change of variables 
	\[
		t = \tan \left( \frac{\varphi}{2} \right), \quad 
		\cos \varphi = \frac{1-t^{2}}{1+t^{2}}, \quad 
		d \varphi = \frac{2}{1+t^{2}} dt
		\qquad 
		{\rm for}\ \varphi \in (0,\pi)
	\]
and 
	\[
		\cos \varphi 
		= (\sqrt{2} + \cos \varphi) - \sqrt{2},
		\qquad 
		\cos^{2} \varphi 
		= ( \cos \varphi + \sqrt{2} ) 
		(\cos \varphi - \sqrt{2} ) + 2, 
	\]
we easily get (IV). Thus we complete the proof. 
\end{proof}

\

We will now  prove Proposition \ref{p:exp-W-to}.

	\begin{proof}[Proof of Proposition \ref{p:exp-W-to}]
By the above arguments, 
it is enough to prove \eqref{eq:W1-dot}. 
First we express $h$ in \eqref{eq:def-h} 
in terms of ${\rm Sc}_{P}$ and $\Ric_{P}$. 
Substituting 
$g_{\alpha \beta} = \delta_{\alpha \beta}$ and 
\eqref{eq:RicciDec} into \eqref{eq:def-h}, we get 
	\begin{equation}\label{eq:ex-h}		
			h_{\alpha\beta} (x)
			= \frac{{\rm Sc}_P}{6} ( 
			|x|^2 \delta_{\alpha\beta} 
			- x_\alpha x_\beta) 
			- \frac{1}{3} \delta_{\alpha\beta} 
			\Ric_P (x,x) 
			- \frac{1}{3}|x|^2 R_{\alpha\beta} 
			+ \frac{1}{3} 
			( x_\alpha R_{\beta\mu}x^\mu + 
			x_\beta R_{\alpha\mu} x^\mu).
	\end{equation}
For the normal derivative of $h$, 
it follows from \eqref{eq:ex-h} that 
	\begin{equation}\label{eq:nor-deri-h}
		\begin{aligned}
			\frac{\partial h_{\alpha\beta}}{\partial n} (x) 
			=& \frac{d}{d t} h_{\alpha\beta} 
				(x + tn) |_{t=0}
			\\
			=& \frac{{\rm Sc}_{P}}{6} 
				\left( 2 \langle x, n \rangle 
				\delta_{\alpha\beta} - x_\alpha n_\beta 
				- n_\alpha x_\beta \right) 
				- \frac{2}{3}\delta_{\alpha\beta} 
				\Ric_{P} ( x, n) 
				- \frac{2}{3} \langle x, n \rangle 
				R_{\alpha\beta} 
			\\
			& + 
				\frac{1}{3} ( x_\alpha R_{\beta\mu} 
				n^\mu + n_\alpha R_{\beta\mu} x^\mu  
				+ x_\beta R_{\alpha\mu} n^\mu 
				+ n_\beta R_{\alpha\mu} x^\mu). 
		\end{aligned}
	\end{equation}

		In what follows, 
we compute each term in \eqref{eq:W-dot-0} 
using Lemma \ref{l:basic-pro}, \eqref{eq:ex-h} and 
\eqref{eq:nor-deri-h}.

\bigskip

\noindent
{\bf Step 1:} 
$\displaystyle 2 \sum_{i=1}^2 
\left\la \frac{\partial h}{\partial n} e_i,e_i \right\ra H$.

\bigskip

		Since it follows from \eqref{eq:nor-deri-h} that 
	\[
		\begin{aligned}
			{\rm tr}_{\R^3} 
			\left( \frac{\partial h}{\partial n} (X) \right)
			&= 
				\frac{{\rm Sc}_{P}}{6} 
				\left\{ 6  \la X,n \ra - 2 \la X,n \ra \right\} 
				- 2 \Ric_P ( X,n ) - \frac{2}{3} \la X, n \ra {\rm Sc}_{P} 
				+ \frac{4}{3} \Ric_{P} (X,n)
			\\
			&= - \frac{2}{3} \Ric_P(X,n),
			\\
			\left\la \frac{\partial h}{\partial n} n, n \right\ra
			&= 
				\frac{{\rm Sc}_{P}}{6}\left\{ 2 \la X,n \ra - 2 \la X,n \ra \right\} 
				- \frac{2}{3} \Ric_P (X,n) 
				- \frac{2}{3} \la X, n \ra \Ric_P (n,n)
			\\
			&\quad 
				+ \frac{2}{3} \{ \la X,n \ra \Ric_P(n,n) + \Ric_P(X,n) \}  = 0,
		 \end{aligned}
	\]
we obtain 
	\[
		\sum_{i=1}^2 
		\left\la \frac{\partial h}{\partial n}(X) e_i, e_i \right\ra
		= - \frac{2}{3} \Ric_P (X,n). 
	\]
Therefore, by Lemma \ref{l:basic-pro}, we have 
	\[
		\begin{aligned}
			&\int_{\T} 2 \sum_{i=1}^2 
			\left\la \frac{\partial h}{\partial n}(X) e_i, e_i \right\ra H d \sigma_0
			\\
			=& - \frac{4}{3} \int_0^{2\pi}\int_{0}^{2\pi} 
			\Ric_P(X,n) ( \sqrt{2} + 2 \cos \varphi ) 
				d \theta d \varphi
			\\
			=& 
				-\frac{4 \pi }{3} \int_0^{2\pi} 
				\left[ {\rm Sc}_{P} 
				( \sqrt{2} + \cos \varphi ) \cos \varphi 
				+ R_{33} 
				\left\{ 2 \sin^2 \varphi - 
				( \sqrt{2} + \cos \varphi ) \cos \varphi
				 \right\} \right] 
				 ( \sqrt{2} + 2 \cos \varphi ) d \varphi.
		\end{aligned}
	\]
Since 
	\[
		\begin{aligned}
			\int_0^{2\pi} (\sqrt{2} + \cos \varphi ) 
			( \sqrt{2} + 2 \cos \varphi ) \cos \varphi 
			d \varphi 
			&= \int_{0}^{2\pi} 
			( 2 + 3 \sqrt{2} \cos \varphi 
			+ 2 \cos^2 \varphi ) \cos \varphi d \varphi 
			= 3 \sqrt{2} \pi,
			\\
			\int_0^{2\pi} 
			2 \sin^2 \varphi ( \sqrt{2} + 2 \cos \varphi ) 
			d \varphi
			&= \int_0^{2 \pi} 2 \sqrt{2} \sin^2 \varphi 
			+ 4 \sin^2 \varphi \cos \varphi d \varphi 
			= 2 \sqrt{2} \pi,
		\end{aligned}
	\]
we get 
	\begin{equation}\label{eq:4}
		\begin{aligned}
			\int_{\T} 2 \sum_{i=1}^2 
			\left\la \frac{\partial h}{\partial n}(X) e_i, 
			e_i \right\ra H d \sigma_0
			&=
				 -\frac{4 \pi}{3} 
				\left\{3 \sqrt{2} \pi {\rm Sc}_{P} + R_{33} 
				( 2 \sqrt{2} \pi - 3 \sqrt{2} \pi ) 
				\right\}
			\\
			&= -4 \sqrt{2} \pi^2 {\rm Sc}_{P} 
			+ \frac{4\sqrt{2}}{3} \pi^2 R_{33}. 
		\end{aligned}
	\end{equation}

\bigskip

\noindent
{\bf Step 2:} 
$\displaystyle -4 \sum_{i=1}^2 e_i(h_{ni}) H 
+ 4 \sum_{i,j=1}^2 h_{nj} \la \n_{e_i}e_i,e_j \ra H$.

\bigskip

		For a smooth function $f \in C^\infty(\Sigma)$, 
we have 
	\[
		e_1(f) = \frac{\partial f}{\partial \varphi}, 
		\quad 
		e_2 (f) = \frac{1}{\sqrt{2} + \cos \varphi} 
		\frac{\partial f}{\partial \theta}. 
	\]
Thus  integrating by parts, there holds 
	\[
		\int_{\T} e_1 ( h_{n1}) H(0) \rd \sigma_0 
		= \int_{0}^{2\pi} d \theta 
		\int_0^{2\pi} 
		\frac{\partial h_{n1}}{\partial \varphi} 
		( \sqrt{2} + 2 \cos \varphi ) d \varphi 
		=  \int_0^{2\pi} d \theta 
		\int_0^{2\pi} h_{n1} 2 \sin \varphi 
		d \varphi.
	\]
On the other hand, we see that 
	\[
		\int_{\T} e_2(h_{n2})H(0) d \sigma_0 
		= \int_0^{2\pi} d \varphi 
		\int_0^{2\pi} \frac{1}{\sqrt{2}+\cos \varphi} 
		\frac{\partial h_{n2}}{\partial \theta} 
		( \sqrt{2} + 2 \cos \varphi ) d \theta 
		= 0.
	\]
Therefore, one has 
	\[
		\int_{\T} -4 \sum_{i=1}^2 e_i(h_{ni}) H 
		d \sigma_0 
		= -8 \int_0^{2\pi} \int_{0}^{2\pi} 
		h_{n1} \sin \varphi d \varphi d \theta. 
	\]
Also, it follows from Lemma \ref{l:basic-pro} that 
	\[
		\sum_{i,j=1}^2 
		h_{nj} \la \n_{e_i}e_i,e_j \ra H 
		= h_{n1} \la \n_{e_2}e_2,e_1 \ra H 
		= \frac{ \sin \varphi}{\sqrt{2} + \cos \varphi} h_{n1} H.
	\]
Then we have 
	\[
		\int_{\T} 4 \sum_{i,j=1}^2 
		h_{nj} \la \n_{e_i}e_i, e_j \ra H 
		d \sigma_0 
		=  \int_0^{2\pi} \int_0^{2\pi} 
		4 h_{n1} \sin \varphi 
		\frac{\sqrt{2} + 2 \cos \varphi}
		{\sqrt{2} + \cos \varphi} d \varphi d \theta
	\]
and 
	\begin{equation}\label{eq:5}
		\begin{aligned}
			\int_{\T} -4 \sum_{i=1}^2 e_i(h_{ni}) H 
			 +4 \sum_{i,j=1}^2	h_{nj} \la \n_{e_i} e_i, e_j \ra H 
			d \sigma_0 
			=& 
			 4 \int_0^{2\pi}\int_0^{2\pi} h_{n1} 
			\left( - 2 + 
			\frac{\sqrt{2} + 2 \cos \varphi}
			{\sqrt{2} + \cos \varphi}
			 \right) \sin \varphi d \theta d \varphi 
			\\
			=& 
			- 4 \int_0^{2\pi}\int_0^{2\pi} h_{n1} 
			\frac{\sqrt{2}}
			{\sqrt{2} + \cos \varphi } \sin \varphi 
			d \theta d \varphi .
		\end{aligned}
	\end{equation}

		Next, noting that 
	\[
		\begin{aligned}
			h_{n1} =& -\frac{{\rm Sc}_{P}}{6} 
				\langle X,n \rangle \langle X, e_1 \rangle 
				- \frac{1}{3}|X|^2 \Ric_P(n,e_1) 
				+ \frac{1}{3} 
				\{ \langle X,n \rangle \Ric_{P} (X,e_1) 
				+ \langle X, e_1 \rangle \Ric_{P} (X,n) \}
			\\
			=& \frac{{\rm Sc}_{P}}{6} 
				\sqrt{2} \sin \varphi 
				( \sqrt{2} \cos \varphi + 1 ) 
				- \frac{1}{3} ( 3 + 2 \sqrt{2} \cos 
				\varphi ) \Ric_P (n,e_1) 
			\\
			&+ \frac{1}{3} 
				\left\{ (\sqrt{2} \cos \varphi + 1) 
				\Ric_P(X,e_1) - \sqrt{2} \sin \varphi 
				\Ric_P (X,n)  \right\},
		\end{aligned}
	\]
one gets 
	\[
		\begin{aligned}
			\int_0^{2\pi} h_{n1} d \theta
			=& \frac{\sqrt{2}}{3} \pi {\rm Sc}_{P} 
				\sin \varphi ( \sqrt{2} \cos \varphi + 1) 
				- \frac{1}{3} ( 3 + 2 \sqrt{2} \cos \varphi ) 
				\pi ( - {\rm Sc}_P + 3 R_{33}) \cos \varphi 
				\sin \varphi 
			\\
			&+ \frac{1}{3} ( \sqrt{2} \cos \varphi + 1 ) 
				\pi \left[ -{\rm Sc}_{P} ( \sqrt{2} + \cos \varphi ) 
				\sin \varphi + R_{33} 
				\left\{ 
				3 \sin \varphi \cos \varphi + 
				\sqrt{2} \sin \varphi
				\right\} \right]
			\\
			& - \frac{\sqrt{2}}{3} \pi \sin \varphi 
				\left[ 	{\rm Sc}_{P} 
				( \sqrt{2} + \cos \varphi ) \cos \varphi 
				+ R_{33} 
				\left\{ 2 \sin^2 \varphi - 
				( \sqrt{2} + \cos \varphi ) \cos \varphi
				 \right\}  \right] 
			\\
			=& - \frac{R_{33}}{3} \pi \sin \varphi 
			( 2 \cos \varphi + \sqrt{2}). 
		\end{aligned}
	\]
Substituting this equation into \eqref{eq:5}, we obtain
	\[
		\begin{aligned}
		\int_{\T} -4 \sum_{i=1}^2 e_i(h_{ni}) H 
			+ 4 \sum_{i,j=1}^2 	h_{nj} \la \n_{e_i}e_i, e_j \ra H 
			d \sigma_0
		=& \frac{4 \pi}{3} R_{33} 
			\int_0^{2\pi} 
			\frac{(\sqrt{2} + 2 \cos \varphi) 
			\sqrt{2}}
			{\sqrt{2} + \cos \varphi} \sin^2 \varphi 
			d \varphi  
		\\
		=& \frac{4 \pi}{3} R_{33} 
		\int_{0}^{2\pi} \sqrt{2} 
		\left( 2 - \frac{\sqrt{2}}{\sqrt{2}+\cos\varphi} 
		\right)\sin^2 \varphi  d \varphi
		\\
		=&\frac{8\sqrt{2}\pi^{2}}{3} R_{33} 
		- \frac{8\pi}{3} R_{33} 
		\int_{0}^{2\pi} \frac{1-\cos^2\varphi}
		{\sqrt{2} + \cos \varphi} d \varphi .
		\end{aligned}
	\]
Finally by Lemma \ref{l:basic-pro} we obtain
	\begin{equation}\label{eq:6}
		\begin{aligned}
		\int_{\T} -4 \sum_{i=1}^2 e_i(h_{ni}) H 
			+ 4 \sum_{i,j=1}^2 
			h_{nj} \la \n_{e_i}e_i, e_j \ra H d \sigma_0 
			&= \frac{8 \sqrt{2}}{3} \pi^2 R_{33} 
			- \frac{8\pi}{3} R_{33} 
			( 2 \sqrt{2} \pi - 2 \pi ) 
		\\
		&=  \frac{8}{3} \pi^2 R_{33}
		\left( 2 - \sqrt{2} \right).
		\end{aligned}
	\end{equation}

	\bigskip

\noindent
{\bf Step 3:} 
$-2 h_{nn} H^2 + H^2 {\rm tr} (h_{|T_X\T})$.

	\bigskip

		By \eqref{eq:ex-h}, one observes 
	\begin{equation}\label{eq:7}
		-2 h_{nn} 
		= - \frac{{\rm Sc}_{P}}{3} 
		\{ |X|^2 - \la X,n \ra^2 \} 
		+ \frac{2}{3} \Ric_P(X,X) 
		+ \frac{2}{3} |X|^2 \Ric (n,n) 
		- \frac{4}{3} \la X,n \ra \Ric_P (X,n). 
	\end{equation}
Since $\{e_1,e_2\}$ forms an orthonormal basis of $T_X\T$ and 
	\[
		\la h(X) e_i, e_i \ra 
		= \frac{{\rm Sc}_{P}}{6} 
		\{ |X|^2 - \la X,e_i \ra^2 \}  
		- \frac{1}{3} \Ric_P (X,X) 
		- \frac{1}{3}|X|^2 \Ric_P(e_i,e_i) 
		+ \frac{2}{3} \la X,e_i \ra \Ric_P(X,e_i) 
		\quad i=1,2,
	\]
it follows from $\la X, e_2 \ra = 0$ that 
	\[
		\begin{aligned}
			&\mathrm{tr} \left( h_{| T_X \T} \right) 
			\\
			= &\frac{\Sc_P}{6} \left( 2 |X|^2 - \la X,e_1\ra^2 \right) 
			- \frac{2}{3} \Ric_P(X,X) - \frac{|X|^2}{3} 
			\left\{ \Ric_P(e_1,e_1) + \Ric_P(e_2,\e _2) \right\} 
			+ \frac{2}{3} \la X, e_1 \ra \Ric_P(X,e_1). 
		\end{aligned}
	\]
Since ${\rm Sc}_{P} = \Ric_P(e_1,e_1) + \Ric_P(e_2,e_2) + \Ric_P(n,n)$, we have 
	\begin{equation}\label{eq:8}
		{\rm tr}\left( h(X)_{|T_X\T} \right) 
		= - \frac{{\rm Sc}_{P}}{6} \la X,e_1 \ra^2 
		- \frac{2}{3} \Ric_P(X,X) + 
		\frac{1}{3}|X|^2 \Ric_P(n,n) 
		+ \frac{2}{3} \la X,e_1 \ra \Ric_P(X,e_1).
	\end{equation}
Combining \eqref{eq:7} and \eqref{eq:8} with 
$\la X,e_2 \ra = 0$ and $|X|^2 - \la X,n \ra^2 = \la X, e_1 \ra^2$, 
we get 
	\[
		\begin{aligned}
			& -2 h_{nn} 
			+ {\rm tr} \left( h(X)_{|T_X\T} \right) 
			\\
			&= 
				- \frac{{\rm Sc}_{P}}{2} \la X,e_1 \ra^2 + |X|^2 \Ric_P (n,n) 
				+ \frac{2}{3} \la X, e_1 \ra \Ric_P(X,e_1) 
				- \frac{4}{3} \la X,n \ra \Ric_P (X,n) 
			\\
			&= - {\rm Sc}_{P} \sin^2 \varphi 
				+ (3+2\sqrt{2}\cos \varphi ) 
				\Ric_P (n,n) 
				- \frac{2\sqrt{2}}{3} 
				\sin \varphi  \Ric_P(X,e_1) 
				-\frac{4}{3}(\sqrt{2} \cos \varphi + 1) 
				\Ric_P (X,n). 
		\end{aligned}
	\]
Hence, from Lemma \ref{l:basic-pro}, there holds 
	\[
		\begin{aligned}
			&\int_0^{2\pi} -2 h_{nn} 
			+ {\rm tr}\left( h(X)|_{T_X\T} \right) d \theta 
			\\
			=&
				- 2\pi {\rm Sc}_{P} \sin^2 \varphi 
				+ ( 3 + 2 \sqrt{2} \cos \varphi ) \pi 
				\left[ {\rm Sc}_{P} \cos^2 \varphi + R_{33} 
				( 2 - 3 \cos^2 \varphi ) \right] 
			\\ 
			& 
				-\frac{2 \sqrt{2}}{3}\pi  
				\sin \varphi \left[ - {\rm Sc}_{P} 
				( \sqrt{2} + \cos \varphi ) \sin \varphi 
				+ R_{33} \left\{  \sqrt{2} \sin \varphi 
				 + 3 \cos \varphi \sin \varphi 
				 \right\} \right]
			\\
			& - \frac{4}{3} \pi 
				( \sqrt{2} \cos \varphi + 1) 
				\left[ {\rm Sc}_{P}( \sqrt{2} + \cos \varphi )
				 \cos \varphi 
				+ R_{33} \left\{ 2 \sin^2 \varphi 
				- (\sqrt{2} + \cos \varphi ) \cos \varphi
				\right\} \right]
			\\
			=&
			-\frac{\pi}{3} {\rm Sc}_{P} 
			( \cos \varphi + \sqrt{2})^2 
			+ \frac{\pi}{3} R_{33} 
			\left\{ (\cos \varphi + \sqrt{2} ) 
			( -3 \cos \varphi + 5 \sqrt{2}) - 4 \right\}.
		\end{aligned}
	\]
From this equality it follows that 
	\[
		\begin{aligned}
			&\int_{\T} \left\{-2 h_{nn} 
			+ {\rm tr}\left( h_{|T_X\T} \right) 
			\right\} H^2 d \sigma_0
			\\
			=& 
				\int_{0}^{2\pi} 
				\left[ -\frac{\pi}{3} {\rm Sc}_{P} 
				(\cos \varphi + \sqrt{2}) 
				(\sqrt{2} + 2 \cos \varphi)^2 
				\right] d \varphi
			\\
			& + \int_{0}^{2\pi} 
				\frac{\pi}{3} R_{33} 
				\left[
				\left\{ 
				(-3 \cos \varphi + 5 \sqrt{2} ) 
				- \frac{4}{\cos\varphi + \sqrt{2}}
				\right\} 
				( \sqrt{2} + 2 \cos \varphi )^2 
				\right] d \varphi.
		\end{aligned}
	\]
Since 
	\[
		\begin{aligned}
			(\cos \varphi + \sqrt{2}) 
			( \sqrt{2} + 2 \cos \varphi)^2 
			&= 4 \cos^3 \varphi + 8 \sqrt{2} 
				\cos^2 \varphi + 10 \cos \varphi 
				+ 2 \sqrt{2},
			\\
			(-3 \cos \varphi + 5 \sqrt{2} ) 
			( \sqrt{2} + 2 \cos \varphi )^2 
			&= - 12 \cos^3 \varphi + 
			8 \sqrt{2} \cos^2 \varphi 
			+ 34 \cos \varphi 
			+ 10 \sqrt{2},
			\\
			(\sqrt{2} + 2 \cos \varphi)^2 
			&= 
			4 \cos \varphi ( \cos \varphi + \sqrt{2}) 
			+ 2,
		\end{aligned}
	\]
we obtain 
	\[
		\begin{aligned}
			\int_{0}^{2\pi} 
				\left[ -\frac{\pi}{3} {\rm Sc}_{P} 
				(\cos \varphi + \sqrt{2}) 
				(\sqrt{2} + 2 \cos \varphi)^2 
				\right] d \varphi
			=& 
				-\frac{\pi}{3}{\rm Sc}_{P} 
				( 8\sqrt{2} \pi + 4\sqrt{2} \pi); 
			\\
			\int_{0}^{2\pi} 
				\frac{\pi}{3} R_{33} 
				\left[
				\left\{ 
				(-3 \cos \varphi + 5 \sqrt{2} ) 
				- \frac{4}{\cos\varphi + \sqrt{2}}
				\right\} 
				( \sqrt{2} + 2 \cos \varphi )^2 
				\right] d \varphi
			=& \frac{\pi}{3}R_{33} 
			( 8 \sqrt{2} \pi + 20 \sqrt{2} \pi -16\pi ).
		\end{aligned}
	\]
Thus we have 
	\begin{equation}\label{eq:9}
		\int_{\T} \left\{-2 h_{nn} 
		+ {\rm tr}\left( h_{|T_X\T} \right) \right\} H^2
		d \sigma_0
		= - 4 \sqrt{2} \pi^2 {\rm Sc}_{P} 
		+ \frac{\pi^2}{3} (28 \sqrt{2} - 16 ) R_{33}.
	\end{equation}

\bigskip

\noindent
{\bf Step 4:} {\sl Conclusion}

\bigskip

By \eqref{eq:4}, \eqref{eq:6} and \eqref{eq:9}, 
we obtain 
	\[
		\dot{\tilde{W}}(0) 
		=  - 4 \sqrt{2} \pi^{2} ( {\rm Sc}_{P} - R_{33} ) 
		= - 4 \sqrt{2} \pi^{2} (R_{11} + R_{22}) 
	\]
which implies \eqref{eq:W1-dot}. 
Thus we complete the proof. 
\end{proof}

\

\begin{rem}\label{r:sectional}
Concerning the quantity ${\rm Sc}_{P} - \Ric_{P}(n,n)$, 
one can express it by the sectional curvature of the plane of symmetry 
of the Clifford torus  at $P \in M$. 
In fact, choose an orthonormal basis $\{ e_{1}, e_{2}, e_{3} \}$ 
of $T_{P}M$ with $e_{3} = n$. 
Denote by $K_{ij}$ the sectional curvature at $P \in M$  
for the section spanned by $e_{i}, e_{j}$. Then from the relations 
	\[
		R_{11} = K_{12} + K_{13}, \quad 
		R_{22} = K_{12} + K_{23}, \quad 
		R_{33} = K_{13} + K_{23}, \quad 
		{\rm Sc}_{P} = R_{11} + R_{22} + R_{33},
	\]
it follows that 
	\[
		{\rm Sc}_{P} - \Ric_{P}(n,n) = \frac{1}{2} {\rm Sc}_{P} + K_{12}.
	\]
\end{rem}



\subsection{Asymptotics of the Willmore functional on degenerating tori}

The goal of this subsection is to understand the asymptotic behaviour  of the Willmore functional
for the degenerating tori, i.e., $\Sigma_{\e,P,R,\o}$ when $|\o|$ is close to $1$. 
For this purpose, let $\Sigma_{r} \subset T_P M$ 
be a round sphere of radius $r>0$ in the Euclidean metric $(\delta_{\a \b})$ 
which passes through the origin:  we first compute the Willmore energy of 
$\Sigma_{r}$ with respect to the metric $g$. 
Recall that in normal coordinates, the metric $g$ is expressed as 
$g_{\a \b} := \delta_{\a \b} + h_{\a \b} (x)$ where $h$ satisfies 
	\be\label{eq:defh}
		|x|^{-2}|h(x)| + |x|^{-1}|\nabla h(x)| + |\nabla^2 h(x)| \leq  h_0; \qquad 
		\quad |x| \leq \rho
\ee
for some $\rho, h_{0}>0$, see also \eqref{eq:ex-g}. 
%
%
%
In the above setting, we prove

\begin{lem}\label{lem:Sphere}
Let $\Sigma_r \subset B_{\rho} ( 0) (\subset T_{P}M)$ 
be the round sphere of radius $r>0$
with respect to the Euclidean metric $(\delta_{\a \b})$ 
passing through the origin $0$. 
Then
\[
	W_g(\exp_{P}^g (\Sigma_{r})) = 16 \pi - \frac{8\pi}{3} {\rm Sc}_{P} r^{2} + O(r^{3})
\]
where ${\rm Sc}_{P}$ is the scalar curvature of $(M,g)$ at $P$. 
\end{lem}

\begin{proof}
Call $|\Sigma_r|_{g}$ and  $\mathring{A}_{r}$ the area and 
the traceless second fundamental form of $\Sigma_r$ in metric $g$. 
Then from \eqref{eq:defh}, one has
\be\label{eq:expAreaAcirc}
|\Sigma_r|_{g}= 4 \pi r^{2} + O(r^{4}) \quad  \text{and} \quad 
|\mathring{A}_r|^2=O(r^2).
\ee
In fact, the first claim clearly holds by \eqref{eq:defh}. 
The second claim also follows from \eqref{eq:defh} and 
the estimate (12) in \cite{MonSch} since 
the traceless second fundamental form of $\Sigma_r$ 
in the Euclidean space is null. 
Thus we obtain 
\be\label{eq:calU}
{\cal{U}}(\Sigma_r):=\int_{\Sigma_r} |\mathring{A}_r|^2 \, d \sigma_{g} 
= O(r^4)< r |\Sigma_r|_{g},
\ee
for $r$ small enough. 
Combining the area estimate in  \eqref{eq:expAreaAcirc}, 
the Conformal Willmore estimate \eqref{eq:calU} and 
the assumptions \eqref{eq:defh} on $h$, 
we are under the assumption of Theorem 5.1 in \cite{LM}, therefore 
$$ 
\left|W_g(\exp_{P}^g (\Sigma_{r}) ) - 16\pi 
+ \frac{2|\Sigma_r|_{g}}{3} {\rm Sc}_{P}\right| 
\leq  Cr|\Sigma_r|_{g} 
$$
for some constant $C$ depending just on $h_0$. 
Recalling \eqref{eq:expAreaAcirc}, the lemma follows.
\end{proof}

\

\noindent Since the degenerating tori converge smoothly locally 
(outside the origin which is by construction the point of concentration of the shrinking handles) 
to a round sphere passing through the origin with  radius $\sqrt[4]{2 \pi^2}$ 
by Proposition \ref{p:disk} (or Lemma \ref{lem:degTori}), 
it is natural to expect the Willmore energy 
possesses an expansion accordingly to Lemma \ref{lem:Sphere};  
more precisely  we have the following proposition.

\begin{pro}\label{p:expdegtorus}
There exists $C_0>0$, which is independent of $\e $, such that 
$$
	\limsup_{r\uparrow 1} \; \sup_{P \in M, R\in SO(3), |\o|=r} \;
	\left| \frac{1}{\e^{2}} \left(  
	W_{g_\e}(\Sigma_{\e,P,R,\omega}) - 8\pi^{2} + \frac{8 \sqrt{2}}{3} \pi^{2} \e^{2} 
	{\rm Sc}_{P} \right) \right| \leq C_0 \e
$$
for all sufficiently small $\e >0$. 
%
\end{pro}

\begin{proof}
First of all, we recall  the scaling invariance of the Willmore functional, which in our case implies: 
$W_g(\exp_{P}(\Sigma)) = W_{g_\e} ( \exp_{P}^{g_\e}( \e^{-1} \Sigma) )$. 
Hence, Lemma \ref{lem:Sphere} yields 
	\[
		W_{g_\e} ( \exp_{P}^{g_{\e}} 
		(R S_{\sqrt[4]{2 \pi^2}, \sqrt[4]{2\pi^2} \,\ov{\o}} ) )
		= 16 \pi - \frac{8 \sqrt{2}}{3} \pi^2 \Sc_{P} \e ^2 + O(\e^3)
	\]
for sufficiently small $\e $ and 
uniformly with respect to $P \in M$, $R \in SO(3)$ 
and $\ov{\omega} := \o / |\o | \in S^1$. 
Here $S_{r_1,r_2 \ov{\o}}$ denotes the sphere centred at $r_2 \ov{\o}$ 
with  radius $r_1$ for $r_1,r_2>0$ in the Euclidean space. 
Therefore, to prove the Proposition, it is enough to show that 
there exists $\e_0 > 0$ such that 
\be\label{eq:limsup}
\limsup_{r \uparrow 1} \;  \; \sup_{\e\in(0,\e_0], P \in M, R\in SO(3), |\o|=r} \;\frac{1}{\e^2} \left|  
		W_{g_\e}(\Sigma_{\e,P,R,\omega}) 
		- W_{g_\e} 
		( \exp_{P}^{g_\e} (R\, S_{\sqrt[4]{2\pi^2},\sqrt[4]{2 \pi^2}\,\ov{\o}}) ) 
		- 8\pi^{2} + 16 \pi \right| = 0.
\ee


We are going to estimate separately the errors in the handle part 
where the second fundamental form blows up as $|\o|\uparrow 1$ and 
in the complementary region, where we have smooth convergence to a round sphere. 
To this end, for $\rho > 0$, set 
	\[
		\begin{aligned}
			\Sigma_{\e ,P,R,\o ,\rho,1} 
			&:= \exp_{P}^{g_\e} \left( R \T_{\o} \backslash B_{\rho}(0) \right),
			&
			\Sigma_{\e ,P,R,\o ,\rho,2} 
			&:= \exp_{P}^{g_\e} \left( R \T_{\o} \cap B_{\rho}(0) \right),
			\\
			\Sigma_{\e ,P,R,\ov{\o},\rho,1} 
			&:= \exp_{P}^{g_\e} \left( R S_{\sqrt[4]{2 \pi^2}, \sqrt[4]{2 \pi^2} \ov{\o}} 
			\backslash B_{\rho}(0) \right),
			&
			\Sigma_{\e ,P,R,\ov{\o},\rho,2} 
			&:= \exp_{P}^{g_\e} \left( R S_{\sqrt[4]{2 \pi^2}, \sqrt[4]{2 \pi^2} \ov{\o}} 
			\cap B_{\rho}(0) \right). 
		\end{aligned}
	\]
Notice that  $\e =0$ corresponds to the Euclidean case, so  
the surface does not depend on $P \in M$.

We first treat the sphere part, namely $\Sigma_{\e,P,R,\o ,\rho,1}$. 
As in the previous subsection, let $t=\e ^2$, $g_t$ as in \eqref{eq:ex-g-2} and 
	\[
		W(t,\o ,\rho) := \int_{\Sigma_{\e,P,R,\o ,\rho,1}} 
		H^2_{g_t} d \sigma_{g_t}.
	\]
Thanks to Lemma \ref{lem:degTori}, we remark that 
for fixed $\rho>0$, 
$\Sigma_{\e,P,R,\o ,\rho,1}$ converges to $\Sigma_{\e,P,R,\ov{\o} ,\rho,1}$ 
smoothly and uniformly with respect to $(\e ,P,R)$ as $\o \to \ov{\o}$. 
Therefore, for a suitable sequence $\{\rho_r\}$ satisfying $\rho_r \to 0$ 
as $r \to 1$, we observe that for some $t_0>0$ 
	\[
		\begin{aligned}
			&\sup_{t \in (0,t_0), P \in M, R \in SO(3), |\o |=r} 
			\left| \frac{1}{t} 
			\left[ \left\{ W(t,\o ,\rho_r) - W(0,\o, \rho_r )  \right\} 
			- \left\{ W(t,\ov{\o} , \rho_r ) - W(0,\ov{\o}, \rho_r ) \right\}
			 \right] \right|
			 \\
			 =& \sup_{t \in (0,t_0), P \in M, R \in SO(3), |\o |=r} 
			 \left| \frac{1}{t} \int_0^t 
			 \left\{ \frac{\partial W}{\partial s}(s,\o ,\rho_r) - 
			 \frac{\partial W}{\partial s}(s,\ov{\o} ,\rho_r) \right\} d s \right|
			 \to 0
		\end{aligned}
	\]
as $r \to 1$. This implies that 
	\begin{equation}\label{eq:WTtoWSReg}
		\begin{aligned}
			\sup_{\e \in (0,\e_0), P \in M, R \in SO(3), |\o |=r} 
			\frac{1}{\e^2} 
			\Big|
			&\big\{ W_{g_\e} (\Sigma_{\e,P,R ,\o ,\rho_r,1}) - 
			W_{g_0}(\Sigma_{0,P,R,\o ,\rho_r,1})  \big\} 
			\\
			& - 
			\big\{ W_{g_\e} ( \Sigma_{\e ,P,R,\ov{\o}, \rho_r,1 } ) 
			- W_{g_0} ( \Sigma_{0,P,R,\ov{\o}, \rho_r,1 } ) 
			 \big\}\Big| \to 0
		\end{aligned}
	\end{equation}
as $r \to 1$. 

%

Now we estimate the difference between $W_{g_\e}$ and $W_{g_0}$ in 
the handle region. 
We first remark that the surface $\Sigma_{\e ,P,R,\o,\rho_r,2}$ 
is diffeomorphic to $R \T_{\o} \cap B_{\rho_r}(0)$, hence 
we regard all geometric quantities of $\Sigma_{\e ,P,R,\o ,\rho_r,2}$ 
as functions on $R \T_{\o} \cap B_{\rho_r}(0)$. 
By applying \cite[Lemma 2.3 with the choice $\gamma = \e^2 \rho_r^2$]{MonSch} 
and recalling \eqref{eq:ge=d+eh}-\eqref{eq:esti-he}, we get 
	\[
		\left| H_{g_\e}^2- H_{g_0}^2\right| 
		\leq C_0 \e^2 ( \rho_r^2  |A_{g_0}|^2 + 1 ) \quad 
		\text{on } R \T_{\o} \cap B_{\rho_r}(0), 
	\]
where $H_{g_0}$ is mean curvature of $R \T_{\o} \cap B_{\rho_r}(0)$ 
in metric $g_\e$ and 
$H_{g_0}$, $A_{g_0}$ are the mean curvature and the second fundamental form 
of $R \T_{\o} \cap B_{\rho_r}(0)$ in the Euclidean metric, and 
$C_0>0$ is a constant which does not depend on $\e ,P,R,\o $. 
Recalling also the following estimate of the ratio of the area forms 
$d\sigma_{g_\e}$ and $d\sigma_{g_0}$  in metric $g_{\e}$ and Euclidean metric 
(see \cite[Lemma 2.2]{MonSch})
	\[
		(1-4 C_0 \rho_r^2 \e^2) d\sigma_{g_0} 
		\leq d\sigma_{g_\e} 
		\leq (1+4 C_0 \rho_r^2 \e^2) d \sigma_{g_0} 
		\quad  \text{on } R \T_{\o} \cap B_{\rho_r}(0),
	\]
it follows that
	\[
		\begin{aligned}
			\frac{1}{\e^2}\left|W_{g_\e}(\Sigma_{\e ,P,R,\o ,\rho_r,2}) 
			- W_{g_0}(\Sigma_{0 ,P,R,\o ,\rho_r,2})\right| 
			& \leq 
			C_1 \int_{R\T_\o \cap B_{\rho_r}(0)} 
			\left( \rho_r^2 |A_{g_0}|^2 + 1 \right) d \sigma_{g_0} 
			\\
			& \leq C_1 \rho_r^2 \int_{R\T_\o} |A_{g_0}|^2 d \sigma_{g_0} 
			+ C_1 \left| R \T_\o \cap B_{\rho_r}(0) \right|_{g_0}. 
		\end{aligned}
	\]
Noting the conformal invariance of $\int_{\Sigma} |A_{g_0}|^2 d \sigma_{g_0}$ 
in the Euclidean space, $\rho_r \to 0$ and 
$| \T_{r \mathbf{e}_x} \cap B_{\rho_r}(0) |_{g_0} \to 0$ 
due to Lemma \ref{lem:degTori}, we obtain 
	\begin{equation}\label{eq:WhWeHand}
		\begin{aligned}
			&\sup_{\e \in (0,\e_0), P \in M, R \in SO(3), |\o |=r}  
			 \frac{1}{\e^2}\left|W_{g_\e}(\Sigma_{\e ,P,R,\o ,\rho_r,2}) 
			- W_{g_0}(\Sigma_{0 ,P,R,\o ,\rho_r,2})\right|
			\\
			\leq & C_1 \rho_r^2 \int_{\T} |A_{g_0}|^2 d \sigma_{g_0} 
			+ C_1 \left| \T_{r \mathbf{e}_x}  \cap B_{\rho_r}(0) \right|_{g_0} \to 0
			\qquad \text{as } r \to 1. 
		\end{aligned}
	\end{equation}
In an analogous way, but it is actually easier 
since the second fundamental form is bounded, one can check that
\be \label{eq:WhWeSpBr}
\sup_{\e \in (0,\e_0), P \in M, R \in SO(3), \ov{\o} \in S^2} 
\frac{1}{\e^2}\left|W_{g_\e}( \Sigma_{\e ,P,R,\ov{\o}, \rho_r,2 } ) 
- W_{g_0} ( \Sigma_{0,P,R,\ov{\o}, \rho_r,2} )  \right| 
\to 0 \quad \text{ as } r\to 1.
\ee

Finally, since the conformal invariance of the Willmore functional in 
the Euclidean space gives 
	\[
		16 \pi = 
		W_{g_0} (\Sigma_{0,P,R,\ov{\o},\rho_r,1}) 
		+ W_{g_0} (\Sigma_{0,P,R,\ov{\o},\rho_r,2}), 
		\qquad 
		8 \pi^2 = 
		W_{g_0} (\Sigma_{0,P,R,\o,\rho_r,1}) 
		+ W_{g_0} (\Sigma_{0,P,R,\o,\rho_r,2}),
	\]
combining \eqref{eq:WTtoWSReg}, \eqref{eq:WhWeHand} and \eqref{eq:WhWeSpBr}, 
the claim \eqref{eq:limsup} follows and we complete the proof. 
\end{proof}



\section{Proof of the main theorems}

In this short section, collecting the above results, 
we prove Theorems \ref{t:1} and \ref{t:2} starting from the former.

\subsection{Proof of Theorem \ref{t:1}} \label{Sec:ProofT1}
First of all, from \eqref{eq:tke}, we notice that 
	$$
		\partial \mathcal{T}_{K , \e} 
		= \left\{ \exp_P^{g_\e} ( R \T_\o  ) \; : \; P \in M, R \in SO(3),  
		\o \in \partial K \right\}. 
$$
Combining Propositions \ref{p:variational} and \ref{p:variational2}, 
we know that, for every compact subset $K\subset \subset \DD$ to be fixed later, 
there exists  $\bar{\e}_K>0$ such that for every $\e\in (0,\bar{\e}_K]$ 
there is a function $\Phi_\e: \mathcal{T}_{K , \e} \to \R$ with the following properties:

\

\noindent (i) If $(P_\e,R_\e,\o_\e)$ is a critical point of $\Phi_\e$, 
then the perturbed torus $\Sigma_{\e,P_\e, R_\e, \o_\e}[\var_\e(P_\e,R_\e,\o_\e)]$ 
is a Willmore surface with constrained area $4\sqrt{2} \pi^2 \e^2$; 
recall the definition \eqref{eq:defSw} of $\Sigma_{\e,P, R, \o}[\var]$.

\

\noindent (ii) There exists a constant $C_K>0$ depending just on $K$ such that, 
for every $(P,R,\o)\in M\times SO(3)\times K$, it holds
\[
 \left|\Phi_\e(P,R,\o)-W_{g_\e}(\Sigma_{\e,P,R,\o}) \right|\leq C_K \e^4.
\]

\

In order to choose the compact subset $K\subset \subset \DD$ suitably, 
let us recall the expansion of the Willmore functional on symmetric tori and 
on degenerate tori 
(given in Propositions \ref{p:exp-W-to} and \ref{p:expdegtorus} respectively). 

\

\noindent (iii) For every $P\in M$ and $\e$ small enough it holds 
\[
W_{g_\e}(\Sigma_{\e,P,R,0}) 
= 8\pi^2 -4\sqrt{2} \pi^2 \{{\rm Sc}_P-\Ric_P(R n,R n)\} \e^2 + O(\e^3), 
\]
where $n \in T_PM$ denotes the normal axial vector of $\T$. 

\

\noindent (iv) There exist $\e_{0} > 0$ and $C_0>0$ such that
\[
\limsup_{r\uparrow 1} \; \sup_{P \in M, R\in SO(3), |\o|=r} \;
		\frac{1}{\e ^2}\left| 
		W_g(\Sigma_{\e,P,R,\omega}) - 8\pi^{2} + \frac{8 \sqrt{2}}{3} \pi^{2} \e^{2} 
		{\rm Sc}_{P}  \right| \leq C_0 \e
\]
for all $\e \in (0,\e_0]$. 
If we now assume (as in the hypotheses of Theorem \ref{t:1}) that  
\be\nonumber \label{eq:assSc}
3 \sup_{P \in M} \left( {\rm Sc}_{P} - \inf_{| \nu |_{g}=1} {\rm Ric}_{P} ( \nu, \nu) \right) > 2 \sup_{P \in M} {\rm Sc}_{P},
\ee
 (or, respectively $3 \inf_{P \in M} \left( {\rm Sc}_{P} - \sup_{| \nu |_{g}=1} 
		{\rm Ric}_{P} ( \nu, \nu) \right) 
		< 2 \inf_{P \in M} {\rm Sc}_{P}) $
then, combining  (ii), (iii), (iv) above there exists $r\in (0,1), \e_1>0$ such that, chosen as  compact subset $K:=B_r(0)\subset \subset \DD$,  the corresponding reduced functional $\Phi_\e$ satisfies
\be\nonumber \label{eq:PhiInterior}
\inf_{\mathcal{T}_{K , \e}} \Phi_\e < \inf_{\partial \mathcal{T}_{K , \e}}  \Phi_\e \qquad \quad 
\left( \hbox{or, respectively,  } \sup_{\mathcal{T}_{K , \e}} \Phi_\e  > \sup_{\partial \mathcal{T}_{K , \e}}  \Phi_\e
\right), 
\ee  		
for every $\e \in (0,\e_1]$.

It follows that the global minimum of $\Phi_\e$ (resp. the global maximum) is achieved at an interior point 
of $\mathcal{T}_{K , \e}$, which therefore  is a  critical point of $\Phi_\e$.  By recalling (i) above we conclude that the perturbed torus $\Sigma_{\e,P_\e, R_\e, \o_\e}[\var_\e(P_\e,R_\e,\o_\e)]$ is a  Willmore surface with constrained area $4\sqrt{2} 
\pi^2 \e^2$  and finally, by construction, the graph function $\var_\e$ converges to $0$ in $C^{4,\a}$-norm as $\e\to 0$
with decay rate $\e^2$ in the rescaled metric $g_\e $. 
This concludes the proof. 
\hfill$\Box$

\begin{rem}[Proof of Remark \ref {rem:2tori}]\label{Rem:Pf2Tori}
Notice that if both the conditions \eqref{eq:Assump1} and \eqref{eq:Assump2} are satisfied then, by repeating verbatim the proof above,  the reduced functional $\Phi_\e$ has both global minimum and global maximum attained at interior points of $\mathcal{T}_{K , \e}$.  It follows that   exist at least \emph{two} smooth embedded Willmore tori in $(M,g)$ with constrained 
area equal to $4 \sqrt{2} \pi^2 \e^2$.
\end{rem}

\subsection{Proof of Theorem \ref{t:2}} \label{Sec:ProofT2}
The only additional difficulty in proving Theorem \ref{t:2} is that we have to find both a compact subset of the unit disk $\DD$ and a compact subset of $\R^3\setminus\{0\}$ where to perform the finite dimensional reduction. Let us discuss the latter, the former  being analogous to the compact case.

As recalled  in the introduction, 
the Schwarzschild metric $g_{ij}$ on $\R^3\setminus\{0\}$ 
is conformal to the Euclidean one, 
its scalar curvature vanishes identically and it has two asymptotically flat ends 
(one for $r\downarrow 0^+$ and one for $r \uparrow +\infty$). 
For $\tau > 0$, set 
\[
A_\tau := \{ P \in \R^3 \ : \ \tau \leq |P|_{g_0} \leq \tau^{-1} \}, 
\quad 
\eta := \max_{\R^3 \setminus \{0\}} \max_{|\nu|_{g_{Sch}} =1 } 
\left|\Ric_P (\nu,\nu) \right| > 0. 
\]
Remark that the positivity of $\eta$ easily follows from 
either direct computations or observing that 
$\Ric \equiv 0$ implies that $(\R^3  \setminus \{0\}, g_{Sch} )$ is isometric 
to the Euclidean space. 
In addition, since $(\R^3 \setminus \{0\}, g_{Sch})$ has two flat ends, 
there exists a $\tau_0>0$ such that 
$\eta$ is achieved at interior points $P$ 
of $A_{\tau}$ provided $\tau \in (0,\tau_0)$. 
Furthermore, noting that 
$W_{g_0}(R\T_\omega) = W_{g_0}(\T) = 8 \pi^2$ by 
conformal invariance of $W_{g_0}$, 
for every compact set $K \subset \mathbb{D}$, one can find $C_K>0$ such that 
if either $|P|_{g_0} \leq \tau$ or  $\tau^{-1} \leq |P|_{g_0}$, then 
\begin{equation}\label{eq:diff-W0-8}
\begin{aligned}
\sup_{P \in A_{\tau}^{c}, R \in SO(3), \omega \in K} 
\left| W_{g_\e} (\exp_{P}^{g_\e} (R \T_\omega ) ) - 8\pi^2 \right| 
&= \sup_{P \in A_{\tau}^c, R \in SO(3), \omega \in K} 
\left| W_{g_\e} (\exp_{P}^{g_\e} (R \T_\omega ) ) 
- W_{g_0} ( R \T_\omega )  \right| 
\\
&\leq C_K o_{\tau}(1) \e^2
\end{aligned}
\end{equation}
where $o_{\tau}(1) \to 0$ as $\tau \to 0$. 
As in Proposition 4.6, by $\Sc \equiv 0$ on $\R^3 \setminus\{0\}$, 
we may also prove that 
	\[
		\limsup_{r \to 1} \sup_{P \in \R^3 \setminus \{0\}, R \in SO(3), |\omega|=r  } 
		\left| \frac{1}{\e^2} 
		\left( W_{g_\e} ( \Sigma_{\e ,P,R,\omega} [0] ) 
		- 8 \pi^2 \right) \right| 
		\leq C_0 \e .
	\]
Hence, we may find $r_\eta\in (0,1)$ and $\e _\eta >0$ such that 
if $\tau \in (0,\tau_0)$ and $\e \in (0,\e_\eta)$, then 
\[
\sup_{P \in  A_{\tau}, R \in SO(3), |\omega|=r_\eta} 
\left| W_{g_\e} (\exp_{P}^{g_\e} (R \T_\omega ) ) - 8\pi^2 \right| 
\leq \frac{\eta}{4} \e^2
\]
Setting $K := \overline{B_{r_\eta}(0)}$, by \eqref{eq:diff-W0-8}, 
there exists a $\tau_{1,\eta} \in (0,\tau_0)$ satisfying  
\[
\sup_{P \in \partial A_{\tau_{1,\eta}}, R \in SO(3), \omega \in K} 
\left| W_{g_\e} (\exp_{P}^{g_\e} (R \T_\omega ) ) - 8\pi^2 \right| 
\leq \frac{\eta}{2} \e^2
\]
for all $\e \in (0,\e_\eta)$.

Now one can repeat  the proof of Theorem \ref{t:1} 
by replacing the ambient space $M$ with $A_{\tau_{1,\eta}}$ above, 
getting a compact subset $K := \overline{B_{r_\eta}(0)} \subset \subset \DD$ 
and an $\e_0>0$ such that 
the corresponding reduced functional $\Phi_\e$ has an interior critical point. 
Therefore there exists a perturbed torus 
$\Sigma_{\e,P_\e, R_\e, \o_\e}[\var_\e(P_\e,R_\e,\o_\e)]$ 
which is a  Willmore surface with constrained area $4\sqrt{2} 
\pi^2 \e^2$  whose  graph function $\var_\e$ converges to $0$ in $C^{4,\a}$-norm as $\e\to 0$. We conclude by recalling that the area constrained-Willmore surfaces coincide with the critical points of the Hawking mass under area constraint and that the rotational invariance of the Scwarzschild metric implies that all the surfaces obtained by rotating $\Sigma_{\e,P_\e, R_\e, \o_\e}[\var_\e(P_\e,R_\e,\o_\e)]$ around the origin $0 \in \R^3$ are critical points as well.
\hfill$\Box$

\begin{rem}\label{rem:proofRem}
The generalization to ALE scalar flat 3-manifods claimed in Remark \ref{rem:ALE} follows on the same line of the proof above once we replace the annulus $A_{\tau}$ by a large ball $B_{\tau}(x_0)$, for some fixed $x_0\in M$. Indeed,  assumptions 1)-2) imply that 
\begin{equation}\nonumber
\sup_{P \in B_{\tau}(x_0)^{c}, R \in SO(3), \omega \in K} 
\left| W_{g_\e} (\exp_{P}^{g_\e} (R \T_\omega ) ) - 8\pi^2 \right| 
\leq C_K o_{\tau}(1) \e^2
\end{equation}
where $o_{\tau}(1) \to 0$ as $\tau \to 0$;  the rest of the proof can be repeated verbatim.
\end{rem}

\end{document}